\documentclass{jgcc} %%% last changed 2020-01-17

\usepackage{hyperref}
\usepackage{bigints}
\usetikzlibrary{arrows,automata,positioning}
\usepackage{graphicx}
\usepackage{subcaption}
\usepackage[export]{adjustbox}
\usepackage{wrapfig}
\usepackage{mathptmx}
\usepackage[11pt]{moresize} 
\theoremstyle{plain} %\crefname{satz}{Satz}{S\"atze}
\def\eg{{\em e.g.}}

\begin{document}

\title[Instructions]{A Contracting Fractal Group,  Schreier Graphs and The Spectra}
\titlecomment{{\lsuper*}OPTIONAL comment concerning the title, \eg,
  if a variant or an extended abstract of the paper has appeared elsewhere.}

\author[B.Vaziri]{Bozorgmehr Vaziri}	%required
\address{Amirkabir University of Technology, Hafez-Ave, Tehran, Iran}	%required
\email{bozorgmehrvaziri@gmail.com}  %optional
%\thanks{thanks 1, optional.}	%optional

\author[F.Rahmati]{Frahad Rahmati}	%optional
\address{Amirkabir University of Technology, Hafez-Ave, Tehran, Iran}	%optional
\email{frahmati@aut.ac.ir}  %optional

%% etc.

%% required for running head on odd and even pages, use suitable
%% abbreviations in case of long titles and many authors:

%%%%%%%%%%%%%%%%%%%%%%%%%%%%%%%%%%%%%%%%%%%%%%%%%%%%%%%%%%%%%%%%%%%%%%%%%%%

%% the abstract has to PRECEDE the command \maketitle:
%% be sure not to issue the \maketitle command twice!

\begin{abstract}
This paper explores a previously uncharted automaton group generated by a 5-state automaton $(\Pi, A)$ acts by self-similarity on the regular rooted tree $A^{*}$ over a 2-letter alphabet set $A$. Group $G$ has been subjected to several observations that reveal the self-similar $G$-action possesses several notable characteristics: it is a weak branch, contracting, fractal group, and satisfies the open set condition with an exponential growth rate in activity. The corresponding Schreier graphs $\Gamma$ of the group action on the binary rooted tree, for $n\leq 12$, are presented. Furthermore, we delve into a numerical approximation of the spectrum of the discrete Laplacian operator, which is defined on the limit Schreier graph $\hat{\Gamma}=\lim_{n\rightarrow\infty} \Gamma_{n}$.
\end{abstract}

\maketitle

%% start the paper here:
\section{Introduction}\label{sec1}

After the first appearance of automaton groups by Aleshin \cite{aleshin1972finite} in 1960, the idea of self-similarity and fractalness entered geometry group theory and C-algebra through the automaton group, which provided theory of self-similar group algebra, fractal groups, and limit dynamical system $(\mathcal{J}_{G},s)$. An automaton group $G\leq Aut(A^{*})$ inherently is a subgroup of a $\vert A \vert$-regular rooted tree's automorphism, which generated by states of an automaton $\Pi=(Q, A)$ in a recursive manner and acts self-similarly on tree $A^{*}$. The automaton group received much attention after the innovation by R.I.Grigorchuk \cite{grigorchuk1980burnside} in 1980, who introduced a group generated by a finite-state automaton. Particullary, automaton groups has been used to investigating Burnside problem on periodic groups, and the word problem on finitely generated groups. The category of self-similar groups encompasses large classes of groups, like the class of infinite periodic finitely generated groups, groups of intermediate growth, just-infinite groups, and branch groups. V.Nekrashevich and L.Bartholdi \cite{bartholdi2008iterated} reveals a deep relation between the algebraic properties of the automaton group, the asymptotic geometry of the group, and the holomorphic dynamical systems. They introduced the iterated monodromy group $IMG(f)$, which can be associated with the dynamical system $(\mathcal{J},f)$ consisting of partial self-covering rational map 
$f:S\rightarrow S$ of Riemannian surface $S$ and its Julia set $\mathcal{J}=\mathcal{J}(f)$. The iterated monodromy group admits a self-similar action on the rooted tree $\mathcal{T}=\bigcup_{n>0} f^{-n}(x_{0})$, provided by the preimages of covering map $f$ and the arbitrary base point $x_{0}\in S$.
Moreover, Neckrashevich proved that when the rational map $f$ is expanding, $IMG(f)$ is a contracting self-similar group (for more  details, see \cite{nekrashevych2005self,nekrashevych2011iterated}). The author \cite{nekrashevych2009c,nekrashevych2004cuntz} involved the theory of automaton groups  $C^{*}$-algebra of bounded operators and harmonic analysis on fractals (in the latest direction, see \cite{bartholdi1999spectrum,nekrashevych2008groups} and \cite{bartholdi2000spectra}).

In the last few years, Authors \cite{grigorchuk2018spectra,grigorchuk2017combinatorics} tackled the spectral problem of the Laplacian operator on Schreier graphs of automaton groups by attaching an aperiodic ordered family of discrete Schrodinger operators $\lbrace H_{w} \rbrace_{w\in\Omega}$ to the dynamical system $(\Omega, \sigma)$ associated with the self-similar action of the first Grigorchuk group on boundary of binary rooted tree, i.e., $A^{-\mathbb{N}}$, where $\Omega$ is linearly repetitive subshift and $\sigma$ the shift operator, which can have significant implications in theory of quantum chaos on quasi-crystal mediums (see \cite{gnutzmann2006quantum}).

The spectral theory of self-similar groups $G$  has been significantly influenced by seminal works of R.Grigorchuk, L.Bartholdi, and D.Zuk \cite{grigorchuk2006self,bartholdi2000spectra,bartholdi1999spectrum}. They have used the inherent self-similarity of the action provided an operator recursion, known as Schur complement, that facilitated the spectrum computation of non-commutative dynamical systems. Moreover, \cite{bartholdi1999parabolic} investigated Heck algebra and Heck-type operators in non-commutative dynamical systems. In this direction, one of the crucial implications that have been proved by Grigorchuk and Zuk \cite{grigorchuk1999asymptotic} is the continuity of the spectral density function on space of locally finite marked graphs accompanied by Hausdorff-Gromov
metric becomes a metric space.
More persicely; for a sequence of marked graphs $\lbrace (\Gamma_{n} , x_{n} )\rbrace$, they showed
the weak convergence of spectral measures associated with Markov operator to a spectral measure on the limit marked graph $(\hat{\Gamma}, \hat{x})$ on the metric space of regular marked graphs .
This provide a straightforward method for approximating the Laplacian's spectrum on the limit space $\mathcal{J}_{G}$ associated with the self-similar group $(G, A)$ in sence of V.Nekrashevich \cite{nekrashevych2005self}. 

According to the results provided by authors \cite{grigorchuk1999asymptotic}, some numerical approximation of the operator's spectrum $\Delta_{\rho}=D - M_{\rho}$ is given in the present research, which is known as discrete Laplacian operator (Markov operator $M_{\rho}$) on the limit Schreier graph $\hat{\Gamma}\sim\mathcal{J}_{G}$
associated with the automaton group $G$ generated by the automaton \ref{Moore diagram} with the generating set $S=\lbrace a, b, c, d\rbrace$;
\begin{center}
$\Delta_{\rho}:= D-\sum_{s\in S}\rho (s + s^{-1}) $.
\end{center}
Where $\rho(s + s^{-1})$  is a bounded self-adjoint operator for all $s\in S$,
and $\rho :\mathbb{C}G\rightarrow \mathcal{B}( L^{2}(A^{-\mathbb{N}}, \nu)))$ is the lift of the left-regular representation of the group $G$ to group algebra $\mathbb{C}G$, and $D$ is multiplication by $deg(\Gamma)$ operator, i.e.,  $D f =deg(\hat{\Gamma}) f$.

Nevertheless, computing the Laplacians spectrum corresponding to automaton groups is still challenging in a general case; sometimes, the condition of the Schur complement method needs to be provided, as in our case (see \ref{The spectrum of Laplacian}). There is no general algorithmic method to compute the spectrum of Laplacian operators, and only a few known automata group spectra are calculated, which is reviewed in \ref{The spectrum of Laplacian}.

Let $G$ be an automaton group based on automaton $\Pi (S , A)$ over finite alphabet set $A$; $G=\langle S \rangle$. The algebraic structures (relations, lattice of normal subgroups ,etc)  of the group identifys by its self-similar action on the $\vert A \vert = q$-reguar (in general case, spherically homogeneous) rooted tree $\mathcal{T}= A^{*}$.
 Authors \cite{nekrashevych2005self}\cite{bartholdi2006automata} showed any automaton $\Pi (S , A, \nu, \Gamma)$ establishes $G$-invariant asymptotic relation $\sim_{\Pi} $ on the $A^{-\omega}\times G$, and the quotient of $A^{-\mathbb{N}}\times G$ modulus of the asymptotic-relation $\sim_{\Pi} $ denoted by $\mathcal{J}_{\Pi}:=  A^{-\mathbb{N}}\times G / \sim_{\Pi} $ is called the limit space associated with self-similar structure $(G, A)$. The limit space $\mathcal{J}_{\Pi}$ is the set of $\sim_{\Pi} $- equivalence classes of the left (right)-infinite words $\cdots x_{2} x_{1}$ with letters $x_{i} \in A$ which are in bijection with left-infinit paths in Moore graph of the automaton'nucleus.
 Therefore the asymptotic dynamics of the group (for example, the group's growth, linear activity of the group action, Hausdorff dimension) are reflected by the Moore graph of the nucleus.
 Nekrashevich (\cite{nekrashevych2005self} lemma 3.1.2) showed that when $G$ is a contracting automaton group, i.e., $G$ has a finite nucleus $\mathcal{N}$ that generates the group ($G=\langle\mathcal{N}\rangle$), the asymptotic equivalence is a closed $G$-invariant relation; therefore, the quotient $\mathcal{J}_{G}$ is a compact Hausdorff space is equipped with the right-shift operator $\sigma (\cdots x_{2} x_{1}) =\cdots x_{3} x_{2}$ , provides a dynamical system $(\mathcal{J},\sigma)$ which is the content of the modern theory of dynamical systems and fractal geometry.
\\
The present research is concentrated on group $G$, which is generated by states of 5-state automaton $\Pi$ over a 2-letter alphabet $A=\lbrace 0,1 \rbrace$; its Moore graph is shown in Figure \ref{Moore diagram}.\\

The following wreath recursion defines the automaton:
\begin{center}
$a=\sigma (d ,1) , b=(a,c) , c=(a,a) , d=(1,b) $\label{wreath recurssion} , $\Pi:=(\lbrace 1, a,b,c,d \rbrace , \lbrace 0, 1 \rbrace)$ \\
$G:=\langle a^{\pm 1}, b^{\pm 1}, c^{\pm 1}, d^{\pm 1} \rangle$
\end{center}
In sequle the primery automaton group $G$ has been subjected to several observations:
\begin{itemize}
\item[1- ]The self-similar action of group $G$ on the binary tree $A^{*}$ and its boundary $A^{-\mathbb{N}}$ is investigated through theorem \ref{main theorem}.
\item[2- ] 
The associated Schreier graphs $\Gamma_{n}(G, S , A^{n})$ up to level $12$-th and The Moore diagram of the nucleus is demonstrated in Figures \ref{The Schreier graph of automaton}, \ref{The Schreier graph 12}, and \ref{The nucleus's Moore diagram}. Obviously, each Schreier graph $\Gamma_{n}$ gives an approximation of $\hat{\Gamma}$. The Hausdorff-Gromov distance of the Schreier graph $\Gamma_{n}$ and the limit Schreier graph $\hat{\Gamma}$ tends to zero when $n\rightarrow\infty$.
\item[3- ]
By relying on results, which have been provided by R.Grigorchuk and A.Zuk \cite{grigorchuk1999asymptotic}, a numerical approximation of the spectrum and some eigenvectors (see \ref{smallest eigenvalues}) of the asymptotic Laplacian operator $\hat{\Delta}$ as a bounded
operator on space of $\mathbb{C}$-valued bounded function with compact support on the limit Schreier graph $\hat{\Gamma}=\lim_{n\rightarrow\infty} \Gamma_{n}$ is provided. The spectrum of the Laplacian operator $\Delta_{n}=D_{n}- M_{n}$, where $D_{n}$ and $M_{n}$ respectively are diagonal and adjacency $\vert A \vert^{n}\times \vert A \vert^{n}$-matrix on the $n$-th level of the rooted tree $A^{*}$, is computed up to level 12, which in turn gives the spectra approximation.

\item[4- ] Level-stabilizer, rigid-stabilizer subgroup corresponding to the first few levels of the binary tree, the portraits of some element of level-stabilizer subgroups,  and group's relations with a restricted length are computed \footnote{In this research, all computation of the automata group is conducted by a computer program based on Python written by the first author.
}.  
\end{itemize}

The following theorem indicates some basic properties of the self-similar action of group $\mathcal{G}$ on the binary rooted tree $A^{*}$:

\begin{thm}\label{main theorem}
Let $\mathcal{G}$ be self-similar group generated by the automaton $\Pi=\langle a,b,c,d \rangle$
which acts on the binary rooted tree $\mathcal{T}=\lbrace 0,1 \rbrace^{\mathbb{N}}$. Then $G$ satisfies the following:
\begin{itemize}
\item[1- ]
Group $G$ acts by contracting on the binary tree, i.e., there exists a minimal finite non-empty set (nucleus) $\mathcal{N}$ with property defined on  definition \ref{Nucleus}.
\item[2- ]
Group $G$ is a fractal group, i.e., $G$ is self-replicating and level transitive on $\mathcal{T}$ (definition \ref{self-replicating}).
\item[3- ]
The action of automaton $\Pi$ on the binary tree satisfies the open set condition (definition \ref{open set condition}).

\item[4- ] The automaton $\Pi$ has an exponential growth of activity.
\item[5- ] The automaton group $G$ is a weakly branch group.
\end{itemize}
\end{thm}
The requairments of the proof of theorem \ref{main theorem} is provided during section \ref{Self-similar action of}
\\
Section \ref{section one} defines and gives preliminary results related to automaton groups with self-similar action on a regular rooted tree. Section \ref{Self-similar action of} explores
the characteristic of the group's self-similar action on the binary tree, and the proof of the main theorem \ref{main theorem} is completed.
The Schreier graph $\Gamma(G, A^{n})$ on different levels of the rooted tree is presented in section \ref{Schreier graphs and the spectrum}. Furthermore, subsection \ref{Dynamical system and Heck operator} provides requirements for harmonic analysis on dynamical system $(\mathcal{J}_{G}, \sigma)$. Subsection \ref{The spectrum of Laplacian} tried to review the latest self-similar group’s spectra computation and further a numerical solution to the eigenvalue problem of the Laplacian operator $\Delta_{n}=I-M_{n}$ corresponding to main group $G:=\langle a^{\pm 1}, b^{\pm 1}, c^{\pm 1}, d^{\pm 1} \rangle$. In the last section, 4, The group’s relations are computed with a length less than 12, rigid-stabilizer subgroups and level-stabilizer subgroups for the first few levels are determined, and their associated portrait is presented.

\begin{wrapfigure}{r}{0.5\textwidth}
\begin{tikzpicture} [shorten >=1pt,node distance=2cm,auto]
  \tikzstyle{every state}=[fill={rgb:black,1;white,10}]

  \node[state,accepting]   (s)      {$1$};
  \node[state] (d) [left of=s]  {$d$};
  \node[state]           (b) [left of=d]     {$b$};
  \node[state] (a) [below  of=d] {$a$};
  \node[state]           (c) [left of=a]     {$c$};

  \path[->]
  (d)   edge              node {$0\vert 0$} (s)
  (a)      edge  [right, bend right]  node {$1\vert 0$} (s)
   (a)      edge  [right]            node { $0\vert 1$} (d)
    (c)      edge [below]   node {$0\vert 0$,\\$1\vert 1$} (a)
    (d)      edge  [above]  node {$1\vert 1$} (b)
    (b)      edge              node {$0\vert 0$} (a)
    (b)      edge   [left]    node {$1\vert 1$} (c)
  (s) edge [ loop right, above] node {$0\vert 0  , 1\vert 1$} (s);
\end{tikzpicture}

\caption{The Moore diagram of the automaton $\Pi=(\lbrace 1, a,b,c,d \rbrace , \lbrace 0, 1 \rbrace)$ }
    \label{Moore diagram}
\end{wrapfigure}

\section{Preliminary}\label{section one}
\vskip 0.4 true cm
This section reviews the definitions and theorems related to the self-similar groups generated by closed finite-state automata, their Schreier graphs, and the limit space.We have left theorems and propositions without proof, guiding the reader to the references mentioned.
The main resources that are used in this section are \cite{bartholdi2003fractal}\cite{bartholdi2006automata}\cite{nekrashevych2005self}\cite{nekrashevych2008groups}.\\

\begin{defi}
 The \textit{automaton} $\Pi$ is a tuple $(Q, A,\tau,\nu)$ where $Q$ is the set of states  and $A$ alphabet set associated with a \textit{transition} map $\tau : Q\times A \rightarrow  Q$ and \textit{output} map  $\nu : Q\times A \rightarrow  A$ obey the inductive rules: 
\begin{center}
$q\vert_{\emptyset} = q$    , $q(\emptyset) = \emptyset$  , \\
$q\vert_{x v}= q\vert_{x}\vert_{v}$   ,     $q(x v) = q(x)q\vert_{x}(v)$, 
\end{center}
for $x\in A$ , $v\in A^{*}$ and $q\in Q$ and $q\vert_{x}=\tau(q , x) , q(x)=\nu(q,x)$. 

By the composition low of automata , define the automaton $\Pi^{2}$ over alphabet set $A$ and state set $Q^{2}=Q\times Q$ such that
\begin{center}
$(q_{1}q_{2} , v)\mapsto q_{2}\vert_{q_{1}(v)}q_{1}\vert_{v}$ ,\\    $(q_{1}q_{2} , v)\mapsto q_{1}(q_{2}(v))$
\end{center}
for $v\in A^{*}$ and $q_{1} , q_{2} \in Q$.
However, by induction, one can obtain automaton $\Pi^{*}$
over alphabet set $A$ with the state set $Q^{*}=\bigsqcup_{\mathbb{N}}Q$. The states $Q^{*}$
of the automaton $\Pi^{*}$ with an empty set produce semi-group $G=\langle \Pi^{*} \rangle$.
\end{defi}
Note that the transition $\tau$ and output $\nu$ map can be extented to continuous functions $\hat{\tau}: Q\times A^{\omega}\rightarrow Q$ and $\hat{\nu}:Q\times A^{\omega}\rightarrow A^{\omega}$ with respect to to the pro-discrete topology.

The automaton $\Pi$ is \textit{invertible} if, for all $q\in \Pi$, the transformation $\tau (q, . ) : A \rightarrow A$ is inevitable; equivalently, $q$ acts by permutation on $A$. Consequently, the states of an invertible automaton $\Pi=\Pi^{-1}$ generate a group called the \textit{automaton group} denoted by $G=\langle \Pi \rangle$, which is a subgroup $G\leq Aut(A^{*})$ of the automorphism group of the rooted tree. An automaton $\Pi$ can be represented by its \textit{Moore diagram}(for example, figure \ref{Moore diagram}), a directed labeled graph whose vertices are identified with the states of $\Pi$. For every state $q\in \Pi$ and every letter $x \in A$, the diagram has an arrow from $q$ to $q\vert_{x}$ labeled by $x\vert q(x)$. This graph contains complete information about the automaton, and we will identify it with its Moore diagram.

 A faithful action of a group $G$ on $A^{*}$ (or on $A^{\omega}$) is said to be \textit{self-similar} if for every $g \in G$ and every $x \in A$, there exist $h \in G$ and $y \in A$ such that
\begin{center}
$g(x w) = y h(w)$
\end{center}
for every $w \in A^{*}$ (resp. $w \in A^{\omega}$). 
The pair $(h,y)$ is uniquely determined by the pair $(g,x)$, since the action is faithful. Hence we get an automaton with G as a set of states and with the following output and transition functions: 
\begin{center}
$g.x = y.h$
\end{center}
i.e., $y = g(x)$ and $h = g\vert_{x}$. This automaton is called the \textit{complete automaton} of the self-similar action

In other words, the faithful action of group $G$ on $A^{*}$ is self-similar if there exists an automaton, denoted by $(G, A)$, such that for all $g \in G$, the action of $g$ as a group element on the regular rooted tree $A^{*}$ coincides with the action of $g$ as a state of the automaton $(G, A)$ on the tree.

\begin{defi}\label{Nucleus} \cite{nekrashevych2005self}
A self-similar action $(G, A)$ is called \textit{contarcting} (or \textit{hyperbolic}) if ther exists a finite set $\mathcal{N}\subset G$ such that for every $g\in G$ there exists $k\in \mathbb{N}$ such that $g\vert_{v}\in \mathcal{N}$ for all words $v\in A^{*}$ of lenght$\geq k$. th minimal set with this property is called the \textit{nucleus} of the sel-similar action.
The nucleus of contracting action is unique and is defiend by:
\begin{center}
$\mathcal{N}:=\bigcup_{g\in G}\bigcap_{n\geq 0}\lbrace g\vert_{v} ; v\in A^{*} , \vert v \vert\geq n\rbrace$
\end{center}
if $g\vert_{v}= g$ for some $v\in A^{*}\setminus\emptyset$, then $g$ belongs to the nuclues by definition.
\end{defi}
\begin{lem}\label{lemma}\cite{nekrashevych2005self}
A sefl-similar action of a group with a generating set $S=S^{-1}$, $1\in S$ is contracting if and only if there exists a finite set $\mathcal{N}$ and a number $k\in \mathbb{N}$ such that 
\begin{center}
$(S\cup \mathcal{N})^{2}\vert_{A^{k}}\subseteq \mathcal{N}$
\end{center}
\end{lem}

\begin{defi}\label{self-replicating}\cite{bartholdi2003fractal}
A self-similar action is said to be \textit{self-replacing} (\textit{recurrent} or
\textit{fractal}) if it is transitive on the first level $A$ of the tree $ A^{*}$ and the map
\begin{center}
$\psi_{x} : St_{G}(1)\rightarrow G$ , $g\mapsto g\vert_{x}$
\end{center}
is surjective for some (then for every) letter $x \in A$. It can be shown that a self-replicating group acts transitively on $A^{n}$ for every $n \geq 1$. Nekrashevich proved (\cite{nekrashevych2005self}, Proposition 2.11.3) that if a finitely generated contracting group is self-replicating, then its nucleus $\mathcal{N}$ is a generating set $\langle\mathcal{N}\rangle=G$.

\end{defi}\label{open set condition}
\begin{defi}\cite{bartholdi2006automata}
It is said that a contracting action satisfies the \textit{open set condition} if every element of its nucleus has a trivial restriction. It is easy to see that this is equivalent to the condition that every group element $g\in G$ has a trivial restriction.
\begin{center}
$\forall g\in G, \exists v\in\mathcal{T} ;  \psi_{v}(g)=g\vert_{v}=1$
\end{center}
\end{defi}
\begin{defi} 
Let $G \leq Aut(A^{*})$ be a level-transitive group.
If all rigid stabilizers $RiStG(n)$ are infinite (equivalently, non-trivial), then one
say that $G$ is weakly branch.
The group is said to be branch if $RiSt(n)$ has finite index in $G$ for every $n\geq1$. 
 \end{defi}

\subsection{Dynamical system and Heck type operator}\label{Dynamical system and Heck operator}

The level-transitive action of $G$ on each $n$-th level $A^{n}$ is by permutation, which leads a family of unitary
representation on the Hilbert space of square summable (by our assumptions $A^{n}$ n is a finite set, one can substitute the square summability by the boundedness, $l^{2}(A^{n})=C_{b}(A^{n})$,
\begin{center}
$\rho_{n} : G_{n} \rightarrow \mathcal{B}(l^{2}(A^{n}))$ , $(\rho_{n}(g)f)(x) =  f(g^{-1}.x)$,
\end{center}
The regular representation $\rho_{n}$ is a subrepresentation of the tensor product $\bigotimes_{v\in A^{n}}\rho_{G/St(v)}$, authors \cite{bartholdi1999parabolic} showed,
that converges to a faithful regular representation $\lim_{n\rightarrow\infty}\rho_{n}=\rho$.

Let $G$ be finitely generated with generating set $\lbrace s_{1}, \cdots ,s_{i} \rbrace$ ,$1\leq i  \leq\vert S \vert$; for real parameters $\Gamma_{i}\in \mathbb{R}$, the \textit{Heck-type} operator $H_{\rho}$ is defined on $l^{2}(A^{n})$  as following \cite{bartholdi1999spectrum}
\begin{center}
$H_{\rho_{n}}(\Gamma_{1},\cdots , \Gamma_{\vert S \vert})= \sum_{i=1}^{\vert S \vert} \Gamma_{i} \rho_{n} (s_{i})$.
\end{center}

Although, Markov operator on Schreier graph $\Gamma_{n}$ is a restriction to $l^{2}(A^{n}, deg)$ of the Heck-type operator with respect to parameter $\Gamma_{i}=\frac{1}{\vert S \vert}$, for all $1\leq i  \leq\vert S \vert$ is considered. 
\begin{center}
$(M_{n} f)(x)=\dfrac{1}{\vert S \vert}\sum_{s\in S} f(s.x) = \dfrac{1}{deg(x)}\sum_{x\sim y}f(y)$, for $f\in l^{2}(A^{n}, deg)$
\end{center}

Let $G$ be a finitely generated automaton group that admits a faithful level-transitive self-similar action on the tree $A^{*}$ . The action on the tree rises  to an homeomorphic meaure-class preserving action on the boundary.
The boundary of the regular rooted tree $A^{*}$ denoted by $A^{-\mathbb{N}}$, which is homeomorphic to Cantor set, equipped with the uniform Bernoulli measure $\nu$ which is preserved by the $G$-action , i.e., $g_{*}(\nu)=\nu$ for all $g\in G$. 
The \textit{uniform Bernoulli measure} $\nu$ on the boundary $A^{-\mathbb{N}}$ 
is a Radon measure defined on the cylinders $A^{-\mathbb{N}}w$, for $w\in A^{*}$, by $\nu(A^{-\mathbb{N}}w)= d^{-\vert w \vert}$. That is  $G$-invariant, and is the unique invariant measure if $G$ acts transitively on each level $A^{n}$ of the tree, or equivalently if
$(A^{-\mathbb{N}}, S ,\nu)$ is ergodic.
 Suppose $G$ is generated by the symmetric generating set $S=S\cup S^{-1}$, hence the left $G$-action on the tree's boundary setups a (non-commutative) dynamical system denoted by  $(A^{-\mathbb{N}}, S ,\nu)$. Consequently, the left regular representation unitary representation $\rho$ on the Hilbert space  $L^{2}(A^{-\mathbb{N}}, \nu)$;
\begin{center}
 $\rho:G\rightarrow \mathcal{B}(L^{2}(A^{-\mathbb{N}}, \nu))$,
 $(\rho(g) f) (x) =  f(g^{-1}.x)$
\end{center}
Consider a contracting self-similarity denoted  by $(G, A)$. Let examine the uniform Bernoulli measure, denoted as $\nu$, on the space $A^{-\mathbb{N}}$. Additionally,  the counting measure on the group $G$, and the product measure is defined on the space $A^{-\mathbb{N}}\times G$. By applying the factor map $\pi: A^{-\mathbb{N}}\times G \rightarrow \mathcal{A}_G$,  the push-forward of this measure is obtained, which defines the measure $\mu$ on the limit $G$-space $\mathcal{A}_G$. The measure $\mu$ is a $G$-invariant, $\sigma$-finite, and regular Borel measure on $\mathcal{A}_G$. In the work by authors \cite{bondarenko2011graph}, the following properties of $\mu$ are demonstrated:

\begin{itemize}
\item[1- ] For any $v\in A^{*}$ and $u\in A^{n}$;  $\mu(\mathcal{T}_{v}) = \vert A \vert^{n}\mu(\mathcal{T}_{uv})$.

\item[2- ] For any $v,v^{\prime}\in A^{n}$ such that $v\neq v^{\prime}$; $\mu(\mathcal{T}_{v}\cap \mathcal{T}_{v^{\prime}}) = 0$.

\item[3- ] Let $\mu\vert_{\mathcal{T}}$ be the restriction of $\mu$ to the tile $\mathcal{T}=A^{-\mathbb{N}}\times 1$. Then, for any Borel set $E\subset\mathcal{T}$;
\begin{center}
$\mu\vert_{\mathcal{T}}(E) = \sum_{x\in A}\dfrac{1}{\vert A \vert}\mu\vert_{\mathcal{T}}(\sigma_{x}^{-1}(E))$
\end{center}
\end{itemize}
\begin{defi}
The image of $A^{- \mathbb{N}}\times 1$ in $\mathcal{A}_{G}$ is referred to as the \textit{(digit) tile} $\mathcal{T}$ of the action. The image of $A^{-\mathbb{N}}v \times 1$ for $v \in A^{n}$ is known as the tile $\mathcal{T}_{v}$, which can be expressed as $\mathcal{T}_{v} = \sigma_{v}(\mathcal{T})$. By definition, it follows that:
\begin{center}
$\mathcal{A}_{G} = \bigcup_{g\in G}\mathcal{T}.g$ and $\mathcal{T} = \bigcup_{v\in A^{n}}\mathcal{T}_{v}$
\end{center}

where $\sigma_{v}$ represents a branch inverse of the right-shift map, given by $\sigma_{x}(\cdots x_{2}x_{1}) = \cdots x_{2}x_{1}x$. The tile $\mathcal{T}$ in the limit $G$-space $\mathcal{A}_{G}$ can be interpreted as the attractor of the iterated function $\sigma{x}$ for $x\in A$, which can be written as:

\begin{center}
$\mathcal{T} = \bigcup_{v\in A^{n}}\sigma_{v}(\mathcal{T})$
\end{center}

\end{defi}
Therefore, one can construct a unitary representation of the self-similar group $G$ on the bounded linear endomorphisms of the Hilbert space $L^{2}(A^{-\mathbb{N}}, \mu)$;
\begin{center}
 $\rho : G\rightarrow\mathcal{B}(L^{2}(A^{-\mathbb{N}}, \mu))$\\
 $(\rho(g) f) (x) =  f(g^{-1}.x)$, for $f\in L^{2}(A^{-\mathbb{N}}, \mu)$
\end{center}
\begin{defi}
For a finitely generated group $G$ with generating set $\lbrace a_{1}, \cdots ,a_{r} \rbrace$ ,$r\in \mathbb{N}$, authors \cite{bartholdi2000spectra} defined the Markov operator as follow:\\
For some real parameters $s_{1}, \cdots ,s_{r}\in \mathbb{R}$;
\begin{center}
$M_{\rho}(s_{1}, \cdots ,s_{r}):= \sum_{i=1}^{r} s_{i} \rho (a_{i}) \in \mathcal{B}(L^{2}(A^{-\mathbb{N}}, \nu))$.
\end{center}
\end{defi}

\section{Self-similar action on the rooted tree}\label{Self-similar action of}

The current section is devoted to probing the action of group $G$ on the rooted binary tree, and the portrait of some group elements at different levels of the tree is calculated. The nucleus $\mathcal{N}$ of automaton $Q$ is given, which makes the self-similar group action into a contracting action.\\
Recall that group $\mathcal{G}$ is generated by sates of automaton by the wreath recursion:
$a=\sigma (d ,1)$ , $b=(a,c)$ , $c=(a,a)$ , $d=(1,b)$.
The group $G$ is assembled as a quotient of the free group $\mathbb{F}(S)$ with generating set $S=\lbrace a, b, c, d \rbrace$ modulo subgroup $H$, which is a free group on the set of relations. Let $d_P{S}$ be the word-metric associated with the collection of generators $S$ assigns to each element $g\in G$ a real number $\vert g \vert$ which is the length of the shortest path issued from idendity element in $G$. Duo to the word metric, the Cayley graph $\Gamma(\mathbb{F},S)$ of the free group $\mathbb{F}(S)$ is an $\vert S \vert$-regular rooted tree $\mathcal{T}$, and the Schreier graph $\Gamma(\mathbb{F},H,S)$ is a quotient space of the tree. 
A group homomorphism defines the self-similar group action on the regular rooted tree,
\begin{center}
$\psi_{n} :G\rightarrow Sym(A^{n})\wr G$ \\
$\psi_{n}(g) = \sigma(g\vert_{v_{1}},\cdots, g\vert_{v_{d}})$ for $d=\vert A^{n}\vert$
\end{center}
Where $\sigma$ is a permutation of $A^{n}$ belongs to the symmetry group of permutation $Sym(A^{n})$ on the set $A^{n}$, the homomorphism image of each element $g\in G$, $\psi_{n}(g)$ is called the \textit{portrait} of the element $g$ at level $n$-th. The portrait of each generating element at the first level is shown in  \ref{wreath recurssion}.

Since each generator of the group $G$ acts by permutation on alphabet letters, the automaton $(Q, A)$ is invertible, and the inverse automaton $(Q^{-1}, A)$ is defined by wreath recursion:

\begin{center}
$a^{-1}=\sigma (1 ,d^{-1}) , b^{-1}=(a^{-1},c^{-1}) , c^{-1}=(a^{-1},a^{-1}) , d^{-1}=(1,b^{-1}) $
\end{center}
\begin{lem}\label{contracting lemma}
Group $G$ generated by the automaton \ref{Moore diagram}  $(Q,A)$  acts by contracting automorphism on the rooted tree $A^{*}$.
\end{lem}
\begin{proof}
We introduce the set $\mathcal{N}$ as the collection of group elements and assert that it forms the nucleus of the automaton.
Evidently, $\mathcal{N}$ generates the group, as it encompasses the group generators $S = {a, b, c, d}$.

Let put:
\begin{flushleft}
$\mathcal{U} =\lbrace a$, $ b$, $ d$, $ c$, $ da$, $ bd$, $ cb$, $ ac$, $ a^{-1}b$, $ d^{-1}c$, $ a^{-1}c$, $ d^{-1}a$, $ b^{-1}d$, $ c^{-1}b$, $ a^{2}$,
$a^{-1}da$, $ d^{-1}bd$, $ b^{-1}cb$, \\
$ c^{-1}ac$, $a^{-1}bd$,
 $ d^{-1}cb$, $ b^{-1}ac$, $ c^{-1}da$, $ a^{-1}cb$, $ d^{-1}ac$, $ b^{-1}da$,
$ c^{-1}bd$, $ a^{-1}d^{-1}bd$, $ d^{-1}b^{-1}cb$, $ b^{-1}c^{-1}ac$, $ c^{-1}a^{-1}da$, 
$ a^{-1}d^{-1}cb$, $ d^{-1}b^{-1}ac
 \rbrace$
\end{flushleft}
By definition, the nucleus of an automaton should be a symmetry set, therefore define:
\begin{equation}\label{nucleus}
\mathcal{N}= \mathcal{U}\cup \mathcal{U}^{-1}
\end{equation}

To establish $\mathcal{N}$ as the desired nucleus, two key criteria must be met. Firstly, it must be demonstrated that $\mathcal{N}$ remains invariant under the virtual endomorphism $\phi : G \rightarrow G$ (which establishes the self-similarity of the group). This is attempted to be showcased through the Moore diagram of the nucleus depicted in Figure \ref{The nucleus's Moore diagram}. In fact, Moore diagram illustrates that for any $x \in \mathcal{N}$, we have $\phi_{i}(x) \in \mathcal{N}$ for $i = 0, 1$.

Following this, we need to confirm that $\mathcal{N}$ satisfies the conditions outlined in Lemma \ref{lemma}. Since the nuclus contains the group generators $S=\lbrace a, b, c, d \rbrace\subset\mathcal{N}$, to achieve this, it suffices to demonstrate that there exists an integer $n_{0}$ such that for all $n \geq n_{0}$:
\begin{equation}\label{inclusion}
((S\cup S^{-1}) \otimes \mathcal{N})\vert_{A^{n}} \subseteq \mathcal{N}
\end{equation}
Through direct calculations the image of group's element  $a x$ under virtual-endomorphism $\phi$ for each pair of group's
 elements $(a, x)$, where  $a\in  S$ and $x\in\mathcal{N}$,
We observed that for $n_{0}=7$ the above inclusion (\ref{inclusion}) is hold, i.e., $a\otimes x\vert_{X^{7}}\in \mathcal{N}$
 (for brevity, detailed calculations are provided in Appendix \ref{appendix}), as a result, the above condition is satisfied for $n_{0} = 7$, thereby completing the proof.

\end{proof}

\begin{figure}
 \centering
    \includegraphics[width=0.5\linewidth]{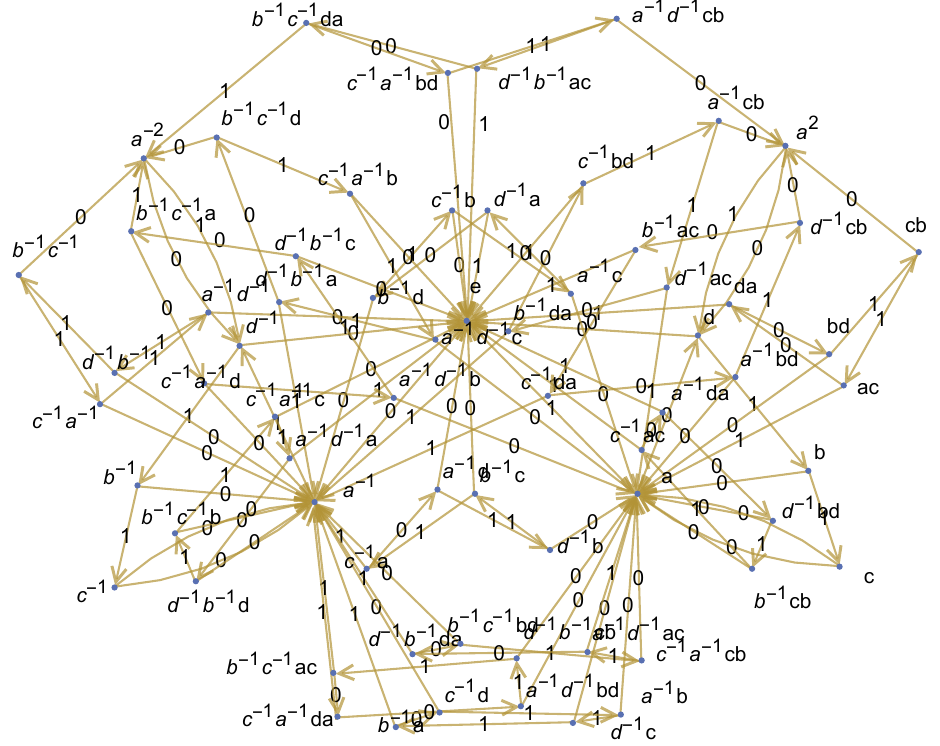}
      \caption{The nucleus's Moore diagram}
  \label{The nucleus's Moore diagram}
\end{figure}

 \begin{lem}\label{fractal lemma}
Group $G$ is a fractal group.

\end{lem}
\begin{proof}
The transitive action on the first level is evident; for each letter $x,y\in \lbrace 0,1 \rbrace$, the element $a$ permutes $x$ and $y$, $a.x=\sigma(x)=y$. For self-replicating, to demonstarting that the virtual endomorphism $\psi_{x}:St_{G}(x)\rightarrow G$ for any $v\in A^{*}$  is surjective, that is sufficient to show for each generator $s \in \lbrace a^{\pm} ,b^{\pm},c^{\pm},d^{\pm} \rbrace$, and any letter $x\in A$ there exists an element $g\in Dom(\psi_{x})=St_{G}(x)$ such that $g\vert_{x}= s $.
\begin{center}

$a^{\pm1}= \psi_{x}(c^{\pm})=c^{\pm}\vert_{x} , x\in A $  ,  $ b=(a d a^{-1})\vert_{0} =a\vert_{1} d\vert_{1} a^{-1}\vert_{0}= d\vert_{1}$ , \\ $c=(a b a^{-1})\vert_{0}= a\vert_{1} b\vert_{1}a^{-1}\vert_{0}$, $d=a\vert_{0}= a^{2}\vert_{1}=a\vert_{0} a\vert_{1}$, \\
 $ b^{-1}=(a d^{-1}a^{-1})\vert_{0} =d^{-1}\vert_{1}$,$c^{-1}=(a b^{-1}a^{-1})\vert_{0}= b^{-1}\vert_{1}$,\\ $d^{-1}=a^{-2}\vert_{0}= a^{-1}\vert_{1}$
\end{center}
Now assume $\psi_{w}$ is surjective for $w\in A^{n-1}$, let $v\in A^{n}$ and $s\in \lbrace a^{\pm 1} ,b^{\pm 1},c^{\pm 1},d^{\pm 1} \rbrace$. The vertiex $v$ is a finite word with length $n$, $v=x_{1}x_{2}\cdots x_{n}$, where $x_{i}\in A$. We have to find element $g\in G$ such that $g\vert_{v}=s$, on the other hand
\begin{center}
$g\vert_{v} = g\vert_{x_{1}}\vert_{x_{2}}\cdots\vert_{x_{n}}$
\end{center}
hence it is suffient to find  elements  $g_{n},g_{n-1}\cdots g_{1}\in G$ such that $g_{i}\vert_{x_{i+1}}=g_{i+1}$, due to above argument it is always possible.

\end{proof}
We have observed the following:
\begin{cor}\label{coro self-replicating}
  The automaton group $G$  has a contracting self-replacing action on the binary tree.
\end{cor}

The nucleus of automaton group $G$ has 67 elements, and by proposition 2.11.3 \cite{nekrashevych2005self}, since the automaton group $G$ is contracting self-replicating, the nucleus generates the group $G=\langle \mathcal{N}\rangle$.

\begin{lem}\label{openset lemma}
The automaton group $G$ satisfies the open set condition. 
\end{lem}
\begin{proof}
Since the automaton group $G$ has contracting action on the binary tree $A^{*}$, it is sufficient to finid finite words $v_{g}\in A^{*}$ for each element $g\in \mathcal{N}$ in the nucleus such that $g\vert_{v_{g}}=1$. \\
$a\vert_{1}=d\vert_{0}=b\vert_{01}=c\vert_{11}=1$   ,               
$(c^{-1}a^{-1}b d)\vert_{000}=1$,
$(b^{-1}c^{-1}d a)\vert_{110}=1$,
$(d^{-1}b^{-1}a c)\vert_{000}=1$,
$(a^{-1}d^{-1}c b)\vert_{111}=1$,
$(a^{-1}d^{-1}a c)\vert_{111}=1$,
$(c^{-1}a^{-1}c b)\vert_{010}=1$,
$(b^{-1}c^{-1}b d)\vert_{000}=1$,
$(d^{-1}b^{-1}d a)\vert_{000}=1$,
$(c^{-1}a^{-1}d a)\vert_{001}=1$,
$(b^{-1}c^{-1}a c)\vert_{110}=1$,
$(d^{-1}b^{-1}c b)\vert_{110}=1$,
$(a^{-1}d^{-1}b d)\vert_{000}=1$,
$(d^{-1}b^{-1}c)\vert_{000}=1$,
      
\end{proof}
Now we are going to complete the proof of Theorem \ref{main theorem};
\begin{proof}
Lemma \ref{contracting lemma}, \ref{fractal lemma},  Corollary \ref{coro self-replicating}, Lemma \ref{openset lemma} proved items 1 - 3 of the main theorem \ref{main theorem}, respectively.\\
The exponential growth of activity of automaton $\Pi$ simply follows from the fact that there are two intersecting cycles in the Moore-diagram of the automaton.
\begin{center}
$b\rightarrow a\rightarrow d\rightarrow b$, and $b\rightarrow c\rightarrow a \rightarrow d\rightarrow b$.
\end{center}

(5-) Now we show that the group $G$ is a weakly branch:\\

Due to the level transitive action of the main group on binary rooted tree $A^{*}$, if for some vertex $v\in A^{n}$ on the level $n$-th the vertex-rigid-stabilizer is non-trivial $Rigid_{G}(v)\neq\lbrace 1 \rbrace$ then $Rigid_{G}(u)\neq\lbrace 1 \rbrace$ is non-trivial for all vertices $u\in A^{n}$  on the same level $n$-th. Hence, it is enough to show that for any $k\geq 0$, a non-trivial group’s element exists such that $1 \neq g\in Rigid(4k)\lhd G$, equivalently, there is an element of digit-tile, i.e., $\omega\in(A^{4})^{-\mathbb{N}}\simeq A^{-\mathbb{N}}$, such that for any $k\geq 0$, the rigid-stabilizer subgroup $Rigid_{G}(\omega(4k))$ is non-trivial.
Let us consider element $1\neq [a , c^{-2^{k}}] \in Stab(4k)\unlhd G$, by induction we can show that $[a , c^{-2^{k}}]\in Rigid((0111)^{k})$:

For $k=1$; 
$[a , c^{-2}]\vert_{v}=
\begin{cases}
        [a , c^{-1}] & \text{if } v=0111\\
        1 & \text{if } v\neq 0111
    \end{cases}$, 
That shows $[a , c^{-2}]$ belongs to the rigid-stabilizer subgroup of the vertex $0111\in A^{4}$.    
    
Assume $[a , c^{-2^{k-1}}]\in Rigid((0111)^{k-1})$ is holed for some $k>1$; 
With respect to  \ref{eq-stablizer}  one has $(c^{2^{k-1}})\in Stab(4)$ for $k> 1$. By product rule of commutator's elements, one has:
\begin{center}
$[a , c^{-2^{k}}]=[a , c^{-2^{k-1}}]\underbrace{c^{-2^{k-1}}[a , c^{-2^{k-1}}]\underbrace{c^{2^{k-1}}}_{\in Stab(4)}}_{\in Rigid(0111)}$,
\end{center}
Hence the non-trivialness of the  element $[a ,c^{-1}]$ induces non-trivialness of  the element $[a , c^{-2^{k}}]\in Stab(3k+1)$, for all  $k\geq 0$.
Further, the portrate of  $[a , c^{-2^{k}}]$ on the level 4-th is:
\begin{equation}\label{p-eq}
\phi_{4}( [a , c^{-2^{k}}])= (1, (1, 1, 1, 1, 1, 1, 1,  \underbrace{[a , c^{-2^{k-1}}]}_{\in Rigid((0111)^{k-1})}, 1, 1, 1, 1, 1, 1, 1, 1))
\end{equation}
where $\phi_{4}: G \rightarrow Sym(A^{4})\wr G$  is the portrate on the level 4.\\
Therefore $1\neq [a , c^{-2^{k}}]\in Rigid((0111)^{k})$ for all $k\geq 1$, and this completes the proof.
\end{proof}

\section{Schreier graphs and the spectrum}\label{Schreier graphs and the spectrum}
L. Bartholdi and R. Grigorchuk conducted a comprehensive study in \cite{bartholdi1999spectrum}, exploring the Schreier graphs that arise from the action of the Fabrykowski-Gupta group on the levels of the tree $A^{*}$. Their analysis included computing the spectra of these graphs. Notably, they made a significant observation: the Schreier graphs exhibit convergence towards a fractal set. This observation played a crucial role in developing and defining the limit space for contracting self-similar groups. As a result, their findings contributed significantly to the understanding of these mathematical structures.

\begin{defi}\label{definition Schreier graph}
Let $G$ be a group generated by a finite set $S$ and let $H$ be a subgroup of $G$. The \textit{Schreier graph} $\Gamma(G,S,H)$ of the group $G$ is the graph whose vertices are the right cosets $G/H = \lbrace Hg : g \in G\rbrace$, and two vertices $Hg_{1}$ and $Hg_{2}$ are adjacent if there exists $s \in S$ such that $g_{2} = g_{1}.s$ or $g_{1} = g_{2}.s$. 
Suppose the finitely generated group $G$ acts on a set $M$. Then the corresponding \textit{Schreier graph}  $\Gamma(G,S,M)$ is the graph with the set of vertices $M$ and set of arrows $S \times M$, where the arrow $(s,v), v\rightarrow s(v)$ starts in $v$ and ends in $s(v)$. The \textit{simplicial Schreier graph} $\overline{\Gamma}(G,S,M)$ remembers only the vertex adjacency (see figure \ref{The Schreier graph of automaton}). The simplicial Schreier graph is an undirected graph with a set of vertices $M$, and two vertices are adjacent if and only if one is an image of another under the action of a generator $s \in S$. 
\end{defi}
  \begin{figure}[h!]
  \begin{subfigure}[b]{0.45\linewidth}
  \centering
    \includegraphics[width=1\textwidth]{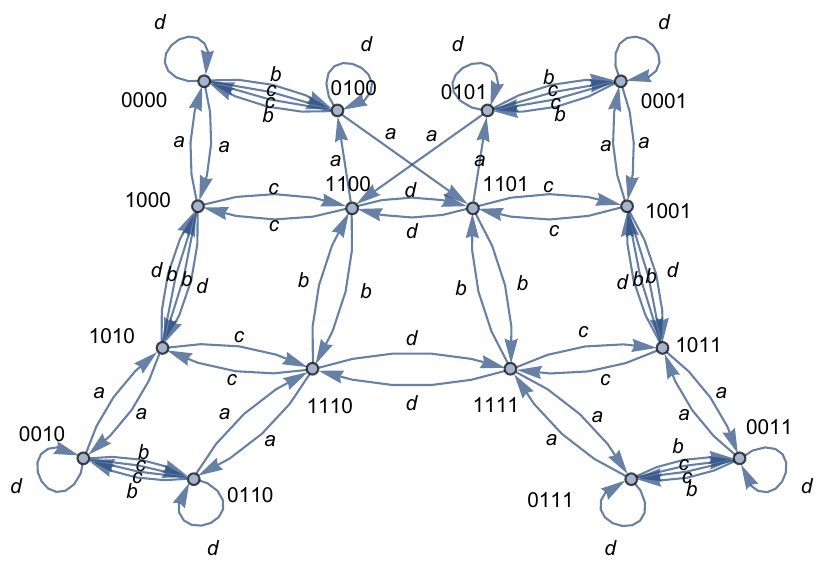}
    \caption{level 4}
  \end{subfigure}
\begin{subfigure}[b]{0.45\linewidth}\centering
    \includegraphics[width=1.2\textwidth]{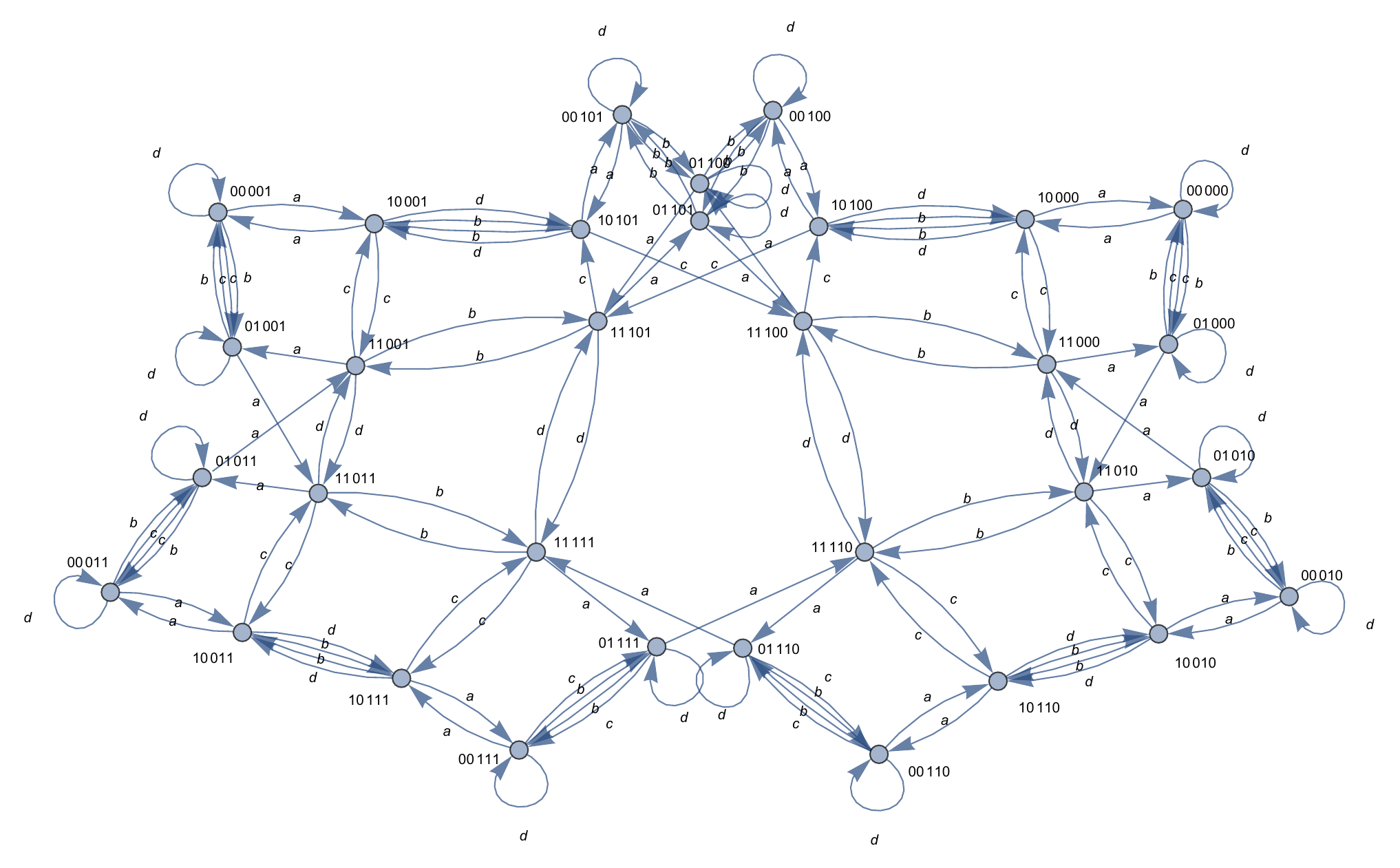}
     \caption{level 5}
  \end{subfigure}
  \caption
{The Schreier graph of the the automaton generated with the wreath recursion $a=\sigma (d ,1) , b=(a,c) , c=(a,a) , d=(1,b)$.}
  \label{The Schreier graph of automaton}
\end{figure}

Assume $G$ is a finitely generated self-similar group that acts faithfully and level transitively on the vertices of the tree $A^{*}$. Then level stabilizer subgroups $St_{G}(n)\unlhd G$ form a decreasing sequence of the normal subgroups of the finite index such that;
\begin{center}
$G\unrhd St(1)\unrhd\cdots St(n)\unrhd St(n+1)\cdots$ ;  $\bigcap_{n\in\mathbb{N}} St_{G}(n) = 1$
  , $[G : St(n)]< \infty$
\end{center}
Which provides a family of open neighborhoods around the identity element $e\in G$, yields a profinite completion $\hat{G}=\underleftarrow{\lim} G_{n}$ where $G_{n}:=G/St(n)$. 

Let $S$ is a finite symmetry generating set of $G$, i.e., $S=S^{-1}$, and $\lbrace v_{n} \rbrace_{n>0}$ is a sequence of finite words such that each word $v_{n}$ has word $v_{n-1}$ as a prefix; $v_{n}= x_{n}\cdots x_{2}x_{1} \in X^{n}$. 
Therfore the associated Schreier graphs $\Gamma_{n}=\Gamma(G, S, St(n)) = \Gamma(G_{n}, S, A^{n})$ (resp; orbital Schreier graphs $\Gamma_{v}=(G/St(v),S, A^{n})$, for $v\in A^{n}$) form a family of connected unifomly bounded degree (resp; marked) graphs $\lbrace(\Gamma_{n})\rbrace_{n\in\mathbb{N}}$ (resp; $\lbrace(\Gamma_{v},v_{n})\rbrace_{n\in\mathbb{N}}$ with base point $v_{n}\in A^{n}$). Further, the Schreier graph on the $n+1$-th level with the shift map $\sigma$, as a covering map, is a covering space on the $n$-th Schreier graph, i.e., $\sigma :(\Gamma_{n+1}, v_{n+1})\rightarrow (\Gamma_{n}, \sigma(v_{n+1})=v_{n})$. It can be seen that the family $\lbrace\Gamma_{n}\rbrace$ converges to an unifomly bounded degree (resp; marked) infinite graph $(\hat{\Gamma}, \hat{v})$. The convergence happens by the Hausdorff-Gromov metric on the space of marked graphs of uniformly bounded degrees (for more details see  \cite{grigorchuk1999asymptotic}).
Let $\cdots x_{2}x_{1}=\hat{v}\in A^{-\mathbb{N}}$;
\begin{center}
$d_{H}((\Gamma_{n}, v_{n}), (\hat{\Gamma}, \hat{v})):=\inf\lbrace \dfrac{1}{r+1} \vert B_{\Gamma_{n}}(v_{n} , r)$ is isometric to $B_{\Gamma_{\hat{v}}}(\hat{v} ,r) \rbrace$\\
$\lim_{n\rightarrow\infty}d_{H}((\Gamma_{n}, v_{n}), (\hat{\Gamma}, \hat{v})) = 0$,
\end{center}
Where $v_{n}=x_{n}\cdots x_{2}x_{1}\in A^{n}$, the infimum is taken over balls $B_{\Gamma}(v_{n},r)\subset \Gamma$ with center $v_{n}$ and radius $r$ in Schreier graph $\Gamma$. Nekrashevich \cite{nekrashevych2005self} showed when the iterated monodromy group $IMG(f)$, for a rational self-covering of a Riemann surface $f:S\rightarrow S$, acts by contraction on the regular rooted tree $A^{*}$, one has homeomorphism $\hat{\Gamma}\simeq\mathcal{J}_{IMG(f)}$.

\subsection{The spectrum of Laplacian}\label{The spectrum of Laplacian}

Let $(G,A)$ be an automaton group that acts level-transitively defined over alphabet $A$ with a finite symmetric generating set $S=S^{-1}$ possess a decreasing sequence of finite index normal subgroups $St(n)$ , i.e., $[G : St(n)]< \infty$. The decreasing sequence provides a tower of finite graphs as a sequence of covering spaces $\cdots\Gamma_{n}\rightarrow\Gamma_{n-1}\cdots$ with covering transformation groups $G_{n}=G/St(n)$. The limit Schreier graph $\hat{\Gamma} = \lim \Gamma_{n}$, which is an infinite locally finite connected graph when the self-similar $G$-action is level-transitive contacting self-replicating. Indeed $(\hat{\Gamma}, \nu)$, is the congruence of a family of marked finite graph $\lbrace (\Gamma_{n}, \nu_{n}) \rbrace$,i.e., $n$-th level Schreier graph, with respect to Gromov-Hausdorff metric, where $\nu=\cdots x_{2}x_{1}\in A^{-\mathbb{N}}$, $v_{n}=x_{n}\cdots x_{2}x_{1}$, and $d_{H}( (\Gamma_{n}, \nu_{n}),(\hat{\Gamma}, \nu) )\rightarrow 0$, with vertex set $A^{-\mathbb{N}}$.
 
 The combination of Kesten \cite{kesten1959symmetric}(lemma 1.4) and
the generalized Alon–Boppana by \cite{grigorchuk1999asymptotic} theorem results:
\begin{center}
$\dfrac{2 \sqrt{\vert S \vert -1}}{\vert S \vert}\leq\Gamma(G,S) \leq \liminf_{n\rightarrow\infty} \Gamma(G/St(n), \bar{S})$
\end{center}
Where $\Gamma(G,S)$ and $\Gamma(G/St(n) ,\bar{S})$ respectively is the spectral radius of the Markov operator on Cayley graph $\Gamma(G,S)$
graphs and $\Gamma(G/St(n) ,\bar{S})$, i.e., 
\begin{center}
$\Gamma(G,S) = \limsup_{q\in\mathbb{N}} [$ Probability of returning to $e$ at the $q$-th step given  that one starts at $e ]^{1/q}$ , and  $\Gamma(G/St(n) ,\bar{S})= \limsup_{q\in\mathbb{N}} [$ Probability of returning to $St(n)$ at the $q$-th step given  that one starts at $e ]^{1/q}$.
\end{center}

  \begin{figure}[h!]
  \begin{subfigure}[a]{0.45\linewidth}
  \centering
    \includegraphics[width=1\textwidth]{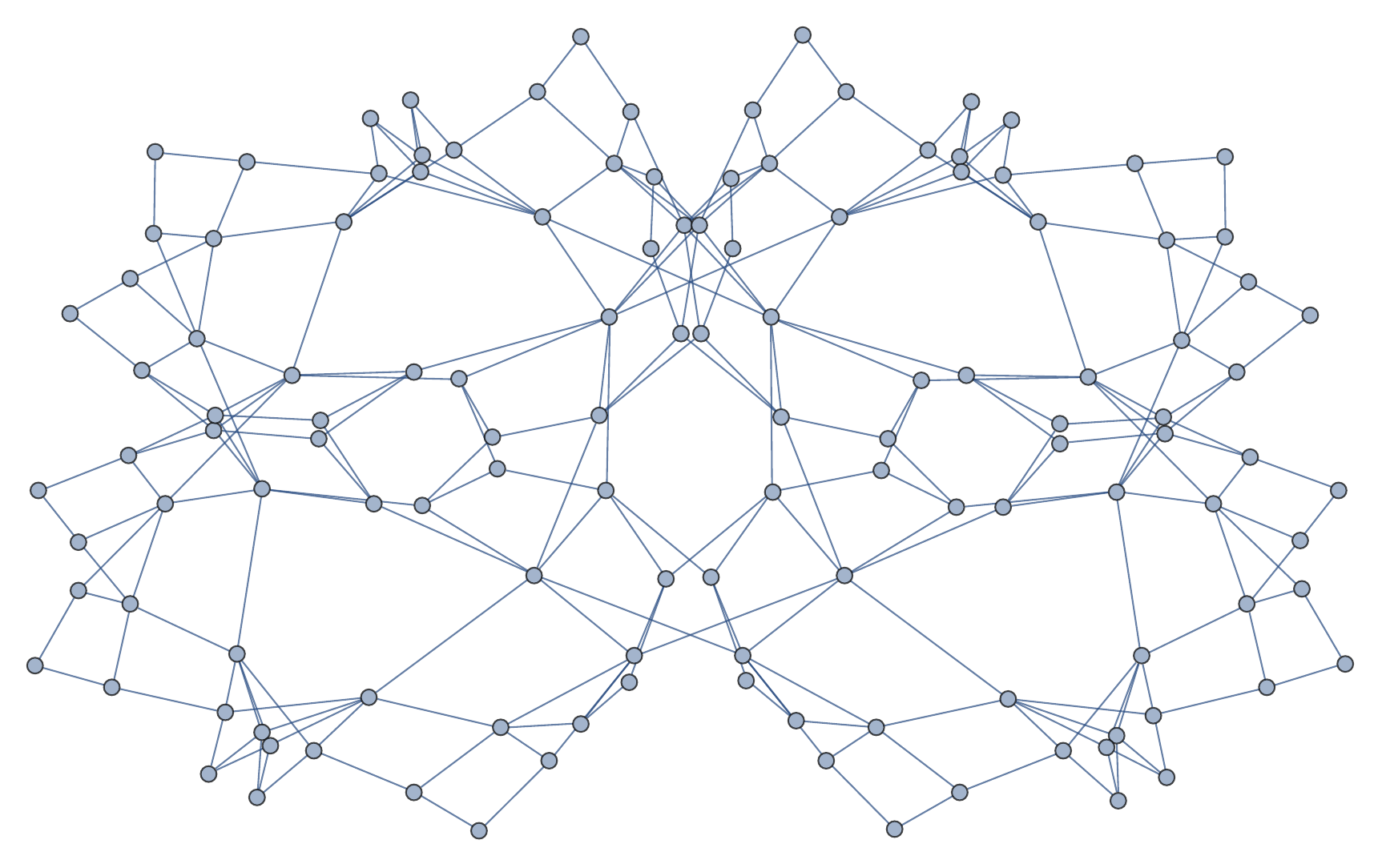}
    \caption{Schreier graph}
  \end{subfigure}
\begin{subfigure}[b]{0.4\linewidth}
\centering
    \includegraphics[width=1\textwidth]{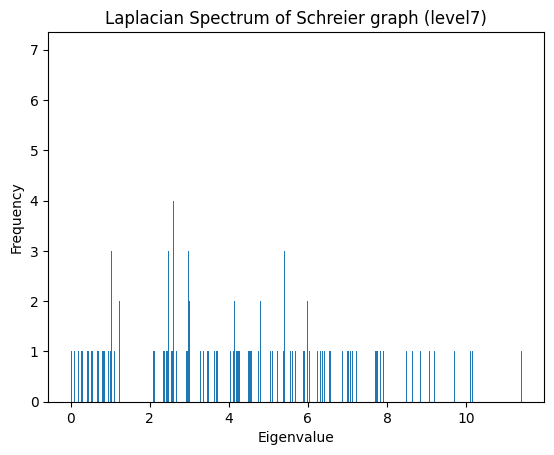}
     \caption{The spectrum}
  \end{subfigure}
  \caption
{The simplicial Schreier graph and its spectra on the 7-th level.}
  \label{The Schreier graph 7}
\end{figure}

  \begin{figure}[h!]
  \begin{subfigure}[b]{0.45\linewidth}
  \centering
    \includegraphics[width=1\textwidth]{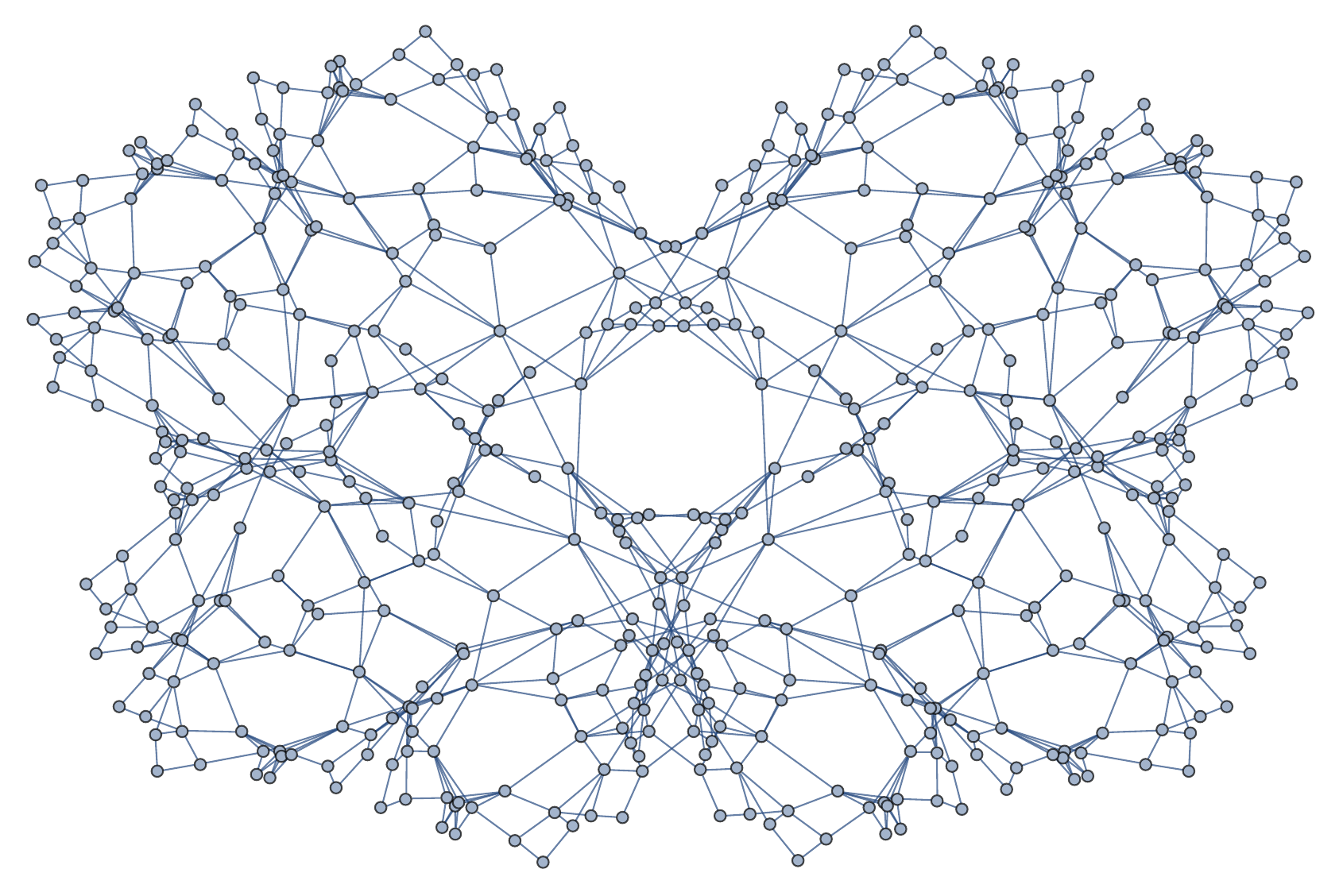}
    \caption{Schreier graph}
  \end{subfigure}
\begin{subfigure}[b]{0.4\linewidth}\centering
    \includegraphics[width=1\textwidth]{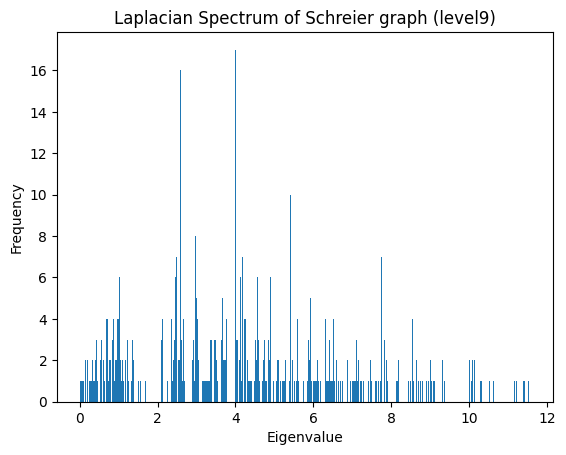}
     \caption{The spectrum}
  \end{subfigure}
  \caption
{The simplicial Schreier graph and its spectra on the 9-th level.}
  \label{The Schreier graph 9}
\end{figure}
  \begin{figure}[h!]
  \begin{subfigure}[b]{0.5\linewidth}
  \centering
    \includegraphics[width=1\textwidth]{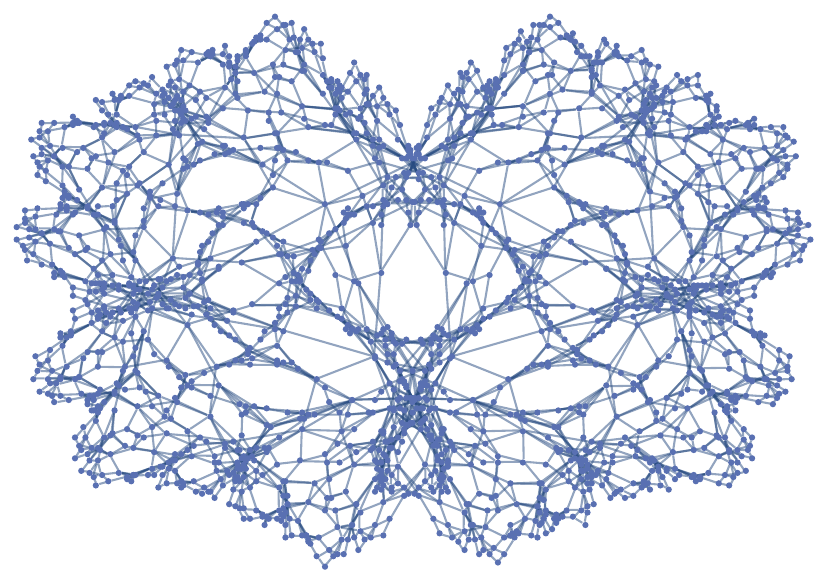}
    \caption{Schreier graph}
  \end{subfigure}
\begin{subfigure}[b]{0.4\linewidth}\centering
    \includegraphics[width=1\textwidth]{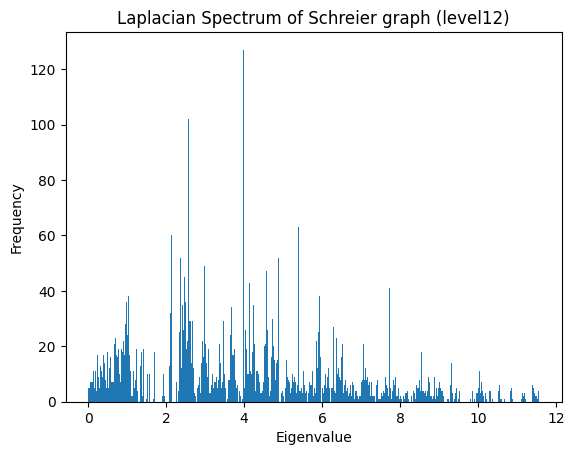}
     \caption{The spectrum}
  \end{subfigure}
  \caption
{The simplicial Schreier graph and its spectra on the 12-th level.}
  \label{The Schreier graph 12}
\end{figure}

Despite massive literature about the spectra of Markov or Laplacian operator on Cayley graphs $(G,S)$ of infinite groups, the first exclusive spectrum computation of a self-similar group appeared in works \cite{grigorchuk2001lamplighter}\cite{bartholdi1999spectrum} by R.Grigorchuk, A.Zuk, and L. Bartholdy. in \cite{bartholdi1999spectrum}, the authors presented the first example of a connected 4-regular graph, which is the Schreier graph of a group (Grigorchuk group) with intermediate growth and Cantor set as the spectrum of the Laplacian operator. They formulated their findings as quasi-regular representations of fractal groups. Authors \cite{grigorchuk2020spectra} also precisely determined the spectrum
for a class of spinal groups with specific parameters: $d \geq 2$, $m \geq 1$, and $\omega \in \Omega_{m, d}$. Another interesting discovery was made in \cite{grigorchuk2012spectral}, where it was demonstrated that the spectrum of the Cayley graph of the Lamplighter group $\mathbb{Z}_{2}\wr \mathbb{Z}$ consists entirely of point spectra. Furthermore, authors in [\cite{grigorchuk1999asymptotic}, \cite{grigorchuk2012spectral}, and \cite{grigorchuk2021spectra} computed the spectral measure of the Markovian operator on the Schreier graph of an infinite discrete group, which possesses a
decreasing sequence of normal subgroups with finite indices. This infinite graph serves as the convergence space for a sequence of graphs forming towers of coverings on a d-regular finite graph. Further, \cite{bondarenko2008classification} provides a classification of groups generated by a 3-state automaton over a 2-letter alphabet. In addition, a histogram of the spectrum of the Schreier graph adjacency operator corresponding to the ninth-level action is displayed.

The Schur complement method, initially employed for computing the Laplacian spectrum on fractals, as discussed in \cite{kigami2001analysis} and \cite{lindstrom1990brownian}, was further applied by Grigorchuk \cite{grigorchuk2006self} to explore the spectra of automaton groups. However, it is important to note that the Schur complement method is not always usable, for instance see \cite{grigorchuk2006self} Example 23.

Ironically, this issue also pertains to our current situation when utilizing the Schur complement method to compute the Schreier spectrum of our automaton group $G:=\langle a^{\pm 1}, b^{\pm 1}, c^{\pm 1}, d^{\pm 1} \rangle$, we encounter:
\begin{center}
$M_{n}-\Gamma I_{n}= a_{n}+ b_{n}+ c_{n}+  d_{n}-\Gamma I_{n} = 
\begin{pmatrix}
2 a_{n+1}-(\Gamma -1)I_{n-1} & 1 \\
d_{n-1}  &   a_{n-1}+  b_{n-1}+  c_{n-1}-\Gamma  I_{n-1}
\end{pmatrix}$
\end{center}
where $I_{n}$ denotes $d^{n}\times d^{n}$ the identity matrix.
In the sequel, we have to determine Schur's components:
\begin{center}
$S_{1}(M_{n}-\Gamma I)=2 a_{n-1}-(\Gamma -1)-( a_{n-1}+  b_{n-1}+  c_{n-1}-\Gamma  I_{n-1})^{-1}d_{n-1}$;
$S_{2}(M_{n}-\Gamma I)=a_{n-1}+  b_{n-1}+  c_{n-1}-\Gamma  I_{n-1}- d_{n-1}(2 a_{n-1}-(\Gamma -1))^{-1}$
$M_{n}-\Gamma I=\begin{pmatrix}
2 a_{n-1}-(\Gamma -1)-( a_{n-1}+  b_{n-1}+  c_{n-1}-\Gamma  I_{n-1})^{-1}d_{n-1} & 0 \\
0  &   a_{n-1}+  b_{n-1}+  c_{n-1}-\Gamma  I_{n-1}
\end{pmatrix}$,
$M_{n}-\Gamma I=\begin{pmatrix}
2 a_{n-1}-(\Gamma -1)I_{n-1} & 0 \\
0  &   a_{n-1}+  b_{n-1}+  c_{n-1}-\Gamma  I_{n-1}- d_{n-1}(2 a_{n-1}-(\Gamma -1)I)^{-1}
\end{pmatrix}$
\end{center}
However, Schur components are only sometimes definable because $ a_{n-1}+  b_{n-1}+  c_{n-1}-\Gamma  I_{n-1}$ or $(2 a_{n-1}-(\Gamma -1)$ is not always invertible. The operator $M_{n}-\Gamma I$ is invertible if and only if simultaneously $S_{1}(M_{n}-\Gamma I)$ and $a_{n-1}+  b_{n-1}+  c_{n-1}-\Gamma  I_{n-1}$, or equivalently, $S_{2}(M_{n}-\Gamma I)$ and $2 a_{n-1}-(\Gamma -1)I_{n-1}$ are invertible. Therefore to determine the spectrum recursively, one has to solve the following;
\begin{center}
$\det (M_{n}-\Gamma I)=\det(2 a_{n-1}-(\Gamma -1)-( a_{n-1}+  b_{n-1}+  c_{n-1}-\Gamma  I_{n-1})^{-1}d_{n-1})\det(a_{n-1}+  b_{n-1}+  c_{n-1}-\Gamma  I_{n-1})$ \\
$=\det(2 a_{n-1}-(\Gamma -1)I_{n-1})\det(a_{n-1}+  b_{n-1}+  c_{n-1}-\Gamma  I_{n-1}- d_{n-1}(2 a_{n-1}-(\Gamma -1)I)^{-1})$
\end{center} 
Nevertheless, we turn to the numerical approximation of the Laplacian spectrum on the limit Schreier graph $\hat{\Gamma}(G, S, A^{-\mathbb{N}})$. The following proposition that is proved by \cite{grigorchuk1999asymptotic} gives an insight into the Laplacian spectrum of limit Schreier graph $\hat{\Gamma}$, a Hausdorff limit of the family of finite Schreier graphs $\lbrace\Gamma_{n}(G, S, A^{n})\rbrace$.

\begin{prop}\label{prop1}\cite{grigorchuk1999asymptotic}
Let $\lbrace(X_{n}, v_{n})\rbrace$ be a sequence of marked graphs that converges to the marked graph $(X, v)$. Then, for every $\Gamma\in spec(\Delta_{X})$ in the spectrum of the random walk’s operator on $X$, and every $\epsilon> 0$, there is an integer $N$ such that for every $n> N$, there is a spectral value $\Gamma_{n}$ of
$\Delta_{X_{n}}$ which is in the interval.
$(\Gamma-\epsilon , \Gamma+\epsilon)$.
\end{prop}

The proposition above ensures that the spectrum of the Schreier graph on each level accurately approximates the spectrum of the Schreier graph $\hat{\Gamma}$ on the boundary, denoted as $spec(\hat{\Delta})=\bigcup_{n>0} spec(\Delta_{n})$. In accordance with Proposition \ref{prop1}, we have approximated the numerical spectrum of the Laplacian operator $\Delta_n$ for specific values of $n$, specifically, $n=7, 9, 12$ (refer to Figures \ref{The Schreier graph 7}, \ref{The Schreier graph 9}, \ref{The Schreier graph 12}).

Let $\Gamma(M_{n})$ denote the spectral radius of $M_{n}$. It is easy to see that:
\begin{center}
$\lim_{i\rightarrow\infty}\sqrt[2i]{p^{2 i}(x , x)}=\Gamma(M_{n}) = \Vert M_{n}\Vert_{l^{2}(A^{n}, deg)}$
\end{center}
For any vertex $x \in A^{n}$, where $\sqrt[2i]{p^{2 i}(x , x)}$ represents the probability of transitioning from $x$ to $x$ in $2i$-steps. Specifically, $\Vert M_{n}\Vert\leq 1$, and the value $\Gamma(M_{n})$ lies within the spectrum of $M_{n}$, i.e., $[-1, 1]$. Since $A^{n}$ is a finite set, then $\Vert M_{n}\Vert =1$, and $1$ is the largest eigenvalue of $M_{n}$. For a finite connected graph $\Gamma_{n}$, we denote the second eigenvalue after 1, arranged in a natural decreasing order of points in the spectrum of $M_{n}$, by $\Gamma(M_{n})$.
Since the Markov operator is  self-adjoint and bounded operator,
by applying the spectral theorem one has the spectral decomposition $M_{n}=\int_{-1}^{1}\Gamma d E_{\Gamma_{n}}(\Gamma)$, where $E_{n}$ is the spectral measure associated with $M_{\Gamma_{n}}$. 
\begin{center}
$\mu_{x_{n}, y_{n}}^{\Gamma_{n}}(B)=\langle E_{\Gamma_{n}}(B) \delta_{x_{n}} , \delta_{y_{n}} \rangle$
\end{center}
Where $\delta_{x_{n}}$ is Dirac delta function and $x_{n},y_{n}\in A^{n}$.
R.Grigorchuk \cite{grigorchuk1999asymptotic} showed the spectral measure $\mu_{x_{n}, y_{n}}^{\Gamma_{n}}$ weakly converges to spectral measure $\mu_{x, y}^{\hat{\Gamma}}$, i.e., 
\begin{center}
$\lim_{n\rightarrow\infty}\bigints_{-1}^{1} f \mu_{x_{n}, y_{n}}^{\Gamma_{n}}=\bigints_{-1}^{1} f \mu_{x, y}^{\hat{\Gamma}}$
\end{center}
for any $f\in C[-1 ,1]$. Consequently, the spectrum of any bounded operator $T_{n}\in \mathcal{B}( l^{2}(A^{n}))$ on $n$-th level Schreier graphs weakly converges to the spectrum of a bounded operator on limit Schreier graph $\tilde{T} \in \mathcal{B}( L^{2}(A^{-\mathbb{N}}, \nu)))$;
\begin{center}
 $\sigma(T_{n})\rightarrow \sigma(\tilde{T})$
\end{center}

  \begin{figure}[h!]
  \begin{subfigure}[a]{0.3\linewidth}
  \centering
    \includegraphics[width=1\textwidth]{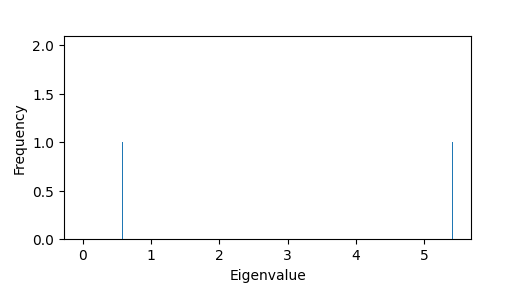}
 
  \end{subfigure}
  \begin{subfigure}[a]{0.3\linewidth}
  \centering
    \includegraphics[width=1\textwidth]{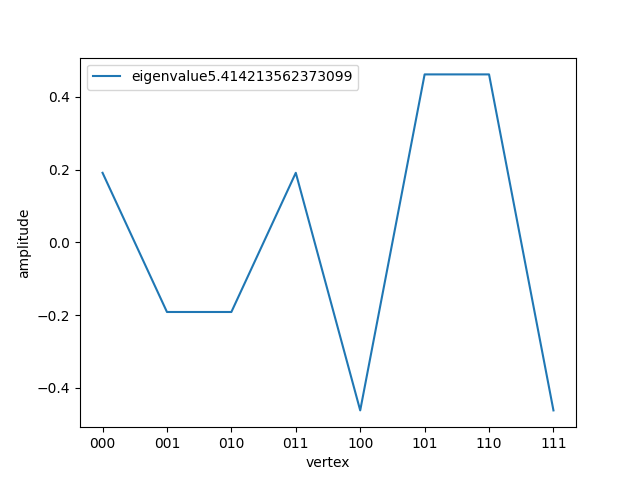}
    \caption{level 3}
  \end{subfigure}
\begin{subfigure}[b]{0.35\linewidth}\centering
    \includegraphics[width=1\textwidth]{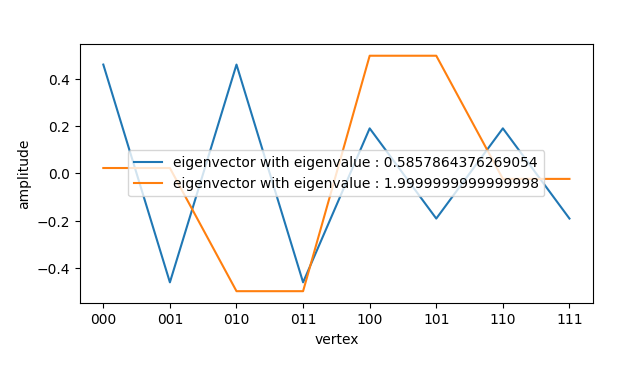}

  \end{subfigure}
    \begin{subfigure}[a]{0.3\linewidth}
  \centering
    \includegraphics[width=1\textwidth]{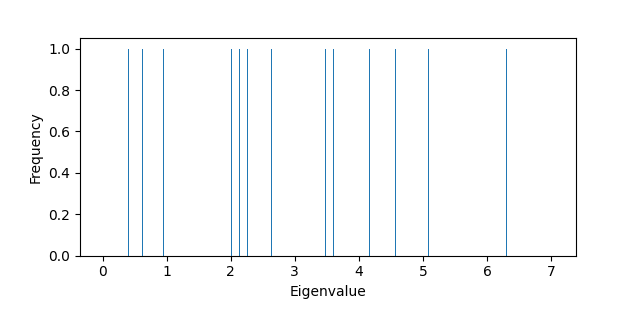}

  \end{subfigure}
  \begin{subfigure}[a]{0.3\linewidth}
  \centering
    \includegraphics[width=1\textwidth]{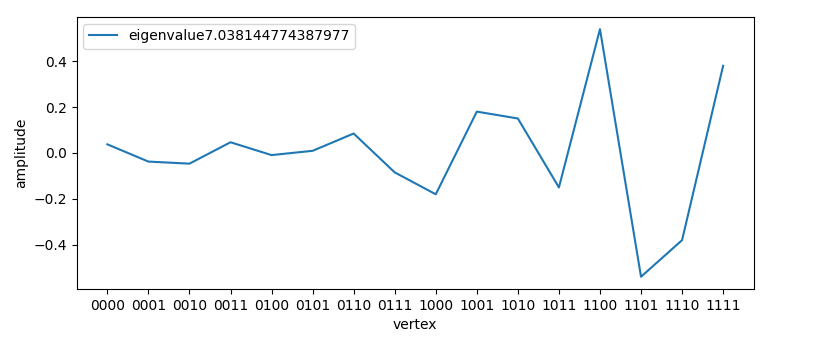}
    \caption{level 4}
  \end{subfigure}
\begin{subfigure}[b]{0.3\linewidth}\centering
    \includegraphics[width=1\textwidth]{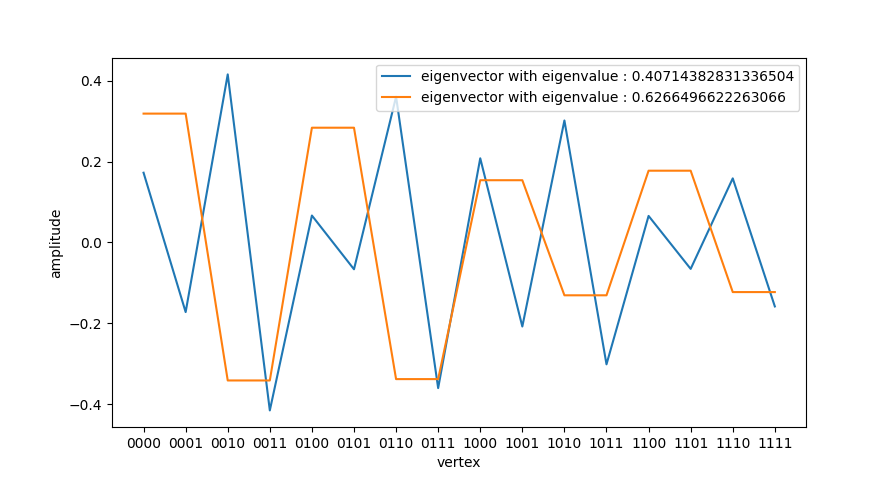}
  \end{subfigure}
    \begin{subfigure}[a]{0.3\linewidth}
  \centering
    \includegraphics[width=1\textwidth]{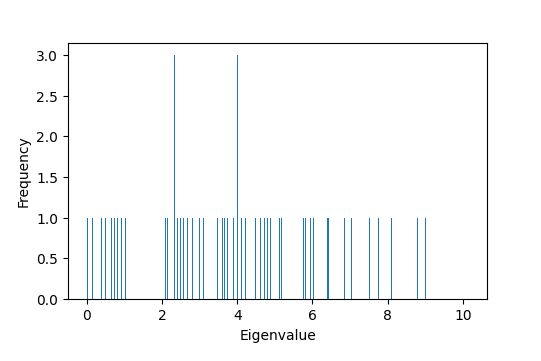}
  \end{subfigure}
    \begin{subfigure}[a]{0.3\linewidth}
  \centering
    \includegraphics[width=1\textwidth]{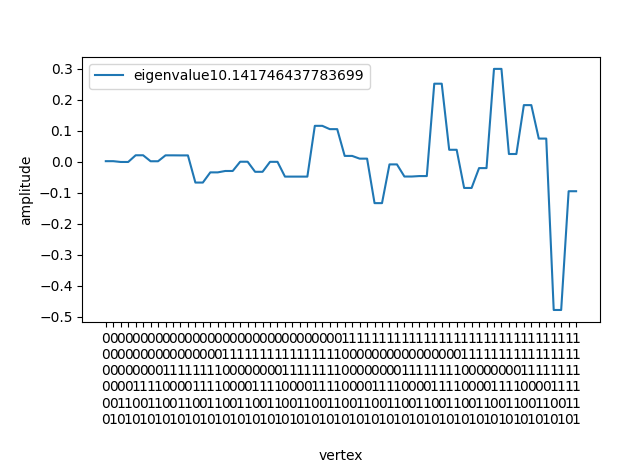}
    \caption{level 6}
  \end{subfigure}
\begin{subfigure}[b]{0.3\linewidth}\centering
    \includegraphics[width=1.3\textwidth]{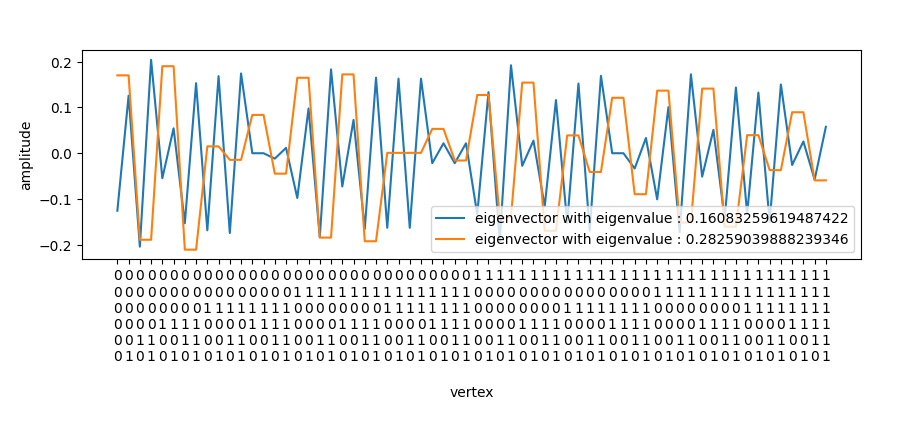}
     
  \end{subfigure}
  \caption
{From left to right, the Laplacian spectrum, the eigenvector of the largest eigenvalue, and eigenvectors for the two smallest eigenvalues at three levels 3,4,6.}
  \label{smallest eigenvalues}
\end{figure}

\begin{figure}
 \centering
    \includegraphics[width=0.5\linewidth]{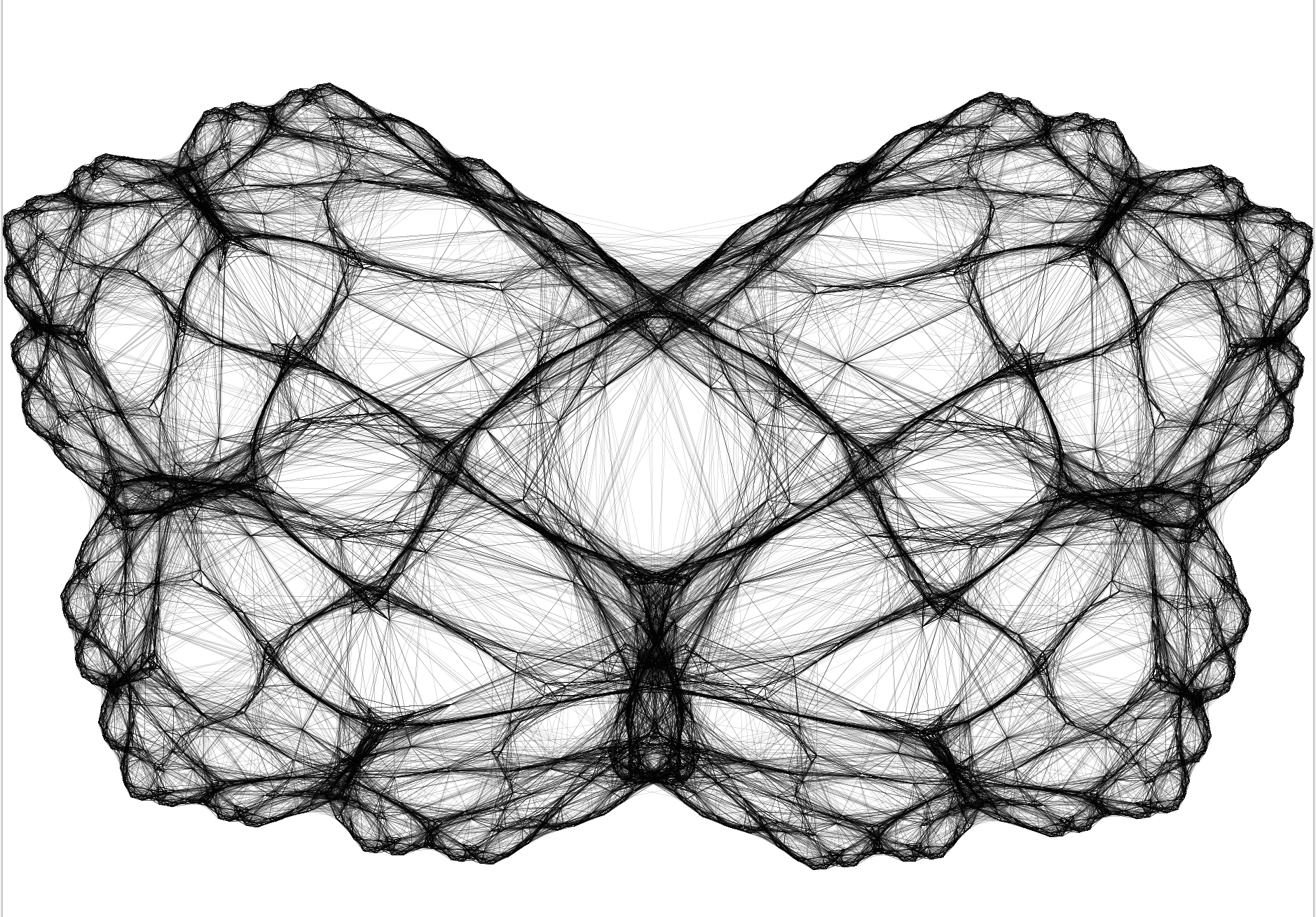}
     \caption{The simplicial Schreier graph of group $G$ is generated by automaton \ref{Moore diagram} at level 12.}
  \label{approximation}
\end{figure}

\section{The group computation} \label{The group computation}

In this section, some details of the main group (defined by automata \ref{Moore diagram}) is demostrated. by finding elements belongs to stabilizer, rigid-stabilizer.
Let consider ball $B(e, r)$ of radius $r$ centered on the group’s identity element, by observing their action on different levels of the binary rooted tree $A^{*}$. The ball $B(1, r)$ is preserved under the virtual endomorphism that determines the self-similarity of the group. We have considered the group’s elements with $r\leq 12$. In addition, the portrait of some elements of stabilizer subgroup  is presented.
\begin{center}
$\phi_{n}: G\rightarrow Sym(A^{n})\wr G$\\
$\phi_{n}(g)=\sigma (\phi_{1}(g), \cdots  \phi_{d^{n}}(g))$, for $\sigma\in Sym(A^{n})$,
\end{center}

\begin{nota}
Despite some results of our group's calculation can also reached by use of GAP program's Automgrp package \cite{muntyan2008automgrp}
, nevertheless, all computation of the automaton group $G$ in peresent papre, involving Schreier graphs and its spectra, Moore graphs, elements of subgroups and their portrates, and the group's relations,  is executed by the the calculator designed by the first author. The codes of the calculator is written in Python programing language\footnote{for the purpse of relibality some results is compared with output of GAP program}.
\end{nota}

\subsection{subgroups}
\subsubsection{Level-stablizer subgroups}
By definition of the main group $G:=\langle a^{\pm 1}, b^{\pm 1}, c^{\pm 1}, d^{\pm 1} \rangle$ it is easy to see  that element $b,c,d$ fix the first level, and $d$ is an element of the second level stablizer $d\in St(2)$

\begin{itemize}
\item
$Stab_{G}(1$-th$)=\langle b, c, d, a^{2},ac^{\pm 1}a, a b^{\pm 1}a , a d^{\pm 1}a  \rangle^{G}$\\
\item
$Stab_{G}(2$-th$)= \langle d$ , $  a^{2}$ , $ b^{2} $ , $ c^{2}$ ,  $ ad^{\pm 1}a$ , $ bd^{\pm 1}b $ , $ cd^{\pm 1}c$ , $ab^{\pm 2}a$ ,
$ cb^{\pm 2}c$ , $ba^{\pm 2}b$ ,$ ca^{\pm 2}c$ , $ad^{\pm2}a$ ,  $ cd^{\pm 2}c$  ,  $ bc^{\pm 2}b$ ,    $bd^{\pm 2}b$ ,
$ [a , c^{\pm 1}]$ , $[a , d^{\pm 1}]  $ ,  $[ b , c^{\pm 1}]$ , 
$[b , d^{\pm 1}] $, $[c , d^{\pm 1}] \rangle^{G} $
\item 
$Stab_{G}(3$-th$)=\langle a^{2}$ , $ b^{2} $ , $ c^{2}$ , $ d^{2}$ ,  $ (ac^{\pm 1})^{2}$ , $ (bc^{\pm 1})^{2}$ , $ (bd^{\pm 1})^{2}$ ,\\
 $ [a , c^{\pm}]$ ,
$[b , c^{\pm 1}]$ , $ [b , d^{\pm 1}] \rangle^{G}$
\end{itemize}

Due to the wreath recurssion $\phi_{n} :G\rightarrow Sym(A^{n})\wr G$ , $g\mapsto (\sigma_{n}(g), ( g\vert_{x})_{x\in A^{n}})$, $\sigma_{n}(g)\in Sym(A^{n})$, if the element
$g\in G$ belongs to $n$-th-level-stablizer subgroup $g$ acts by trivial permutation on the level $n$-th and can be written by $(1 , ( g\vert_{x})_{x\in A^{n}})$, in addition, $g\in Stab(n+1)$ if and only if $g\vert_{x}$  act trivially on the alphabet set $A$, i.e., $g\vert_{x}\in Stab(1)$, for all $x\in A^{n}$. According to the inherent structure of the automata $\Pi=(G , A)$ the condition satisfied when $g\in \langle b, c, d, a^{2},ac^{\pm 1}a, a b^{\pm 1}a , a d^{\pm 1}a  \rangle^{G}$.
In squle by induction, we can see that :
\begin{flushleft}

$\phi_{2} (d^{2^{n}}) = (1, (1, 1, a^{2^{n}}, c^{2^{n}}))$\\
$\phi_{3} (a^{2^{n}})= (1, (1, 1, a^{2^{n-1}}, c^{2^{n-1}}, 1, 1, a^{2^{n-1}}, c^{2^{n-1}}))$\\

$\phi_{4}(c^{2^{n}})= (1, (1, 1, a^{2^{n-1}}, c^{2^{n-1}}, 1, 1, a^{2^{n-1}}, c^{2^{n-1}}, 1, 1, a^{2^{n-1}}, c^{2^{n-1}}, 1, 1, a^{2^{n-1}}, c^{2^{n-1}}))$\\
$\phi_{4} (b^{2^{n}}) = (1, (1, 1, a^{2^{n-1}}, c^{2^{n-1}}, 1, 1, a^{2^{n-1}}, c^{2^{n-1}}, 1, b^{2^{n-1}}, 1, b^{2^{n-1}}, 1, b^{2^{n-1}}, 1, b^{2^{n-1}}))$
\end{flushleft}
Therefore elements $a^{2^{n}}, b^{2^{n}}, c^{2^{n}}, d^{2^{n}}$, generate stablizer-subgroups $Stab(3n + k)$, for $k= 0, 1, 2$ respectivelly. Pricisely, for $n\geq 1$ one have :  
\begin{align}\label{eq-stablizer}
a^{2^{n}}\in Stab(3 n )\setminus Stab(3n+1) , \\
c^{2^{n}} , \quad b^{2^{n}}\in Stab(3 n+1 )\setminus Stab(3n+2) , \\
d^{2^{n}}\in Stab(3 n + 2 )\setminus Stab(3(n+1)) , 
\end{align}

\subsubsection{Rigid-stablizer subgroups}
\begin{itemize}
\item $Rigid(0)= \langle d^{a^{\pm 1}}\rangle^{Stab(1)}$ 
%\begin{flushleft}
%$\phi (ada^{-1}) = (1, (b, 1))$, \qquad
%$\phi (a^{-1}da) =(1, (b^{d^{-1}} , 1))$.
%\end{flushleft}
\item $Rigid(1)=\langle d, c^{-1}b, cb^{-1}, bdc^{-1}, bc^{-1}d, dc^{-1}b, c^{-1}db, cb^{-1}d^{-1}, bd^{-1}b^{-1},b^{-1}d^{-1}c, cdb^{-1}, c^{-1}d^{-1}b \rangle^{Stab(1)}$
%\begin{flushleft}
%$\phi (bdc^{-1})=(1 ,(1, cba^{-1}))$  , \qquad
%$\phi (bc^{-1}d)=(1 ,(1, ca^{-1}b))$  , \qquad\;
%$\phi (dc^{-1}b)= (1 ,(1, ba^{-1}c))$  ,\\
%$\phi (c^{-1}db)= (1 ,(1, a^{-1}bc)) $ , \qquad
%$\phi (cb^{-1}d^{-1})= (1 ,(1, ac^{-1}b^{-1})) $  , \;
%$\phi (bd^{-1}b^{-1})= (1, (1, cb^{-1}c^{-1}))  $ ,\\ 
%$\phi (bc^{-1}d^{-1})= (1, (1, ca^{-1}b^{-1})) $ , 
%$\phi (b^{-1}d^{-1}c)= (1, (1, c^{-1}b^{-1}a))  $ ,\quad
%$\phi (c^{-1}d^{-1}b)= (1, (1, a^{-1}b^{-1}c)) $, \\
%$\phi (cdb^{-1})= (1, (1, abc^{-1}))$ .
%\end{flushleft}

\end{itemize}

%\begin{center}
%$[b^{-1} , c^{-1}]=( 1 , 1, [a^{-1}, d^{-1}], 1 )$
%$[c, b]=( 1 , 1, 1, [d , a] )$
%$[b^{-1}, d]=( 1 , 1, 1, [a^{-1} , c] )$
%\end{center}
\begin{itemize}
\item 
$Rigid(00)= \langle [ac , ac^{-1}], [ac , c^{2}], [ac^{-1} , c^{2}] \rangle^{Stab(2)}$
%\begin{center}
%$\phi([ac , ac^{-1}]= [a, c^{-1}] [a, c]^{c^{-1}})=  (1, ([d , b], 1, 1, 1))$ \\
%$\phi([ac , c^{2}]= [a, c^{a}]^{c^{2}})= \phi([ac^{-1} , c^{2}]=[[a, c], c^{-1}]^{2})= (1, ([b , d], 1, 1, 1))$ \\
%\end{center}

\item $Rigid(01)=\langle [a, c^{2}], [a^{2} , da], [ad , a^{-1}d], [ca , c^{2}], [ca , a^{-1}c]\rangle^{Stab(2)}$
%\begin{flushleft}
 %$\phi ([a^{2} , da]) = (1, (1, [b , c]^{b}, 1, 1))$ \\
% $\phi ([a , c^{2}]) = \phi ([ad , a^{-1}d] = [a, d]^{a}[a, d]) =\phi ([ca , a^{-1}c]=[c^{-1} , [a^{-1} , c]][a , c]) $\\
 %$=\phi ([ca , c^{2}]=[a , c][a , c]^{c})=(1, (1, [b , d], 1, 1))$ 
%\end{flushleft}
\item $Rigid(10)=\langle [b , c], [b , bc], [b , cb], [b , dc], [b , b^{-1}c], [c , b^{2}], [c , cb], [b^{2} , bc], [b^{2} , cb], [b^{2} , dc], [b^{2} , b^{-1}c], [c^{2} , a^{-1}c]\rangle^{Stab(2)}$
%$\phi([b , bc]=[b , b^{-1}c]=[cb, c]=[b , c])= (1, [1, 1, [a, d], 1])$\\ 
%$\phi([b , cb]=[b , c]^{b} )= (1, [1, 1, [a, d]^{a}, 1])$\\ 
%$\phi([b , dc]=[b , c][b , d]^{c} )= (1, [1, 1, [a, dc], 1])$\\
%$\phi([b^{\pm 1}c , b^{2}]=[c , b^{2}]=[c , b][c , b]^{b})=(1, [1, 1, [d, a^{2}], 1])$\\
%$\phi([b^{2} , cb]=[b^{2} , c]^{b} )= (1, [1, 1, [a^{2} , d]^{a}, 1])$\\
%$\phi([b^{2} , dc]=[b^{2} , c][b^{2} ,d]^{c} )= (1, [1, 1, [d^{-1} , b]^{d}, 1])$\\
%$\phi([c^{2} , a^{-1}c] )= (1, [1, 1, [d^{-1} , b]^{d}, 1])$\\
%$\phi([b , c^{-1}] )= (1, [1, 1, [a^{-1} , d]^{a}, 1])$\\
%$\phi([b , c^{-1}d]=[b, d][b , c^{-1}]^{d})= (1, [1, 1, [a^{-1} , d][a^{-1} , c^{-1}]^{d}, 1])$\\
%$\phi([b^{2} , bc^{-1}]=[b^{2} , c]^{b})= (1, [1, 1, [a^{-2} , d]^{a^{-1}}, 1])$\\
%$\phi([b^{2} , c^{-1}d] = [b^{2} , d][b^{2} , c^{-1}]^{d})=(1, [1, 1, [a^{-2} , d][a^{-2} , c^{-1}]^{d}, 1])$\\
%$\phi([a , c^{2}])= (1, [1, 1, [b , d^{-1}], 1])$\\

\item $Rigid(11)=\langle
[b , d], [b , bd], [b , bc^{-1}], [b , bd^{-1}], [b , db], [b , d^{2}], [b , b^{-1}d], [b , c^{-1}d], [c , bc], [c , bc^{-1}], [c , b^{-1}c], [d , a^{2}],$\\
$ [d , b^{2}], [d , bd], [d , bd^{-1}], [d , db], [d , b^{-1}d], [a^{2} , ad], [a^{2} , ad^{-1}], [a^{2} , d^{2}], [a^{2} , a^{-1}d], [ad^{-1} , da], [b^{2} , bd],$\\
$ [b^{2} , bc^{-1}], [b^{2} , bd^{-1}], [b^{2} , db], [b^{2} , d^{2}], [b^{2} , b^{-1}d], [b^{2} , c^{-1}d], [bc , b^{-1}c], [bd , bd^{-1}], [bd , db], [bd , d^{2}], [bd , b^{-1}d], $\\
$[bc^{-1} , cb], [bd^{-1} , c^{2}], [bd^{-1} , db], [bd^{-1} , d^{2}], [bd^{-1} , b^{-1}d], [c^{2} , b^{-1}d], [db , d^{2}], [db , b^{-1}d], [d^{2} , b^{-1}d] \rangle^{Stab(2)}$

\end{itemize}

\subsubsection{Portrate of group's elements}
In this part, the portrate $\phi_{n}:G \rightarrow Sym(A^{n})\wr G$ of a collection of level-stablizer-subgroup's element on the corresponding level is presented:

\begin{center}

$b^{-1}c^2b \mapsto (1, (1, 1, a, c, 1, 1, a, c, a, b^{-1}cb, 1, 1, a, b^{-1}cb, 1, 1))$\\
$(c^{4})^{b}\mapsto (1, (1, 1, a^{2}, c^{2}, 1, 1, a^{2}, c^{2}, a^{2}, (c^{2})^{b}, 1, 1, a^{2}, (c^{2})^{b}, 1, 1]))$\\
$(b^{2})^{c}\mapsto (1, (1, 1, a, c, 1, 1, a, c, 1, b, 1, b^{c}, 1, b, 1, b))$\\
$bcbc^{-1}\mapsto (1, ( 1, b, 1, b, d, d, bd, db^{-1}))$ ,
$[c, b^{-1}] \mapsto (1, (1, 1, 1, 1, 1, 1, d^{-1}b^{-1}, b))$ , \\
$[ac , ac^{-1}] = [a, c^{-1}]^{c} [c, a]^{c^{-1}} \mapsto (1, (1, 1, 1, [c, a], 1, 1, 1, 1, 1, 1, 1, 1, 1, 1, 1, 1))$\\
$[ac , c^{2}] =  [a, c^{2}]^{c} \mapsto (1, (1, 1, 1,  [a, c]^{c}, 1, 1, 1, 1, 1, 1, 1, 1, 1, 1, 1, 1))$\\
$[a^{2} , da] \mapsto (1, (1, 1, 1, 1, 1, 1, [a, d]^{a}, 1, 1, 1, 1, 1, 1, 1, 1, 1))$\\
$[c , b^{2}] \mapsto (1, (1, 1, 1, 1, 1, 1, 1, 1, 1, 1, 1, [c, b], 1, 1, 1, 1))$\\
$[b , d^{2}] \mapsto (1, (1, 1, 1, 1, 1, 1, 1, 1, 1, 1, 1, 1, 1, [b, d], 1, 1))$\\
$[d , a^{2}] \mapsto (1, (1, 1, 1, 1, 1, 1, 1, 1, 1, 1, 1, 1, 1, 1, [d, a], 1))$\\
$(b^{-1}c)^{2}\mapsto (1, ( 1, 1, 1, 1, d^{-1}, d^{-1}b, d^{-1}b, d^{-1}))$ ,\\
$a^{2}\mapsto (1, (1, 1, a, c, 1, 1, a, c))$ , 
$(db^{-1})^{2}\mapsto (1, (1, b^{-1}, 1, b^{-1}, 1, 1, ad^{-1}a, a^{2}d^{-1}))$ , \\
$ba^{2}b^{-1}\mapsto (1, (a, c, 1, 1, 1, 1, c, dad^{-1}))$ ,\\
$db^{-1}db\mapsto(1, (1, 1, 1, 1, d, d, ad^{-1}ad, a^{2}))$ , \\
$[b^{\pm 1}, d^{\pm 1} ] \mapsto (1, ( 1, 1, 1, 1, 1, 1, [d^{\pm 1}, a^{\pm 1}], 1))$ , \\
$b^{2}\mapsto (1, ( 1, 1, a, c, 1, 1, a, c, 1, b, 1, b, 1, b, 1, b))$\\
$(c^{2})^{b}\mapsto (1, ( 1, 1, a, c, 1, 1, a, c, a, c, 1, 1, a, c, 1, 1))$ , \\
$c^{2}\mapsto(1, (1, 1, a, c, 1, 1, a, c, 1, 1, a, c, 1, 1, a, c))$ ,\\
$b^{2}\mapsto (1, (1, 1, a, c, 1, 1, a, c, c^{2}, a^{2}, d, d, 1, b, 1, b))$ , \\
$(c^{2})^{a^{-1}} \mapsto(1, ( 1 ,1 ,a, c, 1, 1, a, c, c^{2}, a^{2}, d, d, b , 1, 1, b^{a^{-1}}))$ , \\
$(b^{2})^{a^{-1}} \mapsto (1, ( c^{2}, a^{2}, d, d, 1, b, 1, b, 1, 1, a, c, 1, 1, a^{c^{-1}}, c^{a^{-1}}))$ ,\\
$d^{2}\mapsto (1, (1, 1, 1, 1, 1, 1, 1, 1, 1, 1, 1, 1, 1, 1, 1, 1, 1, 1, a, c, 1, 1, a, c, c^{2}, a^{2}, d, d, 1, b, 1, b))$ , \\
$a^{4}\mapsto (1, (1, 1, 1, 1, 1, 1, 1, 1, 1, 1, 1, 1, 1, 1, 1, 1, 1, 1, a, c, 1, 1, a, c, c^{2}, a^{2}, d, d, 1, b, 1, b, 1, 1$\\$, 1, 1, 1, 1, 1, 1, 1, 1, 1, 1, 1, 1, 1, 1, 1, 1, a, c, 1, 1, a, c, c^{2}, a^{2} , d, d, 1, b, 1, b))$ ,\\
$[c^{\pm 2}, a^{-1}]\mapsto(1, (1, 1, 1, 1, 1, 1, 1, 1, 1, 1, 1, 1, 1, 1, [a^{\pm}, d^{-1}], 1, 1, 1, 1, 1, 1, 1, 1, 1, 1, 1, 1, 1, 1, 1, 1, 1))$\\
\end{center}

\subsection{relations}

Finally, the groups relation up to length 12 is the following:

\begin{flushleft}

$[d , d^{a^{\pm 1}}]$,
$[a^{2} , a^{c^{\pm1}}]$, 
$[bc^{-1} , d^{a^{\pm1}}]$, 
$[c^{2} , c^{b^{\pm1}}]$,
$[d^{2} , d^{a^{\pm1}}]$,  
$[b^{-1}c , d^{a^{\pm1}}]$, 
$[d^{a^{-1}} , d^{b^{\pm1}}]$,
$[d^{a^{-1}} , bdc^{-1}]$, 
$[d^{a^{-1}} , bc^{-1}d^{\pm 1}]$, 
$[d^{a^{-1}} , bd^{-1}c^{\pm1}]$, 
$[d^{a^{-1}} , d^{c^{\pm1}}]$,
$[d^{a^{-1}} , dc^{-1}b]$,
$[d^{a^{-1}} , b^{-1}cd]$, 
$[d^{a^{-1}} , b^{-1}d^{\pm1}c]$, 
$[d^{a^{-1}} , c^{-1}bd]$, 
$[d^{b^{\pm 1}} , d^{a^{\mp 1}}]$,
$[bdc^{-1} , d^{a}]$, 
$[bd^{-1}c^{-1} , d^{a}]$, 
$[d^{c^{\pm 1}} , d^{a}]$,
$[cb^{-1}d^{\pm 1} , d^{a}]$, 
$[db^{-1}c , d^{a}]$,
$[dc^{-1}b , d^{a}]$, 
$[d^{a} , d^{b}]$, 
$[d^{a} , b^{-1}d^{\pm1}c]$, 
 $[d^{a^{-1}} , bd^{\pm 1}c^{-1}]$, 
$[d^{a^{-1}} , cb^{-1}d^{\pm 1}]$, 
$[d^{a^{\pm 1}} , d^{3}]$, 
$[d^{a^{-1}} , db^{-1}c]$, 
$[d^{a^{\pm 1}} , b^{-1}cd]$,
$[d^{a^{\pm 1}} , b^{-1}dc]$,
$[d^{a^{\pm 1}} , d^{c}]$, 
$[bc^{-1}d^{\pm 1} , d^{a}]$,
$[cb^{-1}d^{\pm 1} , d^{a}]$,
$[d^{a^{\pm 1}} , c^{-1}bd]$,  
$[bc^{-1}d^{\pm 1} , d^{a}]$, 
$[d^{a} , b^{-1}d^{\pm 1}c]$, 
$[d^{a} , b^{-1}dc]$, 

\end{flushleft}

Where  $x^{y}=y^{-1} x y$ and $[x, y] = x^{-1}y^{-1}xy$.\\

\section{Appendix}\label{appendix}
In the continuation of the proof of Lemma \ref{contracting lemma}, we have to verfying the followin inclusion; for some positive integer $n_{0}$:
\begin{center}
$((S\cup S^{-1}) \otimes \mathcal{N})\vert_{A^{n}} \subseteq \mathcal{N}$, for any $n\geq n_{0}$ 
\end{center}
By calculating the image of group's element  $a x$ under virtual-endomorphism $\phi$, for each pair of group's
 elements $(a, x)$, where  $a\in  S$ and $x\in\mathcal{N}$ (is defined  \ref{nucleus}),
We observed that for $n_{0}=7$ the above inclusion (\ref{inclusion}) is hold, i.e., $a\otimes x\vert_{A^{7}}\in \mathcal{N}$.

\begin{wrapfigure}{r}{0.5\textwidth}
  \begin{center}
    \includegraphics[width=0.5\textwidth]{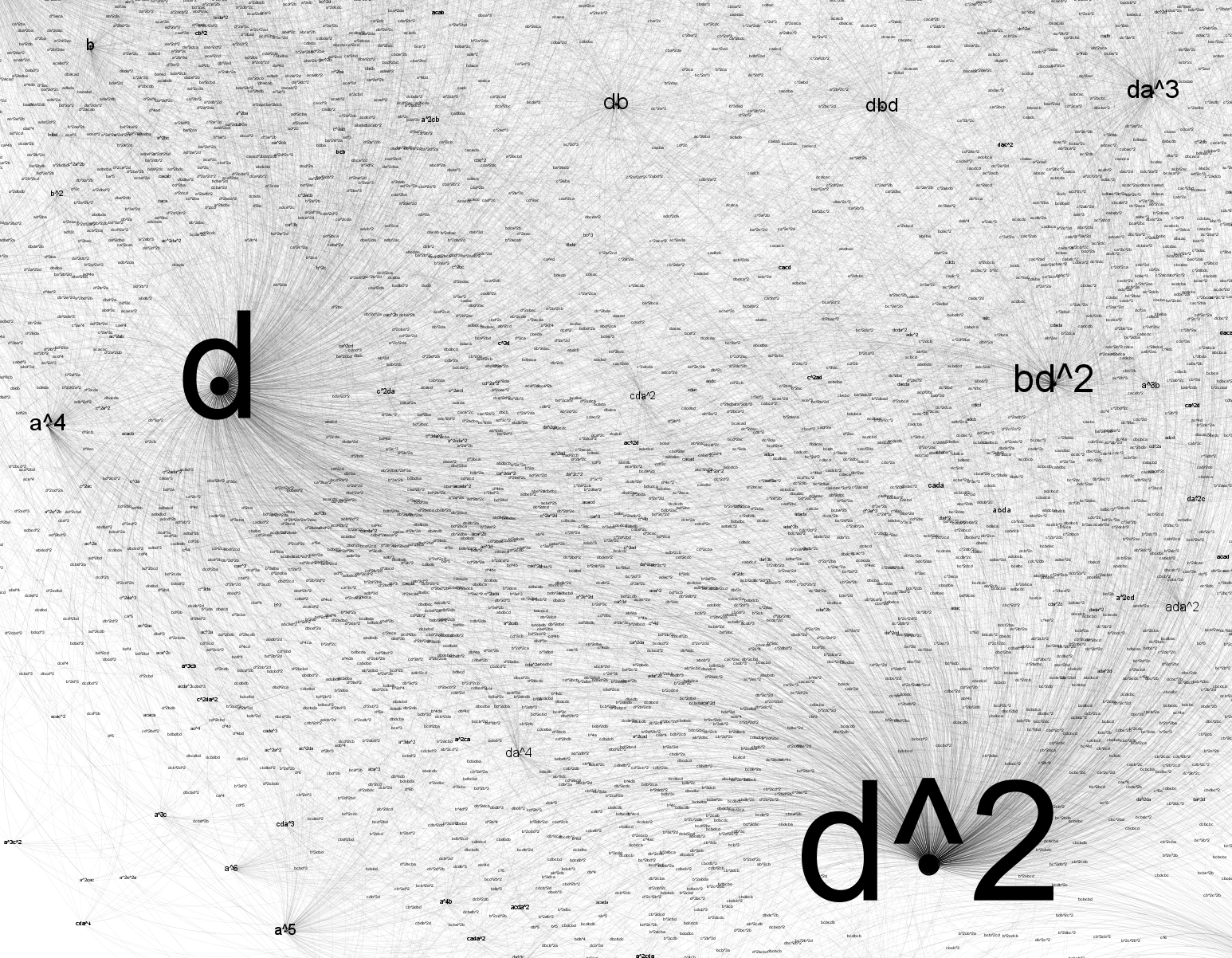}
  \end{center}
  \caption{ A portion of the Moore diagram of the automaton consists elements with length almost 4 on the level 3-th, $\tau : Q^{4}\times A^{3} \rightarrow  Q^{4}$}\label{The Moore graph of level3}
\end{wrapfigure}

$( a \otimes b )\vert_{A^{7}} = \lbrace 1, d, a, c, cb, b\rbrace$ , \\ 
$( a \otimes b^{-1} )\vert_{A^{7}} = \lbrace 1, d, a, c, a^{-1}, c^{-1}, b^{-1}, d^{-1}\rbrace$ , \\ 
$( a \otimes d )\vert_{A^{7}} = \lbrace 1, d, a, c, b\rbrace$ , \\ 
$( a \otimes d^{-1} )\vert_{A^{7}} = \lbrace 1, d, a, c, b^{-1}, d^{-1}, a^{-1}\rbrace$ , \\ 
$( a \otimes c^{-1} )\vert_{A^{7}} = \lbrace 1, d, a, c, a^{-1}, c^{-1}, b^{-1}\rbrace$ , \\ 
$( a \otimes da )\vert_{A^{7}} = \lbrace 1, b, d, a, a^{2}, ac, c\rbrace$ , \\ 
$( a \otimes bd )\vert_{A^{7}} = \lbrace 1, d, a, c, cb, b, da\rbrace$ , \\ 
$( a \otimes cb )\vert_{A^{7}} = \lbrace 1, a, c, b, d, cb, bd\rbrace$ , \\ 
$( a \otimes ac )\vert_{A^{7}} = \lbrace 1, d, a, c, cb, b\rbrace$ , \\ 
$( a \otimes d^{-1}b^{-1} )\vert_{A^{7}} = \lbrace 1, d, a, c, a^{-1}, c^{-1}, b^{-1}, d^{-1}, a^{-1}d^{-1}\rbrace$ , \\ 
$( a \otimes b^{-1}c^{-1} )\vert_{A^{7}} = \lbrace 1, a^{-1}, c^{-1}, b^{-1}, d, a, c, d^{-1}, d^{-1}b^{-1}\rbrace$ , \\ 
$( a \otimes c^{-1}a^{-1} )\vert_{A^{7}} = \lbrace 1, a^{-1}, c^{-1}, b^{-1}, d, a, c, d^{-1}, b^{-1}c^{-1}\rbrace$ , \\ 
$( a \otimes d^{-1}c )\vert_{A^{7}} = \lbrace 1, d, a, c, cb, b, d^{-1}, a^{-1}, a^{-1}b, b^{-1}\rbrace$ , \\ 
$( a \otimes b^{-1}a )\vert_{A^{7}} = \lbrace 1, b^{-1}, d^{-1}, d, a, c, d^{-1}c, a^{-1}, c^{-1}\rbrace$ , \\ 
$( a \otimes c^{-1}d )\vert_{A^{7}} = \lbrace 1, d, a, c, a^{-1}, c^{-1}, b^{-1}, b, b^{-1}a\rbrace$ , \\ 
$( a \otimes d^{-1}a )\vert_{A^{7}} = \lbrace 1, b^{-1}, d^{-1}, a, c, a^{-1}c, d\rbrace$ , \\ 
$( a \otimes b^{-1}d )\vert_{A^{7}} = \lbrace 1, d, a, c, a^{-1}, c^{-1}, b^{-1}, d^{-1}a\rbrace$ , \\ 
$( a \otimes c^{-1}b )\vert_{A^{7}} = \lbrace 1, d, a, c, b, a^{-1}, c^{-1}, b^{-1}d\rbrace$ , \\ 
$( a \otimes c^{-1}a )\vert_{A^{7}} = \lbrace 1, d, a, b^{-1}c, b^{-1}, c, a^{-1}, c^{-1}\rbrace$ , \\ 
$( a \otimes d^{-1}b )\vert_{A^{7}} = \lbrace 1, d, a, c, cb, b, a^{-1}d, a^{-1}, d^{-1}\rbrace$ , \\ 
$( a \otimes b^{-1}c )\vert_{A^{7}} = \lbrace 1, d, a, c, d^{-1}b, d^{-1}, b^{-1}\rbrace$ , \\ 
$( a \otimes d^{-1}bd )\vert_{A^{7}} = \lbrace 1, d, a, c, cb, b, a^{-1}da\rbrace$ , \\ 
$( a \otimes b^{-1}cb )\vert_{A^{7}} = \lbrace 1, d, a, c, cb, b, d^{-1}bd\rbrace$ , \\ 
$( a \otimes c^{-1}ac )\vert_{A^{7}} = \lbrace 1, d, a, c, b^{-1}cb\rbrace$ , \\ 
$( a \otimes d^{-1}b^{-1}d )\vert_{A^{7}} = \lbrace 1, d, a, c, a^{-1}, c^{-1}, b^{-1}, d^{-1}, a^{-1}d^{-1}a\rbrace$ , \\ 
$( a \otimes b^{-1}c^{-1}b )\vert_{A^{7}} = \lbrace 1, d, a, c, a^{-1}, c^{-1}, b^{-1}, d^{-1}b^{-1}d\rbrace$ , \\ 
$( a \otimes c^{-1}a^{-1}c )\vert_{A^{7}} = \lbrace 1, d^{-1}, a^{-1}, c^{-1}, b^{-1}c^{-1}b, d, a, c\rbrace$ , \\ 
$( a \otimes d^{-1}cb )\vert_{A^{7}} = \lbrace 1, a, c, b, d, cb, a^{-1}bd, a^{-1}, d^{-1}\rbrace$ , \\ 
$( a \otimes b^{-1}ac )\vert_{A^{7}} = \lbrace 1, d, a, c, d^{-1}cb, d^{-1}, b^{-1}\rbrace$ , \\ 
$( a \otimes c^{-1}da )\vert_{A^{7}} = \lbrace 1, b, d, a, a^{2}, b^{-1}ac, b^{-1}, c, a^{-1}, c^{-1}\rbrace$ , \\ 
$( a \otimes d^{-1}b^{-1}a )\vert_{A^{7}} = \lbrace 1, b^{-1}, d^{-1}, a, c, a^{-1}d^{-1}c, d, a^{-1}, c^{-1}\rbrace$ , \\ 
$( a \otimes b^{-1}c^{-1}d )\vert_{A^{7}} = \lbrace 1, a^{-1}, c^{-1}, b^{-1}, d, a, c, d^{-1}b^{-1}a\rbrace$ , \\ 
$( a \otimes c^{-1}a^{-1}b )\vert_{A^{7}} = \lbrace 1, a, c, cb, d, a^{-1}, c^{-1}, b^{-1}c^{-1}d, d^{-1}\rbrace$ , \\ 
$( a \otimes d^{-1}ac )\vert_{A^{7}} = \lbrace 1, a, c, a^{-1}cb, d^{-1}, b, b^{-1}, d, cb\rbrace$ , \\ 
$( a \otimes b^{-1}da )\vert_{A^{7}} = \lbrace 1, d, a, a^{2}, d^{-1}ac, b^{-1}, c, a^{-1}, c^{-1}\rbrace$ , \\ 
$( a \otimes c^{-1}bd )\vert_{A^{7}} = \lbrace 1, d, a, c, b, b^{-1}da, a^{-1}, c^{-1}\rbrace$ , \\ 
$( a \otimes b^{-1}c^{-1}a )\vert_{A^{7}} = \lbrace 1, b^{-1}, d^{-1}, d, a, d^{-1}b^{-1}c, a^{-1}, c^{-1}, c\rbrace$ , \\ 
$( a \otimes c^{-1}a^{-1}d )\vert_{A^{7}} = \lbrace 1, a^{-1}, c^{-1}, b^{-1}, a, c, cb, d, b^{-1}c^{-1}a\rbrace$ , \\ 
$( a \otimes d^{-1}b^{-1}c )\vert_{A^{7}} = \lbrace 1, d, a, c, d^{-1}, a^{-1}, a^{-1}d^{-1}b, b^{-1}\rbrace$ , \\ 
$( a \otimes d^{-1}b^{-1}cb )\vert_{A^{7}} = \lbrace 1, d, a, c, cb, b, a^{-1}d^{-1}bd, a^{-1}, d^{-1}, b^{-1}\rbrace$ , \\ 
$( a \otimes b^{-1}c^{-1}ac )\vert_{A^{7}} = \lbrace 1, d, a, c, d^{-1}b^{-1}cb, d^{-1}, b^{-1}, a^{-1}, c^{-1}\rbrace$ , \\ 
$( a \otimes c^{-1}a^{-1}da )\vert_{A^{7}} = \lbrace 1, a, c, cb, d, b^{-1}c^{-1}ac, b^{-1}, a^{-1}, c^{-1}\rbrace$ , \\ 
$( a \otimes d^{-1}b^{-1}da )\vert_{A^{7}} = \lbrace 1, b^{-1}, d^{-1}, d, a, a^{-1}d^{-1}ac, c, a^{-1}, c^{-1}\rbrace$ , \\ 
$( a \otimes b^{-1}c^{-1}bd )\vert_{A^{7}} = \lbrace 1, d, a, c, a^{-1}, c^{-1}, b^{-1}, b, d^{-1}b^{-1}da\rbrace$ , \\ 
$( a \otimes c^{-1}a^{-1}cb )\vert_{A^{7}} = \lbrace 1, a, c, b, a^{-1}, c^{-1}, b^{-1}c^{-1}bd, d, d^{-1}, cb\rbrace$ , \\ 
$( a \otimes d^{-1}b^{-1}ac )\vert_{A^{7}} = \lbrace 1, a, c, a^{-1}d^{-1}cb, d^{-1}, b^{-1}, d\rbrace$ , \\ 
$( a \otimes b^{-1}c^{-1}da )\vert_{A^{7}} = \lbrace 1, d, a, a^{2}, d^{-1}b^{-1}ac, b^{-1}, a^{-1}, c^{-1}, c\rbrace$ , \\ 
$( a \otimes c^{-1}a^{-1}bd )\vert_{A^{7}} = \lbrace 1, d, a, b, c, cb, b^{-1}c^{-1}da, a^{-1}, c^{-1}\rbrace$ , \\ 
$( a \otimes a^{2} )\vert_{A^{7}} = \lbrace 1, d, a, c\rbrace$ , \\ 
$( b \otimes a )\vert_{A^{7}} = \lbrace 1, b, d, da, a, c\rbrace$ , \\ 
$( b \otimes a^{-1} )\vert_{A^{7}} = \lbrace 1, b, d, a, c, d^{-1}, a^{-1}, c^{-1}\rbrace$ , \\ 
$( b \otimes b )\vert_{A^{7}} = \lbrace 1, a, c, b, d\rbrace$ , \\ 
$( b \otimes d^{-1} )\vert_{A^{7}} = \lbrace 1, a, c, b, d, d^{-1}, a^{-1}\rbrace$ , \\ 
$( b \otimes c^{-1} )\vert_{A^{7}} = \lbrace 1, b, d, a^{-1}, c^{-1}, b^{-1}\rbrace$ , \\ 
$( b \otimes c )\vert_{A^{7}} = \lbrace 1, a, c, b, d\rbrace$ , \\ 
$( b \otimes da )\vert_{A^{7}} = \lbrace 1, b, d, a, da, a^{2}, ac, c\rbrace$ , \\ 
$( b \otimes bd )\vert_{A^{7}} = \lbrace 1, a, c, b, d, da\rbrace$ , \\ 
$( b \otimes cb )\vert_{A^{7}} = \lbrace 1, a, c, b, bd, d\rbrace$ , \\ 
$( b \otimes ac )\vert_{A^{7}} = \lbrace 1, b, d, da, a, c, cb\rbrace$ , \\ 
$( b \otimes a^{-1}d^{-1} )\vert_{A^{7}} = \lbrace 1, b, d, a, c, b^{-1}, d^{-1}, a^{-1}, c^{-1}, a^{-2}, c^{-1}a^{-1}\rbrace$ , \\ 
$( b \otimes d^{-1}b^{-1} )\vert_{A^{7}} = \lbrace 1, b^{-1}, d^{-1}, a^{-1}, b, d, a^{-1}d^{-1}\rbrace$ , \\ 
$( b \otimes c^{-1}a^{-1} )\vert_{A^{7}} = \lbrace 1, b, d, a^{-1}, c^{-1}, b^{-1}, d^{-1}\rbrace$ , \\ 
$( b \otimes a^{-1}b )\vert_{A^{7}} = \lbrace 1, b, d, a, c, bd, a^{-1}, c^{-1}, c^{-1}d, d^{-1}\rbrace$ , \\ 
$( b \otimes d^{-1}c )\vert_{A^{7}} = \lbrace 1, a, c, b, d, d^{-1}, a^{-1}, a^{-1}b\rbrace$ , \\ 
$( b \otimes c^{-1}d )\vert_{A^{7}} = \lbrace 1, b, d, da, a, b^{-1}, b^{-1}a, a^{-1}, c^{-1}\rbrace$ , \\ 
$( b \otimes a^{-1}c )\vert_{A^{7}} = \lbrace 1, b, d, a, c, d^{-1}, a^{-1}, c^{-1}b\rbrace$ , \\ 
$( b \otimes d^{-1}a )\vert_{A^{7}} = \lbrace 1, b, d, a, c, d^{-1}, a^{-1}c\rbrace$ , \\ 
$( b \otimes c^{-1}b )\vert_{A^{7}} = \lbrace 1, a, c, b, d, a^{-1}, c^{-1}, b^{-1}d\rbrace$ , \\ 
$( b \otimes c^{-1}a )\vert_{A^{7}} = \lbrace 1, b, d, da, a, b^{-1}c, b^{-1}\rbrace$ , \\ 
$( b \otimes a^{-1}d )\vert_{A^{7}} = \lbrace 1, b, d, a, c, bd, c^{-1}a, a^{-1}, c^{-1}\rbrace$ , \\ 
$( b \otimes d^{-1}b )\vert_{A^{7}} = \lbrace 1, a, c, b, d, a^{-1}d, a^{-1}, d^{-1}\rbrace$ , \\ 
$( b \otimes a^{-1}da )\vert_{A^{7}} = \lbrace 1, a, c, b, bd, d, c^{-1}ac\rbrace$ , \\ 
$( b \otimes d^{-1}bd )\vert_{A^{7}} = \lbrace 1, a, c, b, d, a^{-1}da\rbrace$ , \\ 
$( b \otimes c^{-1}ac )\vert_{A^{7}} = \lbrace 1, b, d, da, a, c, b^{-1}cb\rbrace$ , \\ 
$( b \otimes a^{-1}d^{-1}a )\vert_{A^{7}} = \lbrace 1, a, c, b, b^{-1}, d^{-1}, a^{-1}, c^{-1}a^{-1}c, d\rbrace$ , \\ 
$( b \otimes d^{-1}b^{-1}d )\vert_{A^{7}} = \lbrace 1, b^{-1}, d^{-1}, a^{-1}d^{-1}a, b, d\rbrace$ , \\ 
$( b \otimes c^{-1}a^{-1}c )\vert_{A^{7}} = \lbrace 1, b, d, d^{-1}, a^{-1}, c^{-1}b\rbrace$ , \\ 
$( b \otimes a^{-1}bd )\vert_{A^{7}} = \lbrace 1, b, d, a, c, bd, c^{-1}da, a^{-1}, c^{-1}\rbrace$ , \\ 
$( b \otimes d^{-1}cb )\vert_{A^{7}} = \lbrace 1, a, c, b, a^{-1}bd, a^{-1}, d, d^{-1}\rbrace$ , \\ 
$( b \otimes c^{-1}da )\vert_{A^{7}} = \lbrace 1, b, d, a, da, a^{2}, b^{-1}ac, b^{-1}\rbrace$ , \\ 
$( b \otimes d^{-1}b^{-1}a )\vert_{A^{7}} = \lbrace 1, b^{-1}, d^{-1}, d, da, c, b, a^{-1}d^{-1}c\rbrace$ , \\ 
$( b \otimes c^{-1}a^{-1}b )\vert_{A^{7}} = \lbrace 1, b, d, a^{-1}, c^{-1}, c^{-1}d, d^{-1}\rbrace$ , \\ 
$( b \otimes a^{-1}d^{-1}c )\vert_{A^{7}} = \lbrace 1, b, d, a, c, d^{-1}, a^{-1}, a^{-2}, c^{-1}a^{-1}b, b^{-1}\rbrace$ , \\ 
$( b \otimes a^{-1}cb )\vert_{A^{7}} = \lbrace 1, b, d, a, c, c^{-1}bd, a^{-1}, d^{-1}, bd\rbrace$ , \\ 
$( b \otimes d^{-1}ac )\vert_{A^{7}} = \lbrace 1, b, d, a, c, a^{-1}cb, d^{-1}\rbrace$ , \\ 
$( b \otimes c^{-1}bd )\vert_{A^{7}} = \lbrace 1, a, c, b, d, b^{-1}da, a^{-1}, c^{-1}, da\rbrace$ , \\ 
$( b \otimes c^{-1}a^{-1}d )\vert_{A^{7}} = \lbrace 1, b, d, a^{-1}, c^{-1}, b^{-1}, c^{-1}a\rbrace$ , \\ 
$( b \otimes a^{-1}d^{-1}b )\vert_{A^{7}} = \lbrace 1, b, d, a, c, a^{-1}, c^{-1}, c^{-1}a^{-1}d, a^{-2}, d^{-1}\rbrace$ , \\ 
$( b \otimes d^{-1}b^{-1}c )\vert_{A^{7}} = \lbrace 1, a, c, b, d^{-1}, a^{-1}, d, da, a^{-1}d^{-1}b, b^{-1}\rbrace$ , \\ 
$( b \otimes a^{-1}d^{-1}bd )\vert_{A^{7}} = \lbrace 1, b, d, a, c, bd, c^{-1}a^{-1}da, a^{-1}, c^{-1}, d^{-1}\rbrace$ , \\ 
$( b \otimes d^{-1}b^{-1}cb )\vert_{A^{7}} = \lbrace 1, a, c, b, d, da, a^{-1}d^{-1}bd, a^{-1}, d^{-1}\rbrace$ , \\ 
$( b \otimes c^{-1}a^{-1}da )\vert_{A^{7}} = \lbrace 1, b, d, a, c^{-1}ac, a^{-1}, c^{-1}, b^{-1}\rbrace$ , \\ 
$( b \otimes d^{-1}b^{-1}da )\vert_{A^{7}} = \lbrace 1, d, a, d^{-1}, a^{-1}d^{-1}ac, b^{-1}, da, c, b\rbrace$ , \\ 
$( b \otimes c^{-1}a^{-1}cb )\vert_{A^{7}} = \lbrace 1, b, d, a, c, c^{-1}bd, a^{-1}, d^{-1}\rbrace$ , \\ 
$( b \otimes a^{-1}d^{-1}ac )\vert_{A^{7}} = \lbrace 1, d^{-1}, a^{-1}, a, c, c^{-1}a^{-1}cb, b, b^{-1}, d\rbrace$ , \\ 
$( b \otimes a^{-1}d^{-1}cb )\vert_{A^{7}} = \lbrace 1, b, d, a, c, c^{-1}a^{-1}bd, a^{-2}, d^{-1}, a^{-1}\rbrace$ , \\ 
$( b \otimes d^{-1}b^{-1}ac )\vert_{A^{7}} = \lbrace 1, b, a, c, d, da, a^{-1}d^{-1}cb, d^{-1}, b^{-1}\rbrace$ , \\ 
$( b \otimes c^{-1}a^{-1}bd )\vert_{A^{7}} = \lbrace 1, b, d, c^{-1}da, a^{-1}, c^{-1}\rbrace$ , \\ 
$( b \otimes a^{2} )\vert_{A^{7}} = \lbrace 1, a, c, b, d, da\rbrace$ , \\ 
$( b \otimes a^{-2} )\vert_{A^{7}} = \lbrace 1, a, c, b, d^{-1}, a^{-1}, c^{-1}, d\rbrace$ , \\ 
$( c \otimes a )\vert_{A^{7}} = \lbrace 1, a, c, b, d\rbrace$ , \\ 
$( c \otimes a^{-1} )\vert_{A^{7}} = \lbrace 1, a, c, b, d^{-1}, a^{-1}, c^{-1}\rbrace$ , \\ 
$( c \otimes b^{-1} )\vert_{A^{7}} = \lbrace 1, a, c, b, b^{-1}, d^{-1}\rbrace$ , \\ 
$( c \otimes d )\vert_{A^{7}} = \lbrace 1, a, c, b, bd, d\rbrace$ , \\ 
$( c \otimes d^{-1} )\vert_{A^{7}} = \lbrace 1, a, c, b, b^{-1}, d^{-1}, a^{-1}\rbrace$ , \\ 
$( c \otimes c )\vert_{A^{7}} = \lbrace 1, a, c, b\rbrace$ , \\ 
$( c \otimes da )\vert_{A^{7}} = \lbrace 1, a, c, b, bd, d, a^{2}, ac\rbrace$ , \\ 
$( c \otimes bd )\vert_{A^{7}} = \lbrace 1, a, c, b, d, bd, da\rbrace$ , \\ 
$( c \otimes cb )\vert_{A^{7}} = \lbrace 1, a, c, b, bd, d\rbrace$ , \\ 
$( c \otimes ac )\vert_{A^{7}} = \lbrace 1, d, a, c, cb, b\rbrace$ , \\ 
$( c \otimes a^{-1}d^{-1} )\vert_{A^{7}} = \lbrace 1, a, c, b, b^{-1}, d^{-1}, a^{-1}, c^{-1}, a^{-2}, c^{-1}a^{-1}\rbrace$ , \\ 
$( c \otimes d^{-1}b^{-1} )\vert_{A^{7}} = \lbrace 1, b^{-1}, d^{-1}, a, c, b, a^{-1}, a^{-1}d^{-1}\rbrace$ , \\ 
$( c \otimes b^{-1}c^{-1} )\vert_{A^{7}} = \lbrace 1, a^{-1}, c^{-1}, b^{-1}, d^{-1}, a, c, b, d^{-1}b^{-1}\rbrace$ , \\ 
$( c \otimes a^{-1}b )\vert_{A^{7}} = \lbrace 1, a, c, b, bd, d, a^{-1}, c^{-1}, c^{-1}d, d^{-1}\rbrace$ , \\ 
$( c \otimes d^{-1}c )\vert_{A^{7}} = \lbrace 1, a, c, b, d^{-1}, a^{-1}, a^{-1}b, b^{-1}\rbrace$ , \\ 
$( c \otimes b^{-1}a )\vert_{A^{7}} = \lbrace 1, a, c, b, b^{-1}, d^{-1}, d, d^{-1}c\rbrace$ , \\ 
$( c \otimes a^{-1}c )\vert_{A^{7}} = \lbrace 1, a, c, b, d^{-1}, a^{-1}, c^{-1}b\rbrace$ , \\ 
$( c \otimes d^{-1}a )\vert_{A^{7}} = \lbrace 1, a, c, b, b^{-1}, d^{-1}, a^{-1}c\rbrace$ , \\ 
$( c \otimes b^{-1}d )\vert_{A^{7}} = \lbrace 1, a, c, b, d, b^{-1}, d^{-1}a\rbrace$ , \\ 
$( c \otimes a^{-1}d )\vert_{A^{7}} = \lbrace 1, a, c, b, bd, d, c^{-1}a, a^{-1}, c^{-1}\rbrace$ , \\ 
$( c \otimes d^{-1}b )\vert_{A^{7}} = \lbrace 1, a, c, b, a^{-1}d, a^{-1}, d^{-1}\rbrace$ , \\ 
$( c \otimes b^{-1}c )\vert_{A^{7}} = \lbrace 1, a, c, b, d^{-1}b, d^{-1}, b^{-1}\rbrace$ , \\ 
$( c \otimes a^{-1}da )\vert_{A^{7}} = \lbrace 1, a, c, b, bd, d, c^{-1}ac\rbrace$ , \\ 
$( c \otimes d^{-1}bd )\vert_{A^{7}} = \lbrace 1, a, c, b, bd, d, a^{-1}da\rbrace$ , \\ 
$( c \otimes b^{-1}cb )\vert_{A^{7}} = \lbrace 1, a, c, b, d^{-1}bd\rbrace$ , \\ 
$( c \otimes a^{-1}d^{-1}a )\vert_{A^{7}} = \lbrace 1, a, c, b, b^{-1}, d^{-1}, a^{-1}, c^{-1}a^{-1}c\rbrace$ , \\ 
$( c \otimes d^{-1}b^{-1}d )\vert_{A^{7}} = \lbrace 1, a, c, b, b^{-1}, d^{-1}, a^{-1}d^{-1}a\rbrace$ , \\ 
$( c \otimes b^{-1}c^{-1}b )\vert_{A^{7}} = \lbrace 1, a^{-1}, c^{-1}, b^{-1}, d^{-1}b^{-1}d, a, c, b\rbrace$ , \\ 
$( c \otimes a^{-1}bd )\vert_{A^{7}} = \lbrace 1, a, c, b, d, bd, c^{-1}da, a^{-1}, c^{-1}\rbrace$ , \\ 
$( c \otimes d^{-1}cb )\vert_{A^{7}} = \lbrace 1, a, c, b, a^{-1}bd, a^{-1}, d^{-1}\rbrace$ , \\ 
$( c \otimes b^{-1}ac )\vert_{A^{7}} = \lbrace 1, d, a, c, d^{-1}cb, d^{-1}, b, b^{-1}\rbrace$ , \\ 
$( c \otimes d^{-1}b^{-1}a )\vert_{A^{7}} = \lbrace 1, b^{-1}, d^{-1}, a, c, b, a^{-1}d^{-1}c\rbrace$ , \\ 
$( c \otimes b^{-1}c^{-1}d )\vert_{A^{7}} = \lbrace 1, a^{-1}, c^{-1}, b^{-1}, b, bd, a, c, d^{-1}b^{-1}a\rbrace$ , \\ 
$( c \otimes a^{-1}d^{-1}c )\vert_{A^{7}} = \lbrace 1, a, c, b, d^{-1}, a^{-1}, a^{-2}, c^{-1}a^{-1}b, b^{-1}\rbrace$ , \\ 
$( c \otimes a^{-1}cb )\vert_{A^{7}} = \lbrace 1, a, c, b, c^{-1}bd, a^{-1}, d, d^{-1}, bd\rbrace$ , \\ 
$( c \otimes d^{-1}ac )\vert_{A^{7}} = \lbrace 1, a, c, a^{-1}cb, d^{-1}, b, b^{-1}\rbrace$ , \\ 
$( c \otimes b^{-1}da )\vert_{A^{7}} = \lbrace 1, a, c, b, d, a^{2}, d^{-1}ac, b^{-1}\rbrace$ , \\ 
$( c \otimes b^{-1}c^{-1}a )\vert_{A^{7}} = \lbrace 1, b^{-1}, d^{-1}, b, bd, a, c, d^{-1}b^{-1}c, a^{-1}, c^{-1}\rbrace$ , \\ 
$( c \otimes a^{-1}d^{-1}b )\vert_{A^{7}} = \lbrace 1, a, c, b, a^{-1}, c^{-1}, c^{-1}a^{-1}d, a^{-2}, d^{-1}\rbrace$ , \\ 
$( c \otimes d^{-1}b^{-1}c )\vert_{A^{7}} = \lbrace 1, a, c, b, d^{-1}, a^{-1}, a^{-1}d^{-1}b, b^{-1}\rbrace$ , \\ 
$( c \otimes a^{-1}d^{-1}bd )\vert_{A^{7}} = \lbrace 1, a, c, b, bd, d, c^{-1}a^{-1}da, a^{-1}, c^{-1}, d^{-1}\rbrace$ , \\ 
$( c \otimes d^{-1}b^{-1}cb )\vert_{A^{7}} = \lbrace 1, a, c, b, a^{-1}d^{-1}bd, a^{-1}, d^{-1}, b^{-1}\rbrace$ , \\ 
$( c \otimes b^{-1}c^{-1}ac )\vert_{A^{7}} = \lbrace 1, b, bd, a, c, d^{-1}b^{-1}cb, d^{-1}, b^{-1}\rbrace$ , \\ 
$( c \otimes d^{-1}b^{-1}da )\vert_{A^{7}} = \lbrace 1, a, c, b, b^{-1}, d^{-1}, d, a^{-1}d^{-1}ac\rbrace$ , \\ 
$( c \otimes b^{-1}c^{-1}bd )\vert_{A^{7}} = \lbrace 1, b, d, b^{-1}, d^{-1}b^{-1}da, a, a^{-1}, c^{-1}, bd, c\rbrace$ , \\ 
$( c \otimes a^{-1}d^{-1}ac )\vert_{A^{7}} = \lbrace 1, d^{-1}, a^{-1}, a, c, c^{-1}a^{-1}cb, b, b^{-1}\rbrace$ , \\ 
$( c \otimes a^{-1}d^{-1}cb )\vert_{A^{7}} = \lbrace 1, a, c, b, c^{-1}a^{-1}bd, a^{-2}, d^{-1}, a^{-1}\rbrace$ , \\ 
$( c \otimes d^{-1}b^{-1}ac )\vert_{A^{7}} = \lbrace 1, a, c, a^{-1}d^{-1}cb, d^{-1}, b^{-1}, b\rbrace$ , \\ 
$( c \otimes b^{-1}c^{-1}da )\vert_{A^{7}} = \lbrace 1, a, c, d, b, bd, a^{2}, d^{-1}b^{-1}ac, b^{-1}, a^{-1}, c^{-1}\rbrace$ , \\ 
$( c \otimes a^{2} )\vert_{A^{7}} = \lbrace 1, a, c, b, d\rbrace$ , \\ 
$( c \otimes a^{-2} )\vert_{A^{7}} = \lbrace 1, a, c, b, d^{-1}, a^{-1}, c^{-1}\rbrace$ , \\ 
$( d \otimes a^{-1} )\vert_{A^{7}} = \lbrace 1, b, d, a, d^{-1}, a^{-1}, c^{-1}\rbrace$ , \\ 
$( d \otimes b )\vert_{A^{7}} = \lbrace 1, a, c, b, d\rbrace$ , \\ 
$( d \otimes b^{-1} )\vert_{A^{7}} = \lbrace 1, a^{-1}, c^{-1}, b^{-1}, d, a, d^{-1}\rbrace$ , \\ 
$( d \otimes d )\vert_{A^{7}} = \lbrace 1, b, d, a\rbrace$ , \\ 
$( d \otimes c^{-1} )\vert_{A^{7}} = \lbrace 1, a^{-1}, c^{-1}, b^{-1}, d, a, b\rbrace$ , \\ 
$( d \otimes c )\vert_{A^{7}} = \lbrace 1, a, c, b, d, a^{2}, ac\rbrace$ , \\ 
$( d \otimes da )\vert_{A^{7}} = \lbrace 1, b, d, a^{2}, a, ac\rbrace$ , \\ 
$( d \otimes bd )\vert_{A^{7}} = \lbrace 1, a, c, b, d, da\rbrace$ , \\ 
$( d \otimes cb )\vert_{A^{7}} = \lbrace 1, a, c, b, d, a^{2}, ac, bd\rbrace$ , \\ 
$( d \otimes ac )\vert_{A^{7}} = \lbrace 1, d, a, c, a^{2}, ac, cb, b\rbrace$ , \\ 
$( d \otimes a^{-1}d^{-1} )\vert_{A^{7}} = \lbrace 1, b, d, a, b^{-1}, d^{-1}, a^{-1}, c^{-1}, a^{-2}, c^{-1}a^{-1}\rbrace$ , \\ 
$( d \otimes b^{-1}c^{-1} )\vert_{A^{7}} = \lbrace 1, a^{-1}, c^{-1}, b^{-1}, d, a, d^{-1}\rbrace$ , \\ 
$( d \otimes c^{-1}a^{-1} )\vert_{A^{7}} = \lbrace 1, d, a, b, a^{-1}, c^{-1}, b^{-1}, d^{-1}, b^{-1}c^{-1}\rbrace$ , \\ 
$( d \otimes a^{-1}b )\vert_{A^{7}} = \lbrace 1, d, a, a^{2}, ac, b, a^{-1}, c^{-1}, c^{-1}d, d^{-1}\rbrace$ , \\ 
$( d \otimes b^{-1}a )\vert_{A^{7}} = \lbrace 1, d, a, a^{2}, ac, d^{-1}, d^{-1}c, b^{-1}, a^{-1}, c^{-1}\rbrace$ , \\ 
$( d \otimes c^{-1}d )\vert_{A^{7}} = \lbrace 1, a^{-1}, c^{-1}, b^{-1}, d, a, b, da, b^{-1}a\rbrace$ , \\ 
$( d \otimes a^{-1}c )\vert_{A^{7}} = \lbrace 1, d, a, a^{2}, ac, b, d^{-1}, a^{-1}, c^{-1}b\rbrace$ , \\ 
$( d \otimes b^{-1}d )\vert_{A^{7}} = \lbrace 1, a^{-1}, c^{-1}, b^{-1}, b, d, a, d^{-1}a\rbrace$ , \\ 
$( d \otimes c^{-1}b )\vert_{A^{7}} = \lbrace 1, d, a, b, a^{-1}, c^{-1}, b^{-1}d\rbrace$ , \\ 
$( d \otimes c^{-1}a )\vert_{A^{7}} = \lbrace 1, d, a, b, da, b^{-1}c, b^{-1}, a^{-1}, c^{-1}\rbrace$ , \\ 
$( d \otimes a^{-1}d )\vert_{A^{7}} = \lbrace 1, b, d, a, c^{-1}a, a^{-1}, c^{-1}\rbrace$ , \\ 
$( d \otimes b^{-1}c )\vert_{A^{7}} = \lbrace 1, d, a, a^{2}, ac, d^{-1}b, d^{-1}\rbrace$ , \\ 
$( d \otimes a^{-1}da )\vert_{A^{7}} = \lbrace 1, b, d, a, c^{-1}ac\rbrace$ , \\ 
$( d \otimes b^{-1}cb )\vert_{A^{7}} = \lbrace 1, a, c, b, d, a^{2}, ac, d^{-1}bd\rbrace$ , \\ 
$( d \otimes c^{-1}ac )\vert_{A^{7}} = \lbrace 1, b, d, da, a, c, b^{-1}cb\rbrace$ , \\ 
$( d \otimes a^{-1}d^{-1}a )\vert_{A^{7}} = \lbrace 1, b^{-1}, d^{-1}, a^{-1}, c^{-1}a^{-1}c, b, d, a\rbrace$ , \\ 
$( d \otimes b^{-1}c^{-1}b )\vert_{A^{7}} = \lbrace 1, a^{-1}, c^{-1}, b^{-1}, d, a, b^{-1}d\rbrace$ , \\ 
$( d \otimes c^{-1}a^{-1}c )\vert_{A^{7}} = \lbrace 1, b, d, a, d^{-1}, a^{-1}, c^{-1}, b^{-1}c^{-1}b\rbrace$ , \\ 
$( d \otimes a^{-1}bd )\vert_{A^{7}} = \lbrace 1, d, a, a^{2}, ac, b, c^{-1}da, a^{-1}, c^{-1}\rbrace$ , \\ 
$( d \otimes b^{-1}ac )\vert_{A^{7}} = \lbrace 1, d, a, c, a^{2}, ac, d^{-1}cb, d^{-1}\rbrace$ , \\ 
$( d \otimes c^{-1}da )\vert_{A^{7}} = \lbrace 1, d, a, b, da, a^{2}, b^{-1}ac, b^{-1}, a^{-1}, c^{-1}\rbrace$ , \\ 
$( d \otimes b^{-1}c^{-1}d )\vert_{A^{7}} = \lbrace 1, a^{-1}, c^{-1}, b^{-1}, d, a, b^{-1}a\rbrace$ , \\ 
$( d \otimes c^{-1}a^{-1}b )\vert_{A^{7}} = \lbrace 1, b, d, a, a^{-1}, c^{-1}, b^{-1}c^{-1}d, d^{-1}\rbrace$ , \\ 
$( d \otimes a^{-1}d^{-1}c )\vert_{A^{7}} = \lbrace 1, d, a, a^{2}, ac, b, d^{-1}, a^{-1}, a^{-2}, c^{-1}a^{-1}b, b^{-1}\rbrace$ , \\ 
$( d \otimes a^{-1}cb )\vert_{A^{7}} = \lbrace 1, b, d, a, c, a^{2}, ac, c^{-1}bd, a^{-1}, d^{-1}\rbrace$ , \\ 
$( d \otimes b^{-1}da )\vert_{A^{7}} = \lbrace 1, b, d, a^{2}, a, d^{-1}ac, b^{-1}, ac, a^{-1}, c^{-1}\rbrace$ , \\ 
$( d \otimes c^{-1}bd )\vert_{A^{7}} = \lbrace 1, d, a, b, b^{-1}da, a^{-1}, c^{-1}, da\rbrace$ , \\ 
$( d \otimes b^{-1}c^{-1}a )\vert_{A^{7}} = \lbrace 1, d, a, b^{-1}, d^{-1}, b^{-1}c, a^{-1}, c^{-1}\rbrace$ , \\ 
$( d \otimes c^{-1}a^{-1}d )\vert_{A^{7}} = \lbrace 1, d, a, b, a^{-1}, c^{-1}, b^{-1}, b^{-1}c^{-1}a\rbrace$ , \\ 
$( d \otimes a^{-1}d^{-1}b )\vert_{A^{7}} = \lbrace 1, d, a, a^{2}, ac, b, a^{-1}, c^{-1}, c^{-1}a^{-1}d, a^{-2}, d^{-1}\rbrace$ , \\ 
$( d \otimes a^{-1}d^{-1}bd )\vert_{A^{7}} = \lbrace 1, d, a, a^{2}, ac, b, c^{-1}a^{-1}da, a^{-1}, c^{-1}, d^{-1}\rbrace$ , \\ 
$( d \otimes b^{-1}c^{-1}ac )\vert_{A^{7}} = \lbrace 1, d, a, c, b^{-1}cb, b^{-1}, d^{-1}, a^{-1}, c^{-1}\rbrace$ , \\ 
$( d \otimes c^{-1}a^{-1}da )\vert_{A^{7}} = \lbrace 1, b, d, a, b^{-1}c^{-1}ac, b^{-1}, a^{-1}, c^{-1}\rbrace$ , \\ 
$( d \otimes b^{-1}c^{-1}bd )\vert_{A^{7}} = \lbrace 1, a^{-1}, c^{-1}, b^{-1}, d, a, b, b^{-1}da\rbrace$ , \\ 
$( d \otimes c^{-1}a^{-1}cb )\vert_{A^{7}} = \lbrace 1, d, a, a^{2}, ac, b, a^{-1}, c^{-1}, b^{-1}c^{-1}bd, d^{-1}\rbrace$ , \\ 
$( d \otimes a^{-1}d^{-1}ac )\vert_{A^{7}} = \lbrace 1, d^{-1}, a^{-1}, a, c, c^{-1}a^{-1}cb, b, b^{-1}, d, a^{2}, ac\rbrace$ , \\ 
$( d \otimes a^{-1}d^{-1}cb )\vert_{A^{7}} = \lbrace 1, b, d, a, c, a^{2}, ac, c^{-1}a^{-1}bd, a^{-2}, d^{-1}, a^{-1}\rbrace$ , \\ 
$( d \otimes b^{-1}c^{-1}da )\vert_{A^{7}} = \lbrace 1, d, a, a^{2}, b^{-1}ac, b^{-1}, a^{-1}, c^{-1}\rbrace$ , \\ 
$( d \otimes c^{-1}a^{-1}bd )\vert_{A^{7}} = \lbrace 1, b, d, a, b^{-1}c^{-1}da, a^{-1}, c^{-1}\rbrace$ , \\ 
$( d \otimes a^{2} )\vert_{A^{7}} = \lbrace 1, d, a, c, b, a^{2}, ac\rbrace$ , \\ 
$( d \otimes a^{-2} )\vert_{A^{7}} = \lbrace 1, d^{-1}, a^{-1}, c^{-1}, b, d, a\rbrace$ , \\ 
$( a^{-1} \otimes b^{-1} )\vert_{A^{7}} = \lbrace 1, a^{-1}, c^{-1}, b^{-1}, d^{-1}, a^{-1}d^{-1}\rbrace$ , \\ 
$( a^{-1} \otimes c^{-1} )\vert_{A^{7}} = \lbrace 1, a^{-1}, c^{-1}, b^{-1}, d^{-1}\rbrace$ , \\ 
$( a^{-1} \otimes a^{-1}d^{-1} )\vert_{A^{7}} = \lbrace 1, d^{-1}, a^{-1}, c^{-1}, b^{-1}, a^{-2}, c^{-1}a^{-1}\rbrace$ , \\ 
$( a^{-1} \otimes d^{-1}b^{-1} )\vert_{A^{7}} = \lbrace 1, a^{-1}, c^{-1}, b^{-1}, d^{-1}, a^{-2}, c^{-1}a^{-1}, a^{-1}d^{-1}\rbrace$ , \\ 
$( a^{-1} \otimes b^{-1}c^{-1} )\vert_{A^{7}} = \lbrace 1, a^{-1}, c^{-1}, b^{-1}, d^{-1}, a^{-1}d^{-1}, d^{-1}b^{-1}\rbrace$ , \\ 
$( a^{-1} \otimes c^{-1}a^{-1} )\vert_{A^{7}} = \lbrace 1, d^{-1}, a^{-1}, c^{-1}, b^{-1}, b^{-1}c^{-1}\rbrace$ , \\ 
$( a^{-1} \otimes a^{-1}b )\vert_{A^{7}} = \lbrace 1, d^{-1}, a^{-1}, b, c^{-1}b, d, c^{-1}, c^{-1}d\rbrace$ , \\ 
$( a^{-1} \otimes b^{-1}a )\vert_{A^{7}} = \lbrace 1, b^{-1}, d^{-1}, d, c^{-1}a, a^{-1}, c^{-1}, a, c, a^{-1}d^{-1}c\rbrace$ , \\ 
$( a^{-1} \otimes c^{-1}d )\vert_{A^{7}} = \lbrace 1, a^{-1}, c^{-1}, b^{-1}, b, d, c^{-1}d, d^{-1}, a, b^{-1}a\rbrace$ , \\ 
$( a^{-1} \otimes a^{-1}c )\vert_{A^{7}} = \lbrace 1, d^{-1}, a^{-1}, b, c^{-1}b\rbrace$ , \\ 
$( a^{-1} \otimes b^{-1}d )\vert_{A^{7}} = \lbrace 1, a^{-1}, c^{-1}, b^{-1}, d, c^{-1}a, d^{-1}, a^{-1}d^{-1}a\rbrace$ , \\ 
$( a^{-1} \otimes c^{-1}b )\vert_{A^{7}} = \lbrace 1, a^{-1}, c^{-1}, b, d, c^{-1}d, d^{-1}, b^{-1}d\rbrace$ , \\ 
$( a^{-1} \otimes c^{-1}a )\vert_{A^{7}} = \lbrace 1, d^{-1}, a^{-1}, c^{-1}, d, a, b^{-1}c, b^{-1}\rbrace$ , \\ 
$( a^{-1} \otimes a^{-1}d )\vert_{A^{7}} = \lbrace 1, d^{-1}, a^{-1}, c^{-1}, b, d, c^{-1}a\rbrace$ , \\ 
$( a^{-1} \otimes b^{-1}c )\vert_{A^{7}} = \lbrace 1, a^{-1}, c^{-1}, a, c, a^{-1}d^{-1}b, d^{-1}, b^{-1}\rbrace$ , \\ 
$( a^{-1} \otimes a^{-1}da )\vert_{A^{7}} = \lbrace 1, b, d, a, c^{-1}ac, d^{-1}, a^{-1}, c^{-1}\rbrace$ , \\ 
$( a^{-1} \otimes b^{-1}cb )\vert_{A^{7}} = \lbrace 1, a, c, b, a^{-1}d^{-1}bd, a^{-1}, d^{-1}, c^{-1}\rbrace$ , \\ 
$( a^{-1} \otimes c^{-1}ac )\vert_{A^{7}} = \lbrace 1, d, a, c, b^{-1}cb, d^{-1}, a^{-1}, c^{-1}\rbrace$ , \\ 
$( a^{-1} \otimes a^{-1}d^{-1}a )\vert_{A^{7}} = \lbrace 1, b^{-1}, d^{-1}, a^{-1}, c^{-1}a^{-1}c, c^{-1}\rbrace$ , \\ 
$( a^{-1} \otimes d^{-1}b^{-1}d )\vert_{A^{7}} = \lbrace 1, a^{-1}, c^{-1}, b^{-1}, d^{-1}, a^{-1}d^{-1}a\rbrace$ , \\ 
$( a^{-1} \otimes b^{-1}c^{-1}b )\vert_{A^{7}} = \lbrace 1, a^{-1}, c^{-1}, b^{-1}, d^{-1}, d^{-1}b^{-1}d\rbrace$ , \\ 
$( a^{-1} \otimes c^{-1}a^{-1}c )\vert_{A^{7}} = \lbrace 1, d^{-1}, a^{-1}, c^{-1}, b^{-1}c^{-1}b\rbrace$ , \\ 
$( a^{-1} \otimes a^{-1}bd )\vert_{A^{7}} = \lbrace 1, d^{-1}, a^{-1}, b, c^{-1}b, d, c^{-1}da, c^{-1}\rbrace$ , \\ 
$( a^{-1} \otimes b^{-1}ac )\vert_{A^{7}} = \lbrace 1, d, c^{-1}a, a^{-1}, c^{-1}, a, c, a^{-1}d^{-1}cb, d^{-1}, b^{-1}\rbrace$ , \\ 
$( a^{-1} \otimes c^{-1}da )\vert_{A^{7}} = \lbrace 1, a^{-1}, c^{-1}, b, d, c^{-1}d, d^{-1}, a, a^{2}, b^{-1}ac, b^{-1}\rbrace$ , \\ 
$( a^{-1} \otimes d^{-1}b^{-1}a )\vert_{A^{7}} = \lbrace 1, b^{-1}, d^{-1}, a^{-1}, c^{-1}, a^{-1}c, a^{-1}d^{-1}c\rbrace$ , \\ 
$( a^{-1} \otimes b^{-1}c^{-1}d )\vert_{A^{7}} = \lbrace 1, a^{-1}, c^{-1}, b^{-1}, d^{-1}, d, a, d^{-1}b^{-1}a\rbrace$ , \\ 
$( a^{-1} \otimes c^{-1}a^{-1}b )\vert_{A^{7}} = \lbrace 1, d^{-1}, a^{-1}, c^{-1}, b, d, b^{-1}c^{-1}d\rbrace$ , \\ 
$( a^{-1} \otimes a^{-1}d^{-1}c )\vert_{A^{7}} = \lbrace 1, d^{-1}, a^{-1}, b, c^{-1}b, a^{-2}, c^{-1}a^{-1}b, b^{-1}\rbrace$ , \\ 
$( a^{-1} \otimes a^{-1}cb )\vert_{A^{7}} = \lbrace 1, a, c, b, d^{-1}, a^{-1}, c^{-1}b, c^{-1}bd, d\rbrace$ , \\ 
$( a^{-1} \otimes b^{-1}da )\vert_{A^{7}} = \lbrace 1, d, c^{-1}a, d^{-1}, a, a^{-1}d^{-1}ac, b^{-1}, c, a^{-1}, c^{-1}\rbrace$ , \\ 
$( a^{-1} \otimes c^{-1}bd )\vert_{A^{7}} = \lbrace 1, b, a^{-1}d, a^{-1}, c^{-1}, d, c^{-1}d, b^{-1}da, a\rbrace$ , \\ 
$( a^{-1} \otimes b^{-1}c^{-1}a )\vert_{A^{7}} = \lbrace 1, a^{-1}, c^{-1}, b^{-1}, d^{-1}, a^{-1}d^{-1}, d, a, d^{-1}b^{-1}c\rbrace$ , \\ 
$( a^{-1} \otimes c^{-1}a^{-1}d )\vert_{A^{7}} = \lbrace 1, d^{-1}, a^{-1}, c^{-1}, b^{-1}, b, d, b^{-1}c^{-1}a\rbrace$ , \\ 
$( a^{-1} \otimes a^{-1}d^{-1}b )\vert_{A^{7}} = \lbrace 1, d^{-1}, a^{-1}, b, c^{-1}b, c^{-1}, c^{-1}a^{-1}d, a^{-2}\rbrace$ , \\ 
$( a^{-1} \otimes d^{-1}b^{-1}c )\vert_{A^{7}} = \lbrace 1, a^{-1}, c^{-1}, d^{-1}, a^{-2}, c^{-1}a^{-1}, a^{-1}c, a^{-1}d^{-1}b, b^{-1}\rbrace$ , \\ 
$( a^{-1} \otimes a^{-1}d^{-1}bd )\vert_{A^{7}} = \lbrace 1, d^{-1}, a^{-1}, b, c^{-1}b, d, c^{-1}a^{-1}da, c^{-1}\rbrace$ , \\ 
$( a^{-1} \otimes d^{-1}b^{-1}cb )\vert_{A^{7}} = \lbrace 1, a, c, b, d^{-1}, a^{-1}c, a^{-1}d^{-1}bd, a^{-1}, c^{-1}, a^{-2}, c^{-1}a^{-1}, b^{-1}\rbrace$ , \\ 
$( a^{-1} \otimes b^{-1}c^{-1}ac )\vert_{A^{7}} = \lbrace 1, d, a, c, d^{-1}b^{-1}cb, d^{-1}, b^{-1}, a^{-1}, c^{-1}, a^{-1}d^{-1}\rbrace$ , \\ 
$( a^{-1} \otimes c^{-1}a^{-1}da )\vert_{A^{7}} = \lbrace 1, b, d, a, b^{-1}c^{-1}ac, b^{-1}, a^{-1}, c^{-1}, d^{-1}\rbrace$ , \\ 
$( a^{-1} \otimes d^{-1}b^{-1}da )\vert_{A^{7}} = \lbrace 1, b^{-1}, d^{-1}, d, c^{-1}a, a^{-1}, c^{-1}, a, a^{-1}d^{-1}ac, a^{-1}c\rbrace$ , \\ 
$( a^{-1} \otimes b^{-1}c^{-1}bd )\vert_{A^{7}} = \lbrace 1, a^{-1}, c^{-1}, b^{-1}, b, d, c^{-1}d, d^{-1}, d^{-1}b^{-1}da, a\rbrace$ , \\ 
$( a^{-1} \otimes c^{-1}a^{-1}cb )\vert_{A^{7}} = \lbrace 1, d^{-1}, a^{-1}, b, c^{-1}b, c^{-1}, b^{-1}c^{-1}bd, d\rbrace$ , \\ 
$( a^{-1} \otimes a^{-1}d^{-1}ac )\vert_{A^{7}} = \lbrace 1, d^{-1}, a^{-1}, a, c, c^{-1}a^{-1}cb, b, b^{-1}, c^{-1}b\rbrace$ , \\ 
$( a^{-1} \otimes a^{-1}d^{-1}cb )\vert_{A^{7}} = \lbrace 1, a, c, b, d^{-1}, a^{-1}, c^{-1}b, c^{-1}a^{-1}bd, a^{-2}\rbrace$ , \\ 
$( a^{-1} \otimes d^{-1}b^{-1}ac )\vert_{A^{7}} = \lbrace 1, a^{-1}, c^{-1}, a, c, d^{-1}, a^{-1}c, a^{-1}d^{-1}cb, b^{-1}\rbrace$ , \\ 
$( a^{-1} \otimes b^{-1}c^{-1}da )\vert_{A^{7}} = \lbrace 1, d^{-1}, a^{-1}, c^{-1}, d, a, a^{2}, d^{-1}b^{-1}ac, b^{-1}\rbrace$ , \\ 
$( a^{-1} \otimes c^{-1}a^{-1}bd )\vert_{A^{7}} = \lbrace 1, d^{-1}, a^{-1}, c^{-1}, b, d, b^{-1}c^{-1}da\rbrace$ , \\ 
$( a^{-1} \otimes a^{-2} )\vert_{A^{7}} = \lbrace 1, d^{-1}, a^{-1}, c^{-1}\rbrace$ , \\ 
$( b^{-1} \otimes a^{-1} )\vert_{A^{7}} = \lbrace 1, b^{-1}, d^{-1}, a^{-1}, c^{-1}, b^{-1}c^{-1}\rbrace$ , \\ 
$( b^{-1} \otimes b^{-1} )\vert_{A^{7}} = \lbrace 1, a^{-1}, c^{-1}, b^{-1}, d^{-1}\rbrace$ , \\ 
$( b^{-1} \otimes d^{-1} )\vert_{A^{7}} = \lbrace 1, a^{-1}, c^{-1}, b^{-1}, d^{-1}\rbrace$ , \\ 
$( b^{-1} \otimes a^{-1}d^{-1} )\vert_{A^{7}} = \lbrace 1, b^{-1}, d^{-1}, a^{-1}, b^{-1}c^{-1}, c^{-1}, a^{-2}, c^{-1}a^{-1}\rbrace$ , \\ 
$( b^{-1} \otimes d^{-1}b^{-1} )\vert_{A^{7}} = \lbrace 1, a^{-1}, c^{-1}, b^{-1}, d^{-1}, a^{-1}d^{-1}\rbrace$ , \\ 
$( b^{-1} \otimes b^{-1}c^{-1} )\vert_{A^{7}} = \lbrace 1, a^{-1}, c^{-1}, b^{-1}, d^{-1}, d^{-1}b^{-1}\rbrace$ , \\ 
$( b^{-1} \otimes c^{-1}a^{-1} )\vert_{A^{7}} = \lbrace 1, b^{-1}, d^{-1}, a^{-1}, c^{-1}, d^{-1}b^{-1}, b^{-1}c^{-1}\rbrace$ , \\ 
$( b^{-1} \otimes a^{-1}b )\vert_{A^{7}} = \lbrace 1, a, c, d^{-1}b, d^{-1}, b^{-1}, b, d, a^{-1}, c^{-1}, b^{-1}c^{-1}d\rbrace$ , \\ 
$( b^{-1} \otimes d^{-1}c )\vert_{A^{7}} = \lbrace 1, d^{-1}, a^{-1}, a, d^{-1}c, b^{-1}, b, a^{-1}b\rbrace$ , \\ 
$( b^{-1} \otimes b^{-1}a )\vert_{A^{7}} = \lbrace 1, b^{-1}, d^{-1}, a, d^{-1}a, c, d^{-1}c, a^{-1}, c^{-1}\rbrace$ , \\ 
$( b^{-1} \otimes a^{-1}c )\vert_{A^{7}} = \lbrace 1, a, c, d^{-1}b, d^{-1}, b^{-1}, a^{-1}, c^{-1}, b^{-1}c^{-1}b\rbrace$ , \\ 
$( b^{-1} \otimes d^{-1}a )\vert_{A^{7}} = \lbrace 1, b^{-1}, d^{-1}, a, d^{-1}c, a^{-1}c, a^{-1}, c^{-1}\rbrace$ , \\ 
$( b^{-1} \otimes b^{-1}d )\vert_{A^{7}} = \lbrace 1, a^{-1}, c^{-1}, b^{-1}, d^{-1}, a, d^{-1}a\rbrace$ , \\ 
$( b^{-1} \otimes a^{-1}d )\vert_{A^{7}} = \lbrace 1, b^{-1}, d^{-1}, b, d, b^{-1}c^{-1}a, a^{-1}, c^{-1}\rbrace$ , \\ 
$( b^{-1} \otimes d^{-1}b )\vert_{A^{7}} = \lbrace 1, b^{-1}, d^{-1}, b, a^{-1}d, a^{-1}\rbrace$ , \\ 
$( b^{-1} \otimes b^{-1}c )\vert_{A^{7}} = \lbrace 1, a^{-1}, c^{-1}, b^{-1}, d^{-1}, a, c, d^{-1}b\rbrace$ , \\ 
$( b^{-1} \otimes a^{-1}da )\vert_{A^{7}} = \lbrace 1, b, d, a, b^{-1}c^{-1}ac, b^{-1}, a^{-1}, c^{-1}, d^{-1}\rbrace$ , \\ 
$( b^{-1} \otimes d^{-1}bd )\vert_{A^{7}} = \lbrace 1, b, d, a^{-1}da, b^{-1}, d^{-1}\rbrace$ , \\ 
$( b^{-1} \otimes b^{-1}cb )\vert_{A^{7}} = \lbrace 1, a, c, b, d^{-1}bd, b^{-1}, d^{-1}\rbrace$ , \\ 
$( b^{-1} \otimes a^{-1}d^{-1}a )\vert_{A^{7}} = \lbrace 1, b^{-1}, d^{-1}, a^{-1}, c^{-1}a^{-1}c, c^{-1}\rbrace$ , \\ 
$( b^{-1} \otimes d^{-1}b^{-1}d )\vert_{A^{7}} = \lbrace 1, a^{-1}, c^{-1}, b^{-1}, d^{-1}, a^{-1}d^{-1}a\rbrace$ , \\ 
$( b^{-1} \otimes b^{-1}c^{-1}b )\vert_{A^{7}} = \lbrace 1, a^{-1}, c^{-1}, b^{-1}, d^{-1}, d^{-1}b^{-1}d\rbrace$ , \\ 
$( b^{-1} \otimes c^{-1}a^{-1}c )\vert_{A^{7}} = \lbrace 1, b^{-1}, d^{-1}, a^{-1}, c^{-1}, b^{-1}c^{-1}b\rbrace$ , \\ 
$( b^{-1} \otimes a^{-1}bd )\vert_{A^{7}} = \lbrace 1, a, c, d^{-1}b, d^{-1}, b^{-1}, b, d, b^{-1}c^{-1}da, a^{-1}, c^{-1}\rbrace$ , \\ 
$( b^{-1} \otimes d^{-1}cb )\vert_{A^{7}} = \lbrace 1, a, c, b, d^{-1}, d^{-1}c, b^{-1}, a^{-1}bd, a^{-1}\rbrace$ , \\ 
$( b^{-1} \otimes b^{-1}ac )\vert_{A^{7}} = \lbrace 1, b^{-1}, a, d^{-1}a, c, d^{-1}cb, d^{-1}, a^{-1}, c^{-1}\rbrace$ , \\ 
$( b^{-1} \otimes d^{-1}b^{-1}a )\vert_{A^{7}} = \lbrace 1, b^{-1}, d^{-1}, a, c, a^{-1}d^{-1}c, a^{-1}, c^{-1}\rbrace$ , \\ 
$( b^{-1} \otimes b^{-1}c^{-1}d )\vert_{A^{7}} = \lbrace 1, a^{-1}, c^{-1}, b^{-1}, d^{-1}, a, d^{-1}a, d^{-1}b^{-1}a\rbrace$ , \\ 
$( b^{-1} \otimes c^{-1}a^{-1}b )\vert_{A^{7}} = \lbrace 1, b^{-1}, d^{-1}, a^{-1}, c^{-1}, d, b^{-1}d, b^{-1}c^{-1}d\rbrace$ , \\ 
$( b^{-1} \otimes a^{-1}d^{-1}c )\vert_{A^{7}} = \lbrace 1, a, c, d^{-1}b, d^{-1}, b^{-1}, a^{-1}, c^{-1}b, a^{-2}, c^{-1}a^{-1}b\rbrace$ , \\ 
$( b^{-1} \otimes a^{-1}cb )\vert_{A^{7}} = \lbrace 1, a, c, d^{-1}b, d^{-1}, a^{-1}, c^{-1}, b, b^{-1}c^{-1}bd, d\rbrace$ , \\ 
$( b^{-1} \otimes d^{-1}ac )\vert_{A^{7}} = \lbrace 1, a, b^{-1}c, d^{-1}, d^{-1}c, a^{-1}cb, b, b^{-1}\rbrace$ , \\ 
$( b^{-1} \otimes b^{-1}da )\vert_{A^{7}} = \lbrace 1, b^{-1}, d^{-1}, d, a, d^{-1}a, a^{2}, d^{-1}ac, c, a^{-1}, c^{-1}\rbrace$ , \\ 
$( b^{-1} \otimes b^{-1}c^{-1}a )\vert_{A^{7}} = \lbrace 1, b^{-1}, d^{-1}, a, d^{-1}a, d^{-1}b^{-1}c, a^{-1}, c^{-1}\rbrace$ , \\ 
$( b^{-1} \otimes c^{-1}a^{-1}d )\vert_{A^{7}} = \lbrace 1, b^{-1}, d^{-1}, a^{-1}, c^{-1}, d^{-1}b^{-1}, d, b^{-1}d, b^{-1}c^{-1}a\rbrace$ , \\ 
$( b^{-1} \otimes a^{-1}d^{-1}b )\vert_{A^{7}} = \lbrace 1, a, c, d^{-1}b, d^{-1}, b^{-1}, a^{-1}, b^{-1}c^{-1}, c^{-1}b, c^{-1}a^{-1}d, a^{-2}\rbrace$ , \\ 
$( b^{-1} \otimes d^{-1}b^{-1}c )\vert_{A^{7}} = \lbrace 1, a^{-1}, c^{-1}, b^{-1}, d^{-1}, a, c, a^{-1}d^{-1}b\rbrace$ , \\ 
$( b^{-1} \otimes a^{-1}d^{-1}bd )\vert_{A^{7}} = \lbrace 1, a, c, d^{-1}b, d^{-1}, b^{-1}, b, d, a^{-1}, c^{-1}b, c^{-1}a^{-1}da, b^{-1}c^{-1}\rbrace$ , \\ 
$( b^{-1} \otimes d^{-1}b^{-1}cb )\vert_{A^{7}} = \lbrace 1, a, c, b, a^{-1}d^{-1}bd, a^{-1}, d^{-1}, b^{-1}\rbrace$ , \\ 
$( b^{-1} \otimes b^{-1}c^{-1}ac )\vert_{A^{7}} = \lbrace 1, b^{-1}, a, d^{-1}a, c, d^{-1}b^{-1}cb, d^{-1}, a^{-1}, c^{-1}\rbrace$ , \\ 
$( b^{-1} \otimes c^{-1}a^{-1}da )\vert_{A^{7}} = \lbrace 1, a^{-1}, c^{-1}, d, b^{-1}d, a, b^{-1}c^{-1}ac, b^{-1}, d^{-1}, d^{-1}b^{-1}\rbrace$ , \\ 
$( b^{-1} \otimes d^{-1}b^{-1}da )\vert_{A^{7}} = \lbrace 1, b^{-1}, d^{-1}, a, d^{-1}a, a^{-1}d^{-1}ac, c, a^{-1}, c^{-1}\rbrace$ , \\ 
$( b^{-1} \otimes b^{-1}c^{-1}bd )\vert_{A^{7}} = \lbrace 1, a^{-1}, c^{-1}, b^{-1}, d, d^{-1}b^{-1}da, a, d^{-1}a\rbrace$ , \\ 
$( b^{-1} \otimes c^{-1}a^{-1}cb )\vert_{A^{7}} = \lbrace 1, a, c, d^{-1}b, d^{-1}, b^{-1}, a^{-1}, c^{-1}, b, b^{-1}c^{-1}bd, d, b^{-1}d\rbrace$ , \\ 
$( b^{-1} \otimes a^{-1}d^{-1}ac )\vert_{A^{7}} = \lbrace 1, d^{-1}, a^{-1}, a, b^{-1}c, b^{-1}, c^{-1}a^{-1}cb, c^{-1}b, c, d^{-1}b\rbrace$ , \\ 
$( b^{-1} \otimes a^{-1}d^{-1}cb )\vert_{A^{7}} = \lbrace 1, a, c, d^{-1}b, d^{-1}, b^{-1}, b, a^{-1}, c^{-1}b, c^{-1}a^{-1}bd, a^{-2}\rbrace$ , \\ 
$( b^{-1} \otimes d^{-1}b^{-1}ac )\vert_{A^{7}} = \lbrace 1, b^{-1}, d^{-1}, a, c, a^{-1}d^{-1}cb, a^{-1}, c^{-1}\rbrace$ , \\ 
$( b^{-1} \otimes b^{-1}c^{-1}da )\vert_{A^{7}} = \lbrace 1, b^{-1}, d^{-1}, d, a, d^{-1}a, a^{2}, d^{-1}b^{-1}ac, a^{-1}, c^{-1}\rbrace$ , \\ 
$( b^{-1} \otimes c^{-1}a^{-1}bd )\vert_{A^{7}} = \lbrace 1, b^{-1}, d^{-1}, b, d, a^{-1}, c^{-1}, b^{-1}d, b^{-1}c^{-1}da\rbrace$ , \\ 
$( b^{-1} \otimes a^{2} )\vert_{A^{7}} = \lbrace 1, d, a, b^{-1}c, b^{-1}, d^{-1}, c, d^{-1}c\rbrace$ , \\ 
$( b^{-1} \otimes a^{-2} )\vert_{A^{7}} = \lbrace 1, d^{-1}, a^{-1}, c^{-1}, b^{-1}, b^{-1}c^{-1}\rbrace$ , \\ 
$( c^{-1} \otimes b^{-1} )\vert_{A^{7}} = \lbrace 1, a^{-1}, c^{-1}, b^{-1}, d^{-1}\rbrace$ , \\ 
$( c^{-1} \otimes d^{-1} )\vert_{A^{7}} = \lbrace 1, a^{-1}, c^{-1}, b^{-1}, d^{-1}, a^{-2}, c^{-1}a^{-1}\rbrace$ , \\ 
$( c^{-1} \otimes c^{-1} )\vert_{A^{7}} = \lbrace 1, a^{-1}, c^{-1}, b^{-1}\rbrace$ , \\ 
$( c^{-1} \otimes a^{-1}d^{-1} )\vert_{A^{7}} = \lbrace 1, a^{-1}, c^{-1}, b^{-1}, d^{-1}, b^{-1}c^{-1}, a^{-2}, c^{-1}a^{-1}\rbrace$ , \\ 
$( c^{-1} \otimes d^{-1}b^{-1} )\vert_{A^{7}} = \lbrace 1, a^{-1}, c^{-1}, b^{-1}, d^{-1}, a^{-2}, c^{-1}a^{-1}, a^{-1}d^{-1}\rbrace$ , \\ 
$( c^{-1} \otimes b^{-1}c^{-1} )\vert_{A^{7}} = \lbrace 1, a^{-1}, c^{-1}, b^{-1}, d^{-1}, d^{-1}b^{-1}\rbrace$ , \\ 
$( c^{-1} \otimes c^{-1}a^{-1} )\vert_{A^{7}} = \lbrace 1, a^{-1}, c^{-1}, b^{-1}, d^{-1}, b^{-1}c^{-1}\rbrace$ , \\ 
$( c^{-1} \otimes d^{-1}c )\vert_{A^{7}} = \lbrace 1, d^{-1}, a^{-1}, a, b^{-1}c, b^{-1}, b, a^{-2}, c^{-1}a^{-1}b\rbrace$ , \\ 
$( c^{-1} \otimes b^{-1}a )\vert_{A^{7}} = \lbrace 1, b^{-1}, d^{-1}, d, a, b^{-1}a, a^{-1}, c^{-1}, c, d^{-1}c\rbrace$ , \\ 
$( c^{-1} \otimes c^{-1}d )\vert_{A^{7}} = \lbrace 1, a^{-1}, c^{-1}, b^{-1}, d, b^{-1}d, a, b^{-1}a\rbrace$ , \\ 
$( c^{-1} \otimes d^{-1}a )\vert_{A^{7}} = \lbrace 1, b^{-1}, d^{-1}, a, b^{-1}c, a^{-1}, c^{-1}a^{-1}c, c^{-1}\rbrace$ , \\ 
$( c^{-1} \otimes b^{-1}d )\vert_{A^{7}} = \lbrace 1, a^{-1}, c^{-1}, b^{-1}, d, a, b^{-1}a, d^{-1}a\rbrace$ , \\ 
$( c^{-1} \otimes c^{-1}b )\vert_{A^{7}} = \lbrace 1, a^{-1}, c^{-1}, b^{-1}, d, b^{-1}d\rbrace$ , \\ 
$( c^{-1} \otimes c^{-1}a )\vert_{A^{7}} = \lbrace 1, a^{-1}, c^{-1}, b^{-1}, d, a, b^{-1}c\rbrace$ , \\ 
$( c^{-1} \otimes d^{-1}b )\vert_{A^{7}} = \lbrace 1, b^{-1}, b, c^{-1}a^{-1}d, a^{-2}, d^{-1}, a^{-1}\rbrace$ , \\ 
$( c^{-1} \otimes b^{-1}c )\vert_{A^{7}} = \lbrace 1, a^{-1}, c^{-1}, b^{-1}, a, c, d^{-1}b, d^{-1}\rbrace$ , \\ 
$( c^{-1} \otimes d^{-1}bd )\vert_{A^{7}} = \lbrace 1, b, d, c^{-1}a^{-1}da, a^{-1}, b^{-1}\rbrace$ , \\ 
$( c^{-1} \otimes b^{-1}cb )\vert_{A^{7}} = \lbrace 1, a, c, b, d^{-1}bd, a^{-1}, c^{-1}, b^{-1}\rbrace$ , \\ 
$( c^{-1} \otimes c^{-1}ac )\vert_{A^{7}} = \lbrace 1, d, a, c, b^{-1}cb, a^{-1}, c^{-1}, b^{-1}\rbrace$ , \\ 
$( c^{-1} \otimes a^{-1}d^{-1}a )\vert_{A^{7}} = \lbrace 1, b^{-1}, d^{-1}, a^{-1}, c^{-1}a^{-1}c, c^{-1}\rbrace$ , \\ 
$( c^{-1} \otimes d^{-1}b^{-1}d )\vert_{A^{7}} = \lbrace 1, a^{-1}, c^{-1}, b^{-1}, d^{-1}, a^{-1}d^{-1}a\rbrace$ , \\ 
$( c^{-1} \otimes b^{-1}c^{-1}b )\vert_{A^{7}} = \lbrace 1, a^{-1}, c^{-1}, b^{-1}, d^{-1}b^{-1}d\rbrace$ , \\ 
$( c^{-1} \otimes c^{-1}a^{-1}c )\vert_{A^{7}} = \lbrace 1, a^{-1}, c^{-1}, b^{-1}, d^{-1}, b^{-1}c^{-1}b\rbrace$ , \\ 
$( c^{-1} \otimes d^{-1}cb )\vert_{A^{7}} = \lbrace 1, a, c, b, b^{-1}c, b^{-1}, c^{-1}a^{-1}bd, a^{-2}, d^{-1}, a^{-1}\rbrace$ , \\ 
$( c^{-1} \otimes b^{-1}ac )\vert_{A^{7}} = \lbrace 1, b^{-1}, d, a, b^{-1}a, a^{-1}, c^{-1}, c, d^{-1}cb, d^{-1}\rbrace$ , \\ 
$( c^{-1} \otimes c^{-1}da )\vert_{A^{7}} = \lbrace 1, a^{-1}, c^{-1}, d, b^{-1}d, a, a^{2}, b^{-1}ac, b^{-1}\rbrace$ , \\ 
$( c^{-1} \otimes d^{-1}b^{-1}a )\vert_{A^{7}} = \lbrace 1, b^{-1}, d^{-1}, a^{-1}, c^{-1}, a, c, a^{-1}d^{-1}c\rbrace$ , \\ 
$( c^{-1} \otimes b^{-1}c^{-1}d )\vert_{A^{7}} = \lbrace 1, a^{-1}, c^{-1}, b^{-1}, d, a, d^{-1}b^{-1}a\rbrace$ , \\ 
$( c^{-1} \otimes c^{-1}a^{-1}b )\vert_{A^{7}} = \lbrace 1, a^{-1}, c^{-1}, b^{-1}, d, b^{-1}d, b^{-1}c^{-1}d, d^{-1}\rbrace$ , \\ 
$( c^{-1} \otimes a^{-1}d^{-1}c )\vert_{A^{7}} = \lbrace 1, d^{-1}, a^{-1}, b^{-1}, c^{-1}b, a^{-2}, c^{-1}a^{-1}b\rbrace$ , \\ 
$( c^{-1} \otimes d^{-1}ac )\vert_{A^{7}} = \lbrace 1, a, b^{-1}c, a^{-1}, c^{-1}a^{-1}cb, d^{-1}, b, b^{-1}\rbrace$ , \\ 
$( c^{-1} \otimes b^{-1}da )\vert_{A^{7}} = \lbrace 1, d, c^{-1}a, b^{-1}, a, b^{-1}a, a^{2}, d^{-1}ac, c, a^{-1}, c^{-1}\rbrace$ , \\ 
$( c^{-1} \otimes c^{-1}bd )\vert_{A^{7}} = \lbrace 1, a^{-1}, c^{-1}, b^{-1}, b, d, b^{-1}d, b^{-1}da, a\rbrace$ , \\ 
$( c^{-1} \otimes b^{-1}c^{-1}a )\vert_{A^{7}} = \lbrace 1, a^{-1}, c^{-1}, b^{-1}, d^{-1}, d, a, d^{-1}b^{-1}c\rbrace$ , \\ 
$( c^{-1} \otimes c^{-1}a^{-1}d )\vert_{A^{7}} = \lbrace 1, a^{-1}, c^{-1}, b^{-1}, d, b^{-1}d, b^{-1}c^{-1}a\rbrace$ , \\ 
$( c^{-1} \otimes a^{-1}d^{-1}b )\vert_{A^{7}} = \lbrace 1, b^{-1}, a^{-1}, b^{-1}c^{-1}, c^{-1}b, c^{-1}a^{-1}d, a^{-2}, d^{-1}\rbrace$ , \\ 
$( c^{-1} \otimes d^{-1}b^{-1}c )\vert_{A^{7}} = \lbrace 1, a^{-1}, c^{-1}, b^{-1}, d^{-1}, a^{-2}, c^{-1}a^{-1}, a, c, a^{-1}d^{-1}b\rbrace$ , \\ 
$( c^{-1} \otimes a^{-1}d^{-1}bd )\vert_{A^{7}} = \lbrace 1, b, d, a^{-1}, c^{-1}b, c^{-1}a^{-1}da, b^{-1}, b^{-1}c^{-1}, d^{-1}\rbrace$ , \\ 
$( c^{-1} \otimes d^{-1}b^{-1}cb )\vert_{A^{7}} = \lbrace 1, a, c, b, a^{-1}d^{-1}bd, a^{-1}, d^{-1}, b^{-1}, a^{-2}, c^{-1}a^{-1}\rbrace$ , \\ 
$( c^{-1} \otimes b^{-1}c^{-1}ac )\vert_{A^{7}} = \lbrace 1, d, a, c, d^{-1}b^{-1}cb, d^{-1}, b^{-1}, a^{-1}, c^{-1}\rbrace$ , \\ 
$( c^{-1} \otimes c^{-1}a^{-1}da )\vert_{A^{7}} = \lbrace 1, a^{-1}, c^{-1}, d, b^{-1}d, a, b^{-1}c^{-1}ac, b^{-1}\rbrace$ , \\ 
$( c^{-1} \otimes d^{-1}b^{-1}da )\vert_{A^{7}} = \lbrace 1, b^{-1}, d^{-1}, d, a, b^{-1}a, a^{-1}, c^{-1}, a^{-1}d^{-1}ac, c\rbrace$ , \\ 
$( c^{-1} \otimes b^{-1}c^{-1}bd )\vert_{A^{7}} = \lbrace 1, a^{-1}, c^{-1}, b^{-1}, d, b^{-1}d, d^{-1}b^{-1}da, a\rbrace$ , \\ 
$( c^{-1} \otimes c^{-1}a^{-1}cb )\vert_{A^{7}} = \lbrace 1, a^{-1}, c^{-1}, b, b^{-1}c^{-1}bd, d, d^{-1}, b^{-1}d\rbrace$ , \\ 
$( c^{-1} \otimes a^{-1}d^{-1}ac )\vert_{A^{7}} = \lbrace 1, d^{-1}, a^{-1}, a, b^{-1}c, b^{-1}, c^{-1}a^{-1}cb, c^{-1}b\rbrace$ , \\ 
$( c^{-1} \otimes a^{-1}d^{-1}cb )\vert_{A^{7}} = \lbrace 1, a, c, b, b^{-1}, a^{-1}, c^{-1}b, c^{-1}a^{-1}bd, a^{-2}, d^{-1}\rbrace$ , \\ 
$( c^{-1} \otimes d^{-1}b^{-1}ac )\vert_{A^{7}} = \lbrace 1, a^{-1}, c^{-1}, b^{-1}, a, c, a^{-1}d^{-1}cb, d^{-1}\rbrace$ , \\ 
$( c^{-1} \otimes b^{-1}c^{-1}da )\vert_{A^{7}} = \lbrace 1, a^{-1}, c^{-1}, b^{-1}, d, a, a^{2}, d^{-1}b^{-1}ac\rbrace$ , \\ 
$( c^{-1} \otimes c^{-1}a^{-1}bd )\vert_{A^{7}} = \lbrace 1, a^{-1}, c^{-1}, b^{-1}, b, d, b^{-1}d, b^{-1}c^{-1}da\rbrace$ , \\ 
$( c^{-1} \otimes a^{2} )\vert_{A^{7}} = \lbrace 1, d, a, b^{-1}c, b^{-1}\rbrace$ , \\ 
$( c^{-1} \otimes a^{-2} )\vert_{A^{7}} = \lbrace 1, d^{-1}, a^{-1}, c^{-1}, b^{-1}, b^{-1}c^{-1}\rbrace$ , \\ 
$( d^{-1} \otimes a^{-1} )\vert_{A^{7}} = \lbrace 1, b^{-1}, d^{-1}, a^{-1}, c^{-1}\rbrace$ , \\ 
$( d^{-1} \otimes d^{-1} )\vert_{A^{7}} = \lbrace 1, b^{-1}, d^{-1}, a^{-1}\rbrace$ , \\ 
$( d^{-1} \otimes c^{-1} )\vert_{A^{7}} = \lbrace 1, a^{-1}, c^{-1}, b^{-1}, d^{-1}, d^{-1}b^{-1}\rbrace$ , \\ 
$( d^{-1} \otimes a^{-1}d^{-1} )\vert_{A^{7}} = \lbrace 1, b^{-1}, d^{-1}, a^{-1}, c^{-1}, a^{-2}, c^{-1}a^{-1}\rbrace$ , \\ 
$( d^{-1} \otimes d^{-1}b^{-1} )\vert_{A^{7}} = \lbrace 1, a^{-1}, c^{-1}, b^{-1}, d^{-1}, a^{-1}d^{-1}\rbrace$ , \\ 
$( d^{-1} \otimes b^{-1}c^{-1} )\vert_{A^{7}} = \lbrace 1, a^{-1}, c^{-1}, b^{-1}, d^{-1}, a^{-1}d^{-1}, d^{-1}b^{-1}\rbrace$ , \\ 
$( d^{-1} \otimes c^{-1}a^{-1} )\vert_{A^{7}} = \lbrace 1, d^{-1}, a^{-1}, c^{-1}, b^{-1}, d^{-1}b^{-1}, b^{-1}c^{-1}\rbrace$ , \\ 
$( d^{-1} \otimes a^{-1}b )\vert_{A^{7}} = \lbrace 1, d^{-1}, a^{-1}, a, c, b, a^{-1}b, b^{-1}, d, c^{-1}, c^{-1}d\rbrace$ , \\ 
$( d^{-1} \otimes d^{-1}c )\vert_{A^{7}} = \lbrace 1, a, c, b, d^{-1}, a^{-1}, a^{-1}c, a^{-1}b, b^{-1}\rbrace$ , \\ 
$( d^{-1} \otimes c^{-1}d )\vert_{A^{7}} = \lbrace 1, a^{-1}, c^{-1}, b^{-1}, b, a^{-1}d, d^{-1}, d, a, d^{-1}b^{-1}a\rbrace$ , \\ 
$( d^{-1} \otimes a^{-1}c )\vert_{A^{7}} = \lbrace 1, d^{-1}, a^{-1}, a, c, b, a^{-1}b, b^{-1}, c^{-1}b\rbrace$ , \\ 
$( d^{-1} \otimes d^{-1}a )\vert_{A^{7}} = \lbrace 1, b^{-1}, d^{-1}, a^{-1}, a^{-1}c\rbrace$ , \\ 
$( d^{-1} \otimes c^{-1}b )\vert_{A^{7}} = \lbrace 1, b, a^{-1}d, a^{-1}, d^{-1}, c^{-1}, b^{-1}, d^{-1}b^{-1}d\rbrace$ , \\ 
$( d^{-1} \otimes c^{-1}a )\vert_{A^{7}} = \lbrace 1, d^{-1}, a^{-1}, d, a, d^{-1}b^{-1}c, b^{-1}, c^{-1}\rbrace$ , \\ 
$( d^{-1} \otimes a^{-1}d )\vert_{A^{7}} = \lbrace 1, b^{-1}, d^{-1}, a^{-1}, b, d, c^{-1}a, c^{-1}\rbrace$ , \\ 
$( d^{-1} \otimes d^{-1}b )\vert_{A^{7}} = \lbrace 1, a, c, b, b^{-1}, d^{-1}, a^{-1}, a^{-1}d\rbrace$ , \\ 
$( d^{-1} \otimes a^{-1}da )\vert_{A^{7}} = \lbrace 1, b, d, a, c^{-1}ac, b^{-1}, d^{-1}, a^{-1}\rbrace$ , \\ 
$( d^{-1} \otimes d^{-1}bd )\vert_{A^{7}} = \lbrace 1, a, c, b, d, a^{-1}da, d^{-1}, a^{-1}\rbrace$ , \\ 
$( d^{-1} \otimes c^{-1}ac )\vert_{A^{7}} = \lbrace 1, d, a, c, d^{-1}b^{-1}cb, d^{-1}, b^{-1}, a^{-1}\rbrace$ , \\ 
$( d^{-1} \otimes a^{-1}d^{-1}a )\vert_{A^{7}} = \lbrace 1, b^{-1}, d^{-1}, a^{-1}, c^{-1}a^{-1}c\rbrace$ , \\ 
$( d^{-1} \otimes d^{-1}b^{-1}d )\vert_{A^{7}} = \lbrace 1, a^{-1}, c^{-1}, b^{-1}, d^{-1}, a^{-1}d^{-1}a\rbrace$ , \\ 
$( d^{-1} \otimes b^{-1}c^{-1}b )\vert_{A^{7}} = \lbrace 1, a^{-1}, c^{-1}, b^{-1}, d^{-1}, d^{-1}b^{-1}d\rbrace$ , \\ 
$( d^{-1} \otimes c^{-1}a^{-1}c )\vert_{A^{7}} = \lbrace 1, b^{-1}, d^{-1}, a^{-1}, c^{-1}, b^{-1}c^{-1}b\rbrace$ , \\ 
$( d^{-1} \otimes a^{-1}bd )\vert_{A^{7}} = \lbrace 1, d^{-1}, a^{-1}, a, c, b, a^{-1}b, b^{-1}, d, c^{-1}da, c^{-1}\rbrace$ , \\ 
$( d^{-1} \otimes d^{-1}cb )\vert_{A^{7}} = \lbrace 1, a, c, b, d^{-1}, a^{-1}c, a^{-1}bd, a^{-1}, b^{-1}\rbrace$ , \\ 
$( d^{-1} \otimes c^{-1}da )\vert_{A^{7}} = \lbrace 1, b, a^{-1}d, a^{-1}, d^{-1}, d, a, a^{2}, d^{-1}b^{-1}ac, b^{-1}, c^{-1}\rbrace$ , \\ 
$( d^{-1} \otimes d^{-1}b^{-1}a )\vert_{A^{7}} = \lbrace 1, b^{-1}, d^{-1}, a^{-1}, a^{-1}c, a^{-1}d^{-1}c, c^{-1}\rbrace$ , \\ 
$( d^{-1} \otimes b^{-1}c^{-1}d )\vert_{A^{7}} = \lbrace 1, a^{-1}, c^{-1}, b^{-1}, d^{-1}, a, d^{-1}a, d^{-1}b^{-1}a\rbrace$ , \\ 
$( d^{-1} \otimes c^{-1}a^{-1}b )\vert_{A^{7}} = \lbrace 1, b^{-1}, d^{-1}, a^{-1}, b, d, c^{-1}, b^{-1}c^{-1}d\rbrace$ , \\ 
$( d^{-1} \otimes a^{-1}d^{-1}c )\vert_{A^{7}} = \lbrace 1, d^{-1}, a^{-1}, a, c, b, a^{-1}b, b^{-1}, a^{-2}, c^{-1}a^{-1}b\rbrace$ , \\ 
$( d^{-1} \otimes a^{-1}cb )\vert_{A^{7}} = \lbrace 1, a, c, d^{-1}b, d^{-1}, a^{-1}, b, a^{-1}b, c^{-1}bd, d\rbrace$ , \\ 
$( d^{-1} \otimes d^{-1}ac )\vert_{A^{7}} = \lbrace 1, d^{-1}, a^{-1}, a, c, a^{-1}c, a^{-1}cb, b, b^{-1}\rbrace$ , \\ 
$( d^{-1} \otimes c^{-1}bd )\vert_{A^{7}} = \lbrace 1, b, a^{-1}d, a^{-1}, b^{-1}, d, d^{-1}b^{-1}da, a, c^{-1}\rbrace$ , \\ 
$( d^{-1} \otimes b^{-1}c^{-1}a )\vert_{A^{7}} = \lbrace 1, d^{-1}, a^{-1}, b^{-1}, a^{-1}d^{-1}, a, d^{-1}a, d^{-1}b^{-1}c, c^{-1}\rbrace$ , \\ 
$( d^{-1} \otimes c^{-1}a^{-1}d )\vert_{A^{7}} = \lbrace 1, d^{-1}, a^{-1}, c^{-1}, b^{-1}, d^{-1}b^{-1}, b, d, b^{-1}c^{-1}a\rbrace$ , \\ 
$( d^{-1} \otimes a^{-1}d^{-1}b )\vert_{A^{7}} = \lbrace 1, d^{-1}, a^{-1}, a, c, b, a^{-1}b, b^{-1}, c^{-1}, c^{-1}a^{-1}d, a^{-2}\rbrace$ , \\ 
$( d^{-1} \otimes d^{-1}b^{-1}c )\vert_{A^{7}} = \lbrace 1, d^{-1}, a^{-1}, a^{-1}c, a^{-1}d^{-1}b, b^{-1}\rbrace$ , \\ 
$( d^{-1} \otimes a^{-1}d^{-1}bd )\vert_{A^{7}} = \lbrace 1, d^{-1}, a^{-1}, a, c, b, a^{-1}b, b^{-1}, d, c^{-1}a^{-1}da, c^{-1}\rbrace$ , \\ 
$( d^{-1} \otimes d^{-1}b^{-1}cb )\vert_{A^{7}} = \lbrace 1, a, c, b, d^{-1}, a^{-1}c, a^{-1}d^{-1}bd, a^{-1}, b^{-1}\rbrace$ , \\ 
$( d^{-1} \otimes b^{-1}c^{-1}ac )\vert_{A^{7}} = \lbrace 1, b^{-1}, a, d^{-1}a, c, d^{-1}b^{-1}cb, d^{-1}, a^{-1}, a^{-1}d^{-1}, c^{-1}\rbrace$ , \\ 
$( d^{-1} \otimes c^{-1}a^{-1}da )\vert_{A^{7}} = \lbrace 1, b, d, a, b^{-1}c^{-1}ac, b^{-1}, a^{-1}, c^{-1}, d^{-1}, d^{-1}b^{-1}\rbrace$ , \\ 
$( d^{-1} \otimes d^{-1}b^{-1}da )\vert_{A^{7}} = \lbrace 1, b^{-1}, d^{-1}, a, a^{-1}d^{-1}ac, a^{-1}c, a^{-1}, c^{-1}\rbrace$ , \\ 
$( d^{-1} \otimes b^{-1}c^{-1}bd )\vert_{A^{7}} = \lbrace 1, a^{-1}, c^{-1}, b^{-1}, b, a^{-1}d, d^{-1}, d, d^{-1}b^{-1}da, a, d^{-1}a\rbrace$ , \\ 
$( d^{-1} \otimes c^{-1}a^{-1}cb )\vert_{A^{7}} = \lbrace 1, d^{-1}, a^{-1}, a, c, b, a^{-1}b, b^{-1}, c^{-1}, b^{-1}c^{-1}bd, d\rbrace$ , \\ 
$( d^{-1} \otimes a^{-1}d^{-1}ac )\vert_{A^{7}} = \lbrace 1, d^{-1}, a^{-1}, a, c, c^{-1}a^{-1}cb, b, b^{-1}, a^{-1}b\rbrace$ , \\ 
$( d^{-1} \otimes a^{-1}d^{-1}cb )\vert_{A^{7}} = \lbrace 1, a, c, d^{-1}b, d^{-1}, a^{-1}, b, a^{-1}b, c^{-1}a^{-1}bd, a^{-2}\rbrace$ , \\ 
$( d^{-1} \otimes d^{-1}b^{-1}ac )\vert_{A^{7}} = \lbrace 1, d^{-1}, a^{-1}, a, c, a^{-1}c, a^{-1}d^{-1}cb, b^{-1}\rbrace$ , \\ 
$( d^{-1} \otimes b^{-1}c^{-1}da )\vert_{A^{7}} = \lbrace 1, d^{-1}, a^{-1}, d, a, b^{-1}, d^{-1}a, a^{2}, d^{-1}b^{-1}ac, c^{-1}\rbrace$ , \\ 
$( d^{-1} \otimes c^{-1}a^{-1}bd )\vert_{A^{7}} = \lbrace 1, b^{-1}, d^{-1}, a^{-1}, b, d, b^{-1}c^{-1}da, c^{-1}\rbrace$ , \\ 
$( d^{-1} \otimes a^{2} )\vert_{A^{7}} = \lbrace 1, d, a, c, b^{-1}, d^{-1}, a^{-1}c\rbrace$ , \\ 
$( d^{-1} \otimes a^{-2} )\vert_{A^{7}} = \lbrace 1, d^{-1}, a^{-1}, c^{-1}, b^{-1}\rbrace$ , \\

\bibliographystyle{amsplain}
\bibliography{ref}

\end{document}